\documentclass[twoside,11pt]{amsart}

\usepackage[foot]{amsaddr}

\author{Matthew D. Kvalheim}
\address{Department of Electrical and Systems Engineering, University of Pennsylvania, Philadelphia, PA 19104}

\author{Shai Revzen}
\address{Department of Electrical Engineering and Computer Science, Ecology and Evolutionary Biology Department, Robotics Institute, University of Michigan, Ann Arbor, MI 48109}

\email{kvalheim@seas.upenn.edu, shrevzen@umich.edu}

\title[Existence and uniqueness of global Koopman eigenfunctions]{Existence and uniqueness of global Koopman eigenfunctions for stable fixed points and periodic orbits}

\usepackage{import} 

\usepackage[utf8]{inputenc}
\usepackage[T1]{fontenc}
\usepackage{lmodern}
\usepackage{amsmath}
\usepackage{amssymb}
\usepackage{amsthm}
\usepackage{stmaryrd} 
\usepackage{mathtools}
\usepackage{tikz-cd}
\usepackage{subfig}
\usepackage{lineno}

\newcommand{\concept}[1]{\textit{#1}}

\newcommand{\N}{\mathbb{N}}
\newcommand{\Z}{\mathbb{Z}}
\newcommand{\R}{\mathbb{R}}
\newcommand{\Grp}{\mathbb{T}}
\newcommand{\C}{\mathbb{C}}

\newcommand{\Ws}{W^s}

\newcommand{\Sym}{\textnormal{Sym}}

\newcommand{\slot}{\,\cdot\,} 

\newcommand{\T}{\mathsf{T}}
\newcommand{\D}{\mathsf{D}}

\newcommand{\id}{\textnormal{id}}

\newcommand{\interior}{\textnormal{int}}

\newcommand{\Lip}[1]{\textnormal{Lip}(#1)}

\newcommand{\GL}{\mathsf{GL}}
\newcommand{\CL}{C_{\textnormal{loc}}}
\newcommand{\CLKA}{C_{\textnormal{loc}}^{k,\alpha}}
\newcommand{\CKA}{C^{k,\alpha}}

\newcommand{\A}{\mathcal{A}}

\newcommand{\F}{\mathcal{F}}
\newcommand{\Ll}{\mathcal{L}}

\DeclarePairedDelimiter\norm{\lVert}{\rVert}

\newtheorem{Lem}{Lemma}
\newtheorem{Th}{Theorem}

\newtheorem{Prop}{Proposition}

\usepackage{thm-restate} 
\newcommand{\thistheoremname}{}
\newtheorem*{genericthm}{\thistheoremname}
{\renewcommand{\thistheoremname}{Theorem~\ref{#1}$'$}%
	\begin{genericthm}}
	{\end{genericthm}}

\theoremstyle{definition}
\newtheorem{Def}{Definition}
\newtheorem*{Def*}{Definition}
\newtheorem{Ex}{Example}
\newtheorem{Rem}{Remark}
\newtheorem*{Not}{Notation}

\usepackage{marvosym}
\usepackage[
hypertexnames,%
citecolor=blue,%
colorlinks=true,%
linkcolor=red%
]{hyperref}
\usepackage[all]{hypcap}

\begin{document}

	\begin{abstract}
		We consider $C^1$ dynamical systems having an attracting hyperbolic fixed point or periodic orbit and prove existence and uniqueness results for $C^k$ (actually $\CLKA$) linearizing semiconjugacies---of which Koopman eigenfunctions are a special case---defined on the entire basin of attraction.
        Our main results both generalize and sharpen Sternberg's $C^k$ linearization theorem for hyperbolic sinks, and in particular
        our corollaries include uniqueness statements for Sternberg linearizations and Floquet normal forms.  
        Using our main results we also prove new existence and uniqueness statements for $C^k$ Koopman eigenfunctions, including a complete classification of $C^\infty$ eigenfunctions assuming a $C^\infty$ dynamical system with semisimple and nonresonant linearization.
        We give an intrinsic definition of ``principal Koopman eigenfunctions'' which generalizes the definition of Mohr and Mezi\'{c} for linear systems, and which includes the notions of ``isostables'' and ``isostable coordinates'' appearing in work by Ermentrout, Mauroy, Mezi\'{c}, Moehlis, Wilson, and others. 
        Our main results yield existence and uniqueness theorems for the principal eigenfunctions and isostable coordinates and also show, e.g., that the (a priori non-unique) ``pullback algebra'' defined in \cite{mohr2016koopman} is unique under certain conditions.
        We also discuss the limit used to define the ``faster'' isostable coordinates in \cite{wilson2018greater,monga2019phase} in light of our main results. 
	\end{abstract}

	\maketitle

	\tableofcontents
	
	\section{Introduction} 		
	This paper fills a significant technical gap between the linearization results known from classical dynamical systems theory---e.g., the linearization theorems of Poincar\'{e}-Siegel \cite{poincare1879proprietes,siegel1942iteration,siegel1952uber}, Sternberg \cite{sternberg1957local}, Grobman-Hartman \cite{grobman1959homeomorphism,hartman1960lemma}, and Hartman \cite{hartman1960local}---and the growing interest in applied fields such as engineering and fluid dynamics in using linearizations based on Koopman theory.\footnote{Here ``linearization'' refers to a nonlinear change of coordinates in which a nonlinear dynamical system becomes \emph{exactly} linear, and is distinct from approximate linearization. A recent paper extending and discussing some of the state of the art is \cite{newhouse2017differentiable}.}
	Motivated largely by data-driven applications, this ``applied Koopmanism'' literature has experienced a surge of interest initiated by \cite{mezic1994geometrical,mezic2004comparison,mezic2005spectral} more than 70 years after Koopman's seminal work \cite{koopman1931hamiltonian}.\footnote{See, e.g., \cite{budivsic2012applied,mauroy2012use,mauroy2013isostables,lan2013linearization,mohr2014construction,giannakis2015spatiotemporal,mezic2015applications,williams2015data,mauroy2016global,mohr2016koopman,brunton2016koopman,surana2016koopman, surana2016linear,arbabi2017ergodic, arbabi2017study,kaiser2017data,mezic2019spectrum,proctor2018generalizing,korda2018convergence,korda2018data,korda2019optimal,das2019delay,bruder2019nonlinear,dietrich2019koopman,arbabi2019data}.}

	The practical application of computational Koopman eigenfunction representations of dynamical systems is grounded in (i) eigenfunction existence theorems based on such classical linearization theorems \cite{lan2013linearization,mauroy2013isostables} and (ii) uniqueness theorems applying only to analytic eigenfunctions for certain analytic dynamical systems \cite{mauroy2013isostables}.
    Existence and uniqueness results are desirable since, in analyzing the theoretical properties of any algorithm for computing some quantity, it is desirable to know whether the computation is well-posed \cite{hadamard1902problemes}, and in particular whether the quantity in question \emph{exists} and is \emph{uniquely determined}.

    The results in the present paper yield new precise conditions under which various quantities in the applied Koopmanism literature---including targets of numerical algorithms---exist and are unique, and are especially relevant to work on \concept{principal eigenfunctions} and \concept{isostables} for point attractors \cite{mohr2016koopman,mauroy2013isostables} and to work on \concept{isostable coordinates} for periodic orbit (limit cycle) attractors \cite{wilson2016isostable,shirasaka2017phase,wilson2018greater,monga2019phase}.
    Isostables and isostable coordinates are useful tools for nonlinear model reduction, and it has been proposed that they could prove useful in real-world applications such as treatment design for Parkinson's disease, migraines, cardiac arrhythmias \cite{wilson2016isostable-pde}, and jet lag \cite{wilson2014energy}.	
	
	\subsection{Nontechnical overview of results}\label{sec:flavor}
	This paper was motivated by the following three questions. (A $C^k$ function is one which has continuous mixed partial derivatives up to order $k$.)	
	\begin{description}
	\item[Eigenfunction uniqueness] When---and in what sense---are $C^k$ ($1\leq k \leq +\infty$) Koopman eigenfunctions unique?
	\item[Linearization uniqueness] When---and in what sense---are full $C^k$ ($1\leq k \leq +\infty$) linearizing coordinate changes unique?
	\item[Eigenfunction existence] Can the existence of specific $C^k$ ($2\leq k \leq +\infty$) Koopman eigenfunctions be guaranteed under assumptions which are weaker than those needed to invoke a classical result guaranteeing a full set of linearizing coordinates exists? 
	\end{description}
	We provide answers to each of these three questions for eigenfunctions and coordinate changes defined on the basin of an attracting hyperbolic fixed point or periodic orbit.
	To the best of our knowledge, our answers to the eigenfunction uniqueness and existence questions are new.
	During the review process of this paper we discovered that an answer to the full linearization uniqueness question for the case $k=+\infty$ (equivalent to our answer in that case) can also be obtained from results in the recent book \cite[Thm~6.8.21,~6.8.22]{fisher2019hyperbolic}. 

	Sternberg's work showed that, for $C^k$ ($1\leq k\leq +\infty$) dynamical systems such as those arising from a $C^k$ ordinary differential equation $$\frac{d}{dt}x(t) = f(x(t)) \qquad x(t)\in \R^n$$ having an attracting fixed point $x_0\in \R^n$, there is an intimate relationship between the eigenvalues of the derivative $\D_{x_0}f$ of $f$ at $x_0$ and the existence of $C^k$ linearizing coordinates defined near $x_0$ \cite{sternberg1957local}.
	Our results establish a similarly intimate relationship between certain eigenvalues and the answers to our three questions.
	In the case of the eigenfunction uniqueness and existence questions, we show that the corresponding relationship is less restrictive than in Sternberg's case. 
	In particular, the answer to the eigenfunction existence question is ``yes'' (especially in the case of the ``slowest'' principal eigenfunctions/isostable coordinates).
	The answer to the linearization uniqueness question is this: under Sternberg's hypotheses guaranteeing that $C^k$ linearizing coordinates exist in the vicinity (hence also on the entire basin \cite{lan2013linearization,kvalheim2018global}) of an attracting hyperbolic fixed point, these coordinates are uniquely determined by their derivatives at the fixed point.	
	
	To obtain the answers to our three questions, we prove a general result on the existence and uniqueness of $C^k$ linearizing \emph{semi}conjugacies (partial linearizations), of which Koopman eigenfunctions and linearizing conjugacies are special cases.
    To provide more refined answers to the same three questions, we further examine the $\CLKA$ smoothness classes refining the familiar $C^k$ smoothness classes, wherein $C^k = \CL^{k,0}$. 
    Our main results concern the existence and uniqueness of linearizing semiconjugacies defined on the basin of an attracting hyperbolic fixed point or periodic orbit.
    Notable consequences worked out in this paper, for dynamics in the basin of an attracting hyperbolic fixed point or periodic orbit, include: 
    \begin{itemize}
    \item a fairly tight relationship between properties of the system eigenvalues/Floquet multipliers and the uniqueness of $\CLKA$ \concept{principal} eigenfunctions for $C^1$ dynamics (Proposition \ref{prop:koopman-cka-fix}) or $\CLKA$ dynamics (Proposition \ref{prop:koopman-cka-per});
    \item sufficient conditions for the existence of such $\CLKA$ principal eigenfunctions when the dynamics are $\CLKA$ (Propositions \ref{prop:koopman-cka-fix} and \ref{prop:koopman-cka-per});
    \item sufficient conditions for when such $\CLKA$ eigenfunctions (or isostable coordinates) can be constructed by limiting procedures such as \concept{Laplace averages} (Propositions \ref{prop:koopman-cka-fix} and \ref{prop:koopman-cka-per} and Remark \ref{rem:laplace});
    \item a full classification of \emph{all} $C^\infty$ eigenfunctions in terms of principal eigenfunctions for $C^\infty$ dynamics satisfying a nondegeneracy assumption (Theorems \ref{th:classify-point} and \ref{th:classify-per}); and
    \item sufficient conditions under which a full set of linearizing $C^k$ coordinates exist and are uniquely determined by their first-order approximation (Propositions \ref{prop:sternberg} and \ref{prop:floq-norm-form}).   
    \end{itemize}
    
    To the best of our knowledge, our uniqueness results for principal Koopman eigenfunctions are the first uniqueness results known for non-analytic eigenfunctions.
    Similarly, our classification of \emph{all} $C^\infty$ Koopman eigenfunctions appears to be the first such classification theorem for non-analytic eigenfunctions.
    While certain existence results for principal Koopman eigenfunctions defined on the basin of an attracting hyperbolic equilibrium or periodic orbit have been known for some time \cite{lan2013linearization, mauroy2013isostables}, we believe our new existence results to be the strongest known for $C^k$ eigenfunctions with $2\leq k \leq +\infty$.
    This is because prior existence results (cf. \cite[Sec.~5, 7, 8]{mezic2019spectrum}) construct $C^k$ eigenfunctions by pulling back linear eigenfunctions through the linearizing conjugacies provided by the Sternberg or Poincar\'{e}-Siegel linearization theorems mentioned above, and the full hypotheses of one of these linearization theorems must be assumed in order to invoke it; in contrast, our existence result for \emph{specific} principal eigenfunctions requires much weaker assumptions.
    These assumptions are extremely mild in the case of the ``slowest'' principal eigenfunctions/isostable coordinates; the following subsection contains a precise statement (with more details in Remark \ref{rem:isostable}) as part of a more technical overview of our results.
    
	\subsection{Technical overview of results and organization of the paper}
	In this paper, we consider $C^1$ dynamical systems $\Phi\colon Q\times \Grp \to Q$ for which $Q$ is the basin of attraction of a stable hyperbolic fixed point or periodic orbit.
	Here $Q$ is a smooth manifold, either $\Grp = \Z$ or $\Grp = \R$, and $\Phi$ could be the restriction of a dynamical system defined on a larger space (e.g., $\R^n$) to some basin of attraction $Q$. 
	That $\Phi$ is a \concept{dynamical system} means that $\Phi^0 = \id_Q$ and $\Phi^{t+s}=\Phi^t\circ \Phi^s$ for all $t,s\in \Grp$, where $\Phi^t\coloneqq \Phi(\slot,t)$; in particular, it follows that $\Phi^t\colon Q\to Q$ is a $C^1$ diffeomorphism with inverse $\Phi^{-t}$ for any $t\in \Grp$. 
	When $\Grp = \R$, $\Phi$ is called a \concept{flow}; a common example is that of $t\mapsto \Phi^t(x_0)$ being the solution to the initial value problem $$\frac{d}{dt}x(t) = f(x(t)),\qquad x(0) = x_0$$ determined by a complete $C^1$ vector field $f$ on $Q$.
	Our main contributions are existence and uniqueness results regarding $\CLKA$ linearizing semiconjugacies $\psi\colon Q\to \C^m$ defined on the entire basin of attraction $Q$, where we do not assume that $m$ has any relationship to the dimension of $Q$; in particular, $\psi$ need not be a diffeomorphism or a homeomorphism. 
	By definition, such a semiconjugacy makes the diagram
	\begin{equation}\label{eq:cd-gen-efunc}
	\begin{tikzcd}
	&Q \arrow{r}{\Phi^t}\arrow{d}{\psi}&Q\arrow{d}{\psi}\\
	&\C^m\arrow{r}{e^{tA}}&\C^m
	\end{tikzcd}
	\end{equation}
    commute for some $A\in \C^{m\times m}$ and all $t\in \Grp$.
    By $\CLKA$ with $k\in \N_{\geq 1}$ and $0\leq \alpha \leq 1$, we mean that $\psi\in C^k(Q,\C^m)$ and that all $k$-th partial derivatives of $\psi$ are locally $\alpha$-H\"{o}lder continuous in local coordinates, and by definition $\CL^{\infty,\alpha}\coloneqq \CL^{+\infty,\alpha}\coloneqq C^\infty$.
    We note that $\CL^{k,0} = C^k$, so the reader uninterested in H\"{o}lder continuity can simply keep in mind the case $\CL^{k,0}=C^k$ and the fact that every $C^{k+1}$ function is also $\CLKA$ for every $0\leq \alpha \leq 1$.
    Our motivation for including local H\"{o}lder continuity of derivatives is that, once it is included, our main results become fairly close to optimal, at least in the interesting case $m=1$ of Koopman eigenfunctions $\psi$ (see Examples \ref{ex:thm-sharp} and \ref{ex:koop-converge}).  
    
    Linearizing semiconjugacies are also known as linearizing \concept{factors} or \concept{factor maps} in the literature and can be viewed as a further generalization of the \concept{generalized Koopman eigenfunctions} of \cite{mezic2019spectrum,korda2019optimal}.
    We note that such semiconjugacies are distinct from those in the diagram
	\begin{equation}\label{eq:cd-param-method}
	\begin{tikzcd}
	&Q \arrow{r}{\Phi^t}&Q\\
	&\C^m\arrow{r}{e^{tA}}\arrow{u}{K}&\C^m\arrow{u}{K}
	\end{tikzcd}
	\end{equation}    
    obtained from \eqref{eq:cd-gen-efunc} by flipping the vertical arrows (although the diagrams are equivalent if, e.g., $\psi$ and $K$ are diffeomorphisms).
    In \eqref{eq:cd-param-method} $K$ is a factor of $e^{tA}$, whereas $\psi$ is a factor of $\Phi^t$ in \eqref{eq:cd-gen-efunc}.
    Existence results for semiconjugacies of the type in \eqref{eq:cd-param-method} were obtained by \cite{cabre2003parameterization1,cabre2003parameterization2,cabre2005parameterization3} in the context of proving invariant manifold results using the parameterization method.
    
    Our main result for the case of an attracting hyperbolic fixed point both generalizes and sharpens Sternberg's linearization theorem \cite[Thms~2,3,4]{sternberg1957local} which provides conditions ensuring the existence of a linearizing local $C^k$ diffeomorphism defined on a neighborhood of the fixed point; the results of \cite{lan2013linearization,kvalheim2018global} show that this local diffeomorphism can be extended to a $C^k$ diffeomorphism $\psi \colon Q\to \R^n\subset \C^n$ defined on the entire basin of attraction $Q$ making \eqref{eq:cd-gen-efunc} commute. 
    Under Sternberg's conditions, a corollary of our main result is that this global linearizing diffeomorphism is in fact uniquely determined by its derivative at the fixed point (cf. \cite[Thm~6.8.21,~6.8.22]{fisher2019hyperbolic} for the case $k=+\infty$).
    Additionally, we sharpen  Sternberg's result from $C^k$ to $\CLKA$ linearizations.
    For the case of an attracting hyperbolic periodic orbit of a flow, our main result also yields a similar existence and uniqueness corollary for the Floquet normal form, a nonlinear change of coordinates in which the dynamics become the product of a linear system with constant-rate rotation on a circle \cite{abraham1967transversal,smoothInvariant,abraham2001manifolds}.
    We remark that the Floquet normal form is a \emph{nonlinear} generalization of the comparatively well-known classical Floquet theory of \emph{linear} time-periodic systems \cite[Sec.~III.7]{hale1980ode}
    
    Using our two main results, we make the following contributions to the theory of Koopman eigenfunctions.
    We give an intrinsic definition of principal eigenfunctions for nonlinear dynamical systems which generalizes the definition for linear systems in \cite{mohr2016koopman}.
    We provide existence and uniqueness results for $\CLKA$ principal eigenfunctions, and we also show that the (a priori non-unique) ``pullback algebra'' defined in \cite{mohr2016koopman} is unique under certain conditions.
    For the case of periodic orbit attractors, principal eigenfunctions essentially coincide with the notion of isostable coordinates defined in  \cite{wilson2018greater, monga2019phase}, except that the definition in these references involves a limit which might not exist except for the ``slowest'' isostable coordinate.
    Our techniques shed light on this issue, and our results imply that this limit does in fact always exist for the ``slowest'' isostable coordinate if the dynamical system is at least smoother than $C^{1,\alpha}_{\textnormal{loc}}$ with $\alpha > 0$.
    In fact, our results imply---assuming that there is a unique and algebraically simple ``slowest'' Floquet multiplier which is real---that a corresponding ``slowest'' $\CL^{1,\alpha}$ isostable coordinate with $\alpha > 0$ always exists and is unique modulo scalar multiplication for a $\CL^{1,\alpha}$ dynamical system (e.g., a $C^2$ dynamical system), without the need for any nonresonance or spectral spread assumptions. 
    Similarly, if instead there is a unique and algebraically simple ``slowest'' pair of Floquet multipliers which are complex conjugates, then a corresponding ``slowest'' complex conjugate pair of $\CL^{1,\alpha}$ isostable coordinates always exists and is unique modulo scalar multiplication for a $\CL^{1,\alpha}$ dynamical system with $\alpha > 0$.
    As a final application of our main results, we give a complete classification of $C^\infty$ eigenfunctions for a $C^\infty$ dynamical system with semisimple (diagonalizable over $\C$) and nonresonant linearization, generalizing known results for analytic dynamics and analytic eigenfunctions \cite{mauroy2013isostables,mezic2019spectrum}. 
    
    The remainder of the paper is organized as follows.    
    We explain notation and terminology below to be used in the sequel.
    After some definitions, in \S \ref{sec:main-results} we state Theorems \ref{th:main-thm} and \ref{th:main-thm-per}, our two main results, without proof.
    We also state a proposition on the uniqueness of linearizing factors which does not assume any nonresonance conditions. 
    As applications we derive in \S \ref{sec:applications} several results which are essentially corollaries of this proposition and the two main theorems.
    \S \ref{sec:app:stern-floq} contains existence and uniqueness theorems for global Sternberg linearizations and Floquet normal forms.
    In \S \ref{sec:app-p-eigs} we define principal Koopman eigenfunctions and isostable coordinates for nonlinear dynamical systems and discuss how Theorems \ref{th:main-thm} and \ref{th:main-thm-per} yield corresponding existence and uniqueness results.
    We then discuss the relationship between various notions defined in \cite{mohr2016koopman} and our definitions, and we also discuss the convergence of the isostable coordinate limits in \cite{wilson2018greater,monga2019phase}.
    \S \ref{sec:app-classify} contains our theorem which completely classifies the $C^\infty$ eigenfunctions of $C^\infty$ dynamical systems on the basin of an attracting hyperbolic fixed point or periodic orbit.
    Finally, \S \ref{sec:proofs-main-results} contains the proofs of Theorems \ref{th:main-thm} and \ref{th:main-thm-per}.
    
    \subsection{Notation and terminology}\label{sec:notation-and-terminology}
        In this paper we employ the following mostly standard notation and terminology.    
    
    \subsubsection{Sets of numbers} We denote the real numbers by $\R$, complex numbers by $\C$, integers by $\Z$, and nonnegative integers by $\N$.
        Given $c\in \R$ and $S\subset \R$, we define $S_{\geq c}\coloneqq \{s\in S\colon s\geq c\}$ and $S_{> c}\coloneqq \{s\in S\colon s > c\}$ so that, e.g., $\Z_{\geq 0} = \N_{\geq 0}=\N$.
        
     \subsubsection{Linear algebra} 
     Given $m\in \N_{\geq 1}$, we denote by $\GL(m,\C)\subset \C^{m\times m}$ the invertible $m\times m$ matrices with entries in $\C$ and by  $\GL(m,\R)\subset \GL(m,\C)$ those with entries in $\R$.
     Given $A\in \C^{m\times m}$, we denote by $\textnormal{spec}(A)\subset \C$ the set of eigenvalues of $A$; given $A\in \R^{m\times m}$, we denote by $\textnormal{spec}(A)\subset \C$ the set of eigenvalues of $A\in \R^{m\times m}\subset \C^{m\times m}$ when viewed as a complex matrix.
     If $A\colon V\to V$ is a linear self-map with $V$ a complex vector space, $\textnormal{spec}(A)\subset \C$ also denotes the eigenvalues of $A$.
     If $V$ is a real vector space, then its complexification $V_\C$ is the complex vector space given by all formal linear combinations of vectors in $V$ with complex coefficients (cf. \cite[p.~64]{hirsch1974differential}); if $A\colon V\to W$ is an $\R$-linear map between real vector spaces, then the complexification $A_\C\colon V_\C\to W_\C$ of $A$ is the unique $\C$-linear extension of $A$ (cf. \cite[p.~65]{hirsch1974differential}), and if $W=V$ we define $\textnormal{spec}(A)\coloneqq\textnormal{spec}(A_\C)\subset \C$.
        If $E_1, E_2\subset V$ are linear subspaces of a real or complex vector space $V$, we say that $E_1$ and $E_2$ are complementary if $V = \{e_1+e_2\colon e_1\in E_1, e_2\in E_2\}$ and if $E_1\cap E_2 = \{0\}$.                  
    
    	\subsubsection{Derivatives} 
    	Given a differentiable map $F\colon M\to N$ between smooth manifolds, we use the notation $\D_x F$ for the derivative of $F$ at the point $x\in M$.
    	(Recall that the derivative $\D_x F\colon \T_x M \to \T_{F(x)}N$ is a linear map between tangent spaces \cite{lee2013smooth}, which can be identified with the Jacobian of $F$ evaluated at $x$ in local coordinates.)
       	In particular, given a dynamical system $\Phi\colon Q\times \Grp \to Q$ and fixed $t \in \Grp$, we write $\D_{x}\Phi^t\colon \T_{x}Q\to \T_{\Phi^t(x)}Q$ for the derivative of the time-$t$ map $\Phi^t\colon Q\to Q$ at the point $x\in Q$.
       	A map $F\colon M\to N$ satisfies $F\in C^k(M,N)$ with $k\in \N_{\geq 0}\cup \{+\infty\}$, or briefly $F\in C^k$, if every $x\in M$ is contained in a coordinate chart in which all mixed partial derivatives of $F$ of order less than $k+1$ exist and are continuous.       	
       	For convenience, we define the $0$-th derivative $\D^0_x F\coloneqq F(x)$ to coincide with $F$ for all $x\in M$.   
       	
       	Several of our results include conditions such as ``$\D_{x_0}^i F = 0$ for all integers $0\leq i < r$.''       	
       	This is to be interpreted to mean that, in local coordinates, all mixed partial derivatives of $F$ of order less than $r$ vanish at $x_0$. 
       	This can be made more formal in the following way.
       	Inductively, if $i\geq 2$ and $\D_{x_0}^jF\colon (\T_{x_0}M)^{\otimes j}\to \T_{F(x_0)}N$ is well-defined and zero for all $1\leq j\leq i-1$, then the $i$-th derivative $\D^i_{x_0}F\colon (\T_{x_0}M)^{\otimes i}\to \T_{F(x_0)}N$ is a well-defined linear map from the $i$-th tensor power $(\T_{x_0}M)^{\otimes i}$ to $\T_{F(x_0)}N$
       	represented in local coordinates by the $(1+i)$-dimensional array of $i$-th partial derivatives of $F$ evaluated at $x_0$.\footnote{To define $\D_{x_0}^i F$, use local coordinates. The inductive assumption that the first $i-1$ derivatives are well-defined and zero at $x_0$ ensures that the result is independent of the choice of local coordinates, hence well-defined.}

	\section{Main results}\label{sec:main-results}
	Before stating our main results, we give two definitions which are essentially asymmetric versions of some appearing in \cite{sternberg1957local,sell1985smooth}.
	When discussing eigenvalues and eigenvectors of a matrix or linear self-map (endomorphism) in the remainder of the paper, we are always discussing eigenvalues and eigenvectors of its complexification, although we do not always make this explicit.
	\begin{Def}[$(X,Y)$ $k$-nonresonance]\label{def:nonres}Let $X \in \C^{d\times d}$ and $Y \in \C^{n\times n}$ be matrices with eigenvalues $\mu_1,\ldots,\mu_d$ and $\lambda_1,\ldots,\lambda_n$, respectively, repeated with multiplicities.
		For any $k \in \N_{\geq 1} \cup \{+\infty\}$, we say that $(X,Y)$ is \concept{$k$-nonresonant} if, for any $i\in \{1,\ldots, d\}$ and any $m=(m_1,\ldots, m_n) \in \N^n_{\geq 0}$ satisfying $2\leq m_1 + \cdots + m_n < k+1$, \begin{equation}\label{eq:nonres-lem}
		\mu_i \neq \lambda_1^{m_1}\cdots \lambda_n^{m_n}.
		\end{equation}	
		(Note this condition vacuously holds if $k = 1$; i.e., any two matrices are $1$-nonresonant.)  
		We also extend the definition of $k$-nonresonance to general linear self-maps $X, Y$ of finite-dimensional complex vector spaces by identifying $X, Y$ with their matrix representations with respect to any choice of bases.
		We say that linear self-maps $(X,Y)$ of finite-dimensional real vector spaces are $k$-nonresonant if the complexifications $(X_\C,Y_\C)$ are $k$-nonresonant.
	\end{Def}
	For the definition below, recall that the spectral radius $\rho(X)$ of a matrix is defined to be the largest modulus (absolute value) of the eigenvalues of (the complexification of) $X$.
	\begin{Def}[$(X,Y)$ spectral spread] 
		\label{def:nonres-spread}
		Let $X\in \GL(m,\C)$ and $Y\in \GL(n,\C)$ be invertible matrices with the spectral radius $\rho(Y)$ satisfying $\rho(Y) < 1$. 
		We define the spectral spread $\nu(X,Y)$ to be
		\begin{equation}\label{eq:spread-def}
		\begin{split}
		\nu(X,Y)&\coloneqq \max_{\substack{\mu \in \textnormal{spec}(X)\\ \lambda \in \textnormal{spec}(Y)}}\frac{\ln(|\mu|)}{\ln(|\lambda|)}\\
		&= \min\left\{r\in \R\colon \left(\min_{\mu\in \textnormal{spec}(X)}|\mu|\right)\geq \left(\max_{\lambda\in \textnormal{spec}(Y)}|\lambda|^r\right)\right\}\\
		&= \min\left\{r\in \R\colon  \rho(X^{-1})\left(\rho(Y)\right)^r\leq 1\right\}.
		\end{split}
		\end{equation}
		We also extend the definition of $\nu(X,Y)$ to general linear automorphisms $X, Y$ of finite-dimensional complex vector spaces by identifying $X, Y$ with their matrix representations with respect to any choice of bases.	
		We extend the definition to linear automorphisms $(X,Y)$ of real vector spaces by defining $\nu(X,Y)\coloneqq \nu(X_\C,Y_\C)$ to be the spectral spread of the complexifications.	
	\end{Def}		
	The second line in \eqref{eq:spread-def} follows since $\ln(|\mu|)/\ln(|\lambda|)\leq r$ if and only if $\ln(|\mu|)\geq r \ln(|\lambda|)$, with the inequality flipping because $\ln(|\lambda|)<0$ since $|\lambda|\leq \rho(Y)<1$, and this in turn holds if and only if $|\mu|\geq |\lambda|^r$.
	The third line follows from the definition of the spectral radius $\rho(Y)$.
	Figure \ref{fig:spread} illustrates the condition $\nu(e^{A},\D_{x_0} \Phi^1) < k + \alpha$ in Theorem \ref{th:main-thm} below.		
	Finally, we now recall the definition of $\CL^{k,\alpha}$ functions.
	\begin{Def}[$\CL^{k,\alpha}$ functions]\label{def:ckal-funcs}
		Let $M,N$ be smooth manifolds of dimensions $m$ and $n$, let $\psi\in C^k(M,N)$ be a $C^k$ map $\psi\colon M\to N$ with $k\in \N_{\geq 0}$, and let $0\leq \alpha \leq 1$.
		We will say that $\psi \in \CL^{k,\alpha}(M,N)$ if for every $x\in M$ there exist charts $(U_1,\varphi_1)$ and $(U_2,\varphi_2)$ containing $x$ and $\psi(x)$ such that all $k$-th partial derivatives of $\varphi_2 \circ \psi \circ \varphi_1^{-1}$ 
		are H\"{o}lder continuous with exponent $\alpha$.
		If $k=+\infty$, we use the convention $\CL^{+\infty,\alpha}(M,N)\coloneqq C^\infty(M,N)$ for any $0\leq \alpha \leq 1$.
		If the domain and codomain $M$ and $N$ are clear from context, we will sometimes write $C^k$ and $\CL^{k,\alpha}$ instead of $C^k(M,N)$ and $\CL^{k,\alpha}(M,N)$ and write, e.g., $\psi \in C^k$ or $\psi\in \CLKA$.
		We note that $\CL^{k,\beta}\subset \CL^{k,\alpha}$ for any $k\in \N\cup \{+\infty\}$ and $0\leq \alpha \leq \beta \leq 1$, and that $\psi\in \CL^{k,0}$ if and only if $\psi \in C^k$.
	\end{Def}
	
	\begin{Rem}
		Using the chain rule and the fact that compositions and products of locally $\alpha$-H\"{o}lder continuous functions are again locally $\alpha$-H\"{o}lder, it follows that the property of being $\CLKA$ on a manifold does not depend on the choice of charts in Definition \ref{def:ckal-funcs}.
	\end{Rem}	
	
	We now state our main results, Theorems \ref{th:main-thm} and \ref{th:main-thm-per}, as well as Proposition \ref{prop:uniqueness-without-nonresonance}.
	An example and several clarifying remarks follow the statement of Theorem \ref{th:main-thm}; in particular, see Remark \ref{rem:intuition} for intuitive remarks and Example \ref{ex:intuition} for concreteness.
	Simple analytic (counter-)examples demonstrating Theorems \ref{th:main-thm} and \ref{th:main-thm-per} in the case $m=1$ of Koopman eigenfunctions---demonstrating in particular the sharpness of the uniqueness statement---are Examples \ref{ex:thm-sharp} and \ref{ex:koop-converge} of \S \ref{sec:app-p-eigs}.

	We emphasize that Theorems \ref{th:main-thm} and \ref{th:main-thm-per} do not assume that the linear map $B$ is invertible, and do not claim anything about the semiconjugacy $\psi\colon Q\to \C^m$ being a diffeomorphism (but see Propositions ~\ref{prop:sternberg} and \ref{prop:floq-norm-form}); moreover, nothing is assumed about the relationship of $m$ to $\dim(Q)$.	
	\begin{figure}
		\centering
		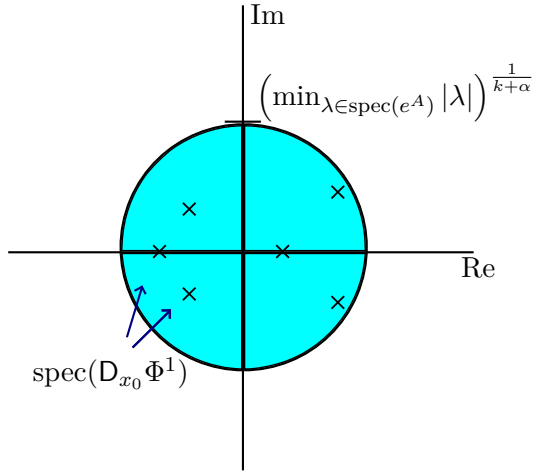
		\caption{An illustration of the condition $\nu(e^{A},\D_{x_0} \Phi^1) < k + \alpha$ of Theorem \ref{th:main-thm}. 
		This condition is equivalent to every eigenvalue of $\D_{x_0} \Phi^1$ (represented by an ``$\times$'' above) belonging to the open disk with radius given by raising the smallest modulus of the eigenvalues of $e^A$ to the power $\frac{1}{k+\alpha}$. }\label{fig:spread}
	\end{figure}	
	
	\begin{restatable}[Existence and uniqueness of $\CLKA$ global linearizing factors for a point attractor]{Th}{ThmMain}\label{th:main-thm}
	Let $\Phi\colon Q \times \Grp \to Q$ be a $C^{1}$ dynamical system with $Q$ the basin of an  attracting hyperbolic fixed point $x_0\in Q$,  where $Q$ is a smooth manifold with $\dim(Q)\geq 1$ and either $\Grp = \Z$ or $\Grp = \R$. 
	Let $m\in \N_{\geq 1}$ and $e^{A}\in \GL(m,
	\C)$ have spectral radius $\rho(e^{A})<1$, and let the linear map $B\colon \T_{x_0}Q\to \C^m$ satisfy
	\begin{equation}\label{eq:main-th-1}
	\forall t \in \Grp\colon B \D_{x_0}\Phi^t = e^{tA} B.
	\end{equation}
	Fix $k\in \N_{\geq 1}\cup \{+\infty\}$ and $0 \leq \alpha \leq 1$, assume that $(e^{A},\D_{x_0} \Phi^1)$ is $k$-nonresonant, and assume that $\nu(e^{A},\D_{x_0} \Phi^1) < k + \alpha$. 	
	
	\emph{\textbf{Uniqueness.}} 
	Any $\psi \in \CLKA(Q,\C^m)$ satisfying 
	\begin{equation*}
	\psi \circ \Phi^1 = e^{A} \psi, \qquad \D_{x_0} \psi = B
	\end{equation*}
	is unique, and if $B\colon \T_{x_0}Q\to \R^m\subset \C^m$ and $e^A\in \GL(m,\R)\subset \GL(m,\C)$ are real, then $\psi\colon Q\to \R^m \subset \C^m$ is real.
	
	\emph{\textbf{Existence}.} If furthermore $\Phi\in \CLKA$, then such a unique $\psi \in \CLKA(Q,\C^m)$ exists and additionally satisfies
	\begin{equation}\label{eq:main-th-3}
	\forall t \in \Grp\colon \psi \circ \Phi^t = e^{tA} \psi.
	\end{equation}	
	In fact, if $P\in \CLKA(Q,\C^m)$ is any ``approximate linearizing factor'' satisfying $\D_{x_0}P = B$ and 
	\begin{equation}\label{eq:main-th-4}
	P \circ \Phi^1 = e^A P + R
	\end{equation} 
	with $\D_{x_0}^i R = 0$ for all integers $0\leq i < k+\alpha$, then 
	\begin{equation}\label{eq:main-th-5}
	\psi = \lim_{t\to \infty} e^{-tA} P \circ \Phi^t
	\end{equation}
	in the topology of $C^{k,\alpha}$-uniform convergence on compact subsets of $Q$ if $k < +\infty$, and in the topology of $C^{k'}$-uniform convergence on compact subsets of $Q$ for any $k'\in \N_{\geq 1}$ if $k=+\infty$.
	\end{restatable}
	
	\begin{Rem}\label{rem:intuition}
	The spectral spread condition $\nu(e^A,\D_{x_0}\Phi^1)<k+\alpha$ means that, if $\lambda$ is any eigenvalue of $\D_{x_0}\Phi^1$, then $|\lambda|^{k+\alpha}$ is smaller than $|\mu|$ for every eigenvalue $\mu$ of $e^A$.
	The $k$-nonresonance condition means that no eigenvalue $\mu$ of $e^A$ can be written as a product (with repetitions allowed) of $\ell\in \{2,\ldots, k\}$ eigenvalues of $\D_{x_0}\Phi^1$.
	The uniqueness statement of Theorem \ref{th:main-thm} then says that, under these two conditions, any linearizing semiconjugacy $\psi\in \CLKA(Q,\C^m)$ is uniquely determined by its derivative at the fixed point $x_0$.
	Under the additional assumption that $\Phi\in \CLKA$ rather than merely $\Phi\in C^1$, the existence statement of Theorem \ref{th:main-thm} gives sufficient conditions ensuring that, given a \emph{linear} linearizing semiconjugacy $B\colon \T_{x_0}Q\to \C^m$ for the \emph{linear} dynamical system $(v,t)\mapsto \D_{x_0}\Phi^t \cdot v$, there exists a unique \emph{nonlinear} linearizing semiconjugacy $\psi\in \CLKA(Q,\C^m)$ for the \emph{nonlinear} dynamical system $\Phi$ satisfying $\D_{x_0}\psi = B$. 
	Thus, the existence statement can be thought of as supplying sufficient conditions under which a linearizing semiconjugacy can be constructed from an ``infinitesimal'' one.
	\end{Rem}
	
	\begin{Rem}[Weaker nonresonance assumption in the case $\Grp = \R$]\label{rem:weaker-nonresonance-real-time}
	Assume $\Grp = \R$ in the setting of Theorem \ref{th:main-thm}.
	Differentiating the identity $\Phi^{t+s}=\Phi^t\circ \Phi^s$ at $x_0$ yields
	\begin{equation}\label{eq:d-phi-group-property}
	\forall t,s\in \R\colon \D_{x_0}\Phi^{t+s}= \D_{x_0}\Phi^t\circ \D_{x_0}\Phi^s.
	\end{equation}
	If $\Phi\in C^1$, then the map $t\mapsto \D_{x_0}\Phi^t$ is, a priori, merely continuous.
	However, continuity together with \eqref{eq:d-phi-group-property} actually implies the existence of a linear map $J\colon \T_{x_0}Q\to \T_{x_0}Q$ such that
	\begin{equation}
	\forall t\in \R\colon \D_{x_0}\Phi^t = e^{tJ}.
	\end{equation}
	See \cite[Thm~2.9]{engel2000one} and \cite[p.~9, para.~1]{engel2000one}.
	Let $\lambda_1,\ldots, \lambda_n$ be the eigenvalues of (the complexification of) $J$, repeated with multiplicities.
	Taking the natural logarithm of \eqref{eq:nonres-lem}, we see that $k$-nonresonance of $(e^A, \D_{x_0}\Phi^1)$ means that, for any (possibly complex) eigenvalue $\mu$ of $A$, any $n$-tuple $(m_1,\ldots, m_n)\in \N_{\geq 0}^n$ satisfying $2\leq m_1 + \cdots + m_n < k + 1$, and any $\ell\in \Z$,
	\begin{equation}\label{eq:nonresonance-log-rem}
	\mu \neq m_1 \lambda_1 + \cdots + m_n \lambda_n + i 2\pi \ell,
	\end{equation}
	where  $i=\sqrt{-1}$.
    If $c>0$ and we define the rescaled linear map $A_c\coloneqq cA$ and the time-rescaled flow $\Phi_c\coloneqq \Phi^{ct}$, then we see that $B \D_{x_0}\Phi_c^t= B\D_{x_0}\Phi^{ct} = e^{ctA}B= e^{tA_c}B$ and $\psi \circ \Phi_c^t = \psi \circ \Phi^{ct} = e^{ctA}\psi = e^{tA_c}\psi$ for all $t\in \R$.
    Thus, $B$ and $\psi$ satisfy \eqref{eq:main-th-1} and \eqref{eq:main-th-3} with $\Grp=\R$ if and only if $B$ and $\psi$ satisfy \eqref{eq:main-th-1} and \eqref{eq:main-th-3} with $A$ and $\Phi$ replaced by $A_c$ and $\Phi_c$ for every $c> 0$. 
    Now, the $k$-nonresonance condition for $(e^{A_c},\D_{x_0}\Phi_c^1)$ is obtained from \eqref{eq:nonresonance-log-rem} by multiplying the eigenvalues $\mu$ and $\lambda_1,\ldots,\lambda_n$ by $c$; dividing by $c$ then yields
   	\begin{equation}\label{eq:nonresonance-log-rem-rescale}
   	\mu \neq m_1 \lambda_1 + \cdots + m_n \lambda_n + i \frac{2\pi}{c} \ell.
   	\end{equation}
	Because there are only finitely many eigenvalues of $A$ and $J$, \eqref{eq:nonresonance-log-rem-rescale} can be violated for all $c>0$ if and only if it is violated with $\ell = 0$.
	It follows that, if $\Grp = \R$, the $k$-nonresonance assumption in Theorem \ref{th:main-thm} (as well as in Theorem \ref{th:classify-point} and Propositions \ref{prop:sternberg} and \ref{prop:koopman-cka-fix}) can be replaced with the less restrictive condition 
	\begin{equation}\label{eq:nonresonance-log-rem-no-2-pi}
	\mu \neq m_1 \lambda_1 + \cdots + m_n \lambda_n
	\end{equation}
	for all $(m_1,\ldots, m_n)\in \N_{\geq 0}^n$ satisfying $2\leq m_1+\cdots m_n < k + 1$ and all (possibly complex) eigenvalues $\mu$ of $A$, where $\lambda_1,\ldots, \lambda_n$ are the eigenvalues of (the complexification of) $J$, repeated with multiplicity. 	
	\end{Rem}
	
	\begin{Ex}\label{ex:intuition}
	Consider the setting of Theorem \ref{th:main-thm} in the special case that $Q\subset \R^n$, $\Grp = \R$, and with $\Phi$ the flow of the ordinary differential equation
	$$\frac{dx}{dt}=f(x)$$
	with $f\in C^1$ a complete vector field, so that $f(x_0)=0$.
	Let $B\in \C^{m\times n}$ be any matrix such that $B\D_{x_0}f = AB$ for some $A\in \C^{m\times m}$. (For example, in the case $m = 1$, $B$ is a left eigenvector of $\D_{x_0}f$ if $B\neq 0$.)
	It follows that $Bh(t\D_{x_0}f)=h(tA)B$ for any analytic function $h$ and $t\in \R$, so in particular $B e^{t\D_{x_0}f}= e^{tA}B$.
	Since $\D_{x_0}\Phi^t = e^{t\D_{x_0}f}$ for all $t\in \R$ (because $f(x_0)=0$), it follows that $B \D_{x_0}\Phi^t = e^{tA}B$ for all $t\in \R$.
	Thus, Theorem \ref{th:main-thm} can be applied in this situation as long as the spectral spread and nonresonance conditions are satisfied.
	Since $\D_{x_0}\Phi^1 = e^{\D_{x_0}f}$ in the present setting, satisfaction of the spectral spread condition $\nu(e^A,\D_{x_0}\Phi^1)<k+\alpha$ according to Definition \ref{def:nonres-spread} (after taking logarithms) means that, if $\lambda$ is any (possibly complex) eigenvalue of $\D_{x_0}f$, and if $\mu$ is any (possibly complex) eigenvalue of $A$, then the real part $\textnormal{Re}(\lambda)<0$ satisfies $$(k+\alpha)\textnormal{Re}(\lambda)< \textnormal{Re}(\mu).$$
	Satisfaction of the $k$-nonresonance condition means that, for any (possibly complex) eigenvalue $\mu$ of $A$, any $n$-tuple $(m_1,\ldots, m_n)\in \N_{\geq 0}^n$ satisfying $2\leq m_1 + \cdots + m_n < k+1$, and any $\ell\in \Z$,
	\begin{equation}\label{eq:nonresonance-log}
	\mu \neq m_1 \lambda_1 + \cdots + m_n \lambda_n + i 2\pi \ell,
	\end{equation}
	where $\lambda_1,\ldots, \lambda_n$ are the (possibly complex) eigenvalues of $\D_{x_0}f$, repeated with multiplicity, and $i=\sqrt{-1}$.
	By Remark \ref{rem:weaker-nonresonance-real-time}, the $k$-nonresonance condition of Theorem \ref{th:main-thm} can actually be replaced with the less restrictive condition
	\begin{equation}\label{eq:nonresonance-log-no-2-pi}
	\mu \neq m_1 \lambda_1 + \cdots + m_n \lambda_n
	\end{equation}
	in the setting of the present example.
	To make easier the application of the existence portion of Theorem \ref{th:main-thm}, we note that if the vector field $f\in \CLKA$, then also the flow $\Phi\in \CLKA$ \cite[Thm~A.6]{eldering2013normally}.	
	\end{Ex}

   	\begin{Rem}
   	Definitions \ref{def:nonres} and \ref{def:nonres-spread} are not independent.
   	In particular, if $(X,Y)$ is $(\ell-1)$-nonresonant and $\nu(X,Y) < \ell$ for $\ell\in \N_{\geq 2}$, then it follows that $(X,Y)$ is $\infty$-nonresonant.
   	Hence an equivalent statement of Theorem \ref{th:main-thm} could be obtained by replacing $k$-nonresonance with $\infty$-resonance everywhere (alternatively, for the existence statement only $(k-1)$-nonresonance need be assumed in the case $\alpha = 0$).
   	We prefer to use the stronger-sounding statement of the theorem above since it makes it clear that the set of matrix pairs $(e^A, \D_{x_0}\Phi^1)$ satisfying its hypotheses are \emph{open} in the space of all matrix pairs.
   	Openness for $k < +\infty$ is immediate, and openness for $k = +\infty$ follows the fact that $\nu(e^A,\D_{x_0}\Phi^1)$ is always finite.  
   \end{Rem}

    \begin{Rem}\label{rem:no-nonres-needed}
     The statement in Theorem \ref{th:main-thm} regarding the limit in \eqref{eq:main-th-5} actually holds without any nonresonance assumptions if an approximate linearizing factor $P\in \CLKA(Q,\C^m)$ satisfying \eqref{eq:main-th-4} exists; see  Lemma  \ref{lem-make-approx-exact} in \S \ref{sec:main-proof-exist}. 
    \end{Rem}

    \begin{Rem}[the $C^\infty$ case]\label{rem:point-cinf-case}
    In the case that $k = +\infty$, the hypothesis $\nu(e^A,\D_{x_0}\Phi^1) < k + \alpha$ becomes $\nu(e^A,\D_{x_0}\Phi^1) < +\infty$ which is automatically satisfied since $\nu(e^A,\D_{x_0}\Phi^1)$ is always finite. 	
    Hence for the case $k = +\infty$, no assumption is needed on the spectral spread in Theorem \ref{th:main-thm}; we need only assume that $(e^A, \D_{x_0}\Phi^1$) is $\infty$-nonresonant.
    Similar remarks hold for all of the following results in this paper which include a condition of the form $\nu(\slot,\slot) < k + \alpha$.
    \end{Rem} 
    \begin{Rem}[sketch of the proof of the existence portion of Theorem \ref{th:main-thm}]\label{rem:main-thm-1-outline}
    Here we sketch the proof of the existence statement of Theorem \ref{th:main-thm}, which is somewhat more involved than the uniqueness proof.
    (The existence proof also yields uniqueness, but under the additional assumption $\Phi\in \CLKA$ not needed for the uniqueness statement in Theorem \ref{th:main-thm}.)
   	Since the basin of attraction $Q$ of $x_0$ is always diffeomorphic to $\R^n$  \cite[Lem~2.1]{wilson1967structure}, we may assume that $Q = \R^n$ and $x_0 = 0$.
   	For now we consider the case $k < +\infty$.
   	First, the $k$-nonresonance assumption implies that we can uniquely solve \eqref{eq:main-th-4} order by order (in the sense of Taylor polynomials) for $P$ up to order $k$.
   	Once we obtain a polynomial $P$ of sufficiently high order, we derive a fixed point equation for the high-order remainder term $\varphi$, where $\psi = P + \varphi$ is the desired linearizing factor.
   	Given a sufficiently small, positively invariant, closed ball $N$ centered at the fixed point, the proof of Lemma \ref{lem-make-approx-exact} shows that the spectral spread condition $\nu(e^A,\D_{x_0}\Phi^1)< k+\alpha$ implies that the restriction $\varphi|_N$ of the desired high-order term is the fixed point of a map $S\colon \CKA(N,\C^m)\to \CKA(N,\C^m)$ which is a contraction, with respect to the standard $\CKA$ norm $\norm{\slot}_{k,\alpha}$ making $\CKA(N,\C^m)$ a Banach space, when restricted to the closed linear $S$-invariant subspace $\F\subset  \CKA(N,\C^m)$ of functions with vanishing $i$-th derivatives at the fixed point for all integers $0\leq i < k+\alpha$.\footnote{Note, however, that $\norm{\slot}_{k,\alpha}$ must be induced by an appropriate underlying \concept{adapted norm} \cite[Sec.~A.1]{cabre2003parameterization1} on $\R^n$ to ensure that $S$ is a contraction.}
   	In fact, $S$ is the affine map defined by 
    	\begin{equation}\label{eq:contraction}
    	S(\varphi|_N)\coloneqq -P|_N + e^{-A}\left(P|_N + \varphi|_N\right)\circ \Phi^1.
    	\end{equation}
    Hence we can obtain $\varphi|_N$ by the standard contraction mapping theorem, thereby obtaining the function $\psi|_N = \varphi|_N + P|_N\in \CKA(N,\C^m)$ satisfying $\psi|_N \circ \Phi^1|_N = e^A \psi|_N$.
    (The preceding techniques are an extension of Sternberg's \cite{sternberg1957local} and owe much to Sternberg's work.)
    We then extend the domain of $\psi|_N$ using the globalization techniques of \cite{lan2013linearization,kvalheim2018global} to obtain a function $\psi \in \CLKA(Q,\C^m)$ defined on the entire basin $Q$ and satisfying $\psi \circ \Phi^1 = e^A \psi$.
    To show that the function $\psi$ satisfies \eqref{eq:main-th-3} when $\Grp = \R$, i.e., that $\psi$ is actually a linearizing factor of $\Phi^t$ for \emph{all} $t\in \R$, we use an argument of Sternberg \cite[Lem.~4]{sternberg1957local} in combination with the uniqueness statement of Theorem \ref{th:main-thm}.
    We extend the result to the case that $k = +\infty$ using a bootstrapping argument.
        
        \begin{Rem}[a numerical consideration]\label{rem:numerical}
        Our proof of the existence portion of Theorem \ref{th:main-thm}, outlined above, was inspired by Sternberg's proof of his linearization theorem \cite[Thms~2,~3,~4]{sternberg1957local} and also has strong similarities with the techniques used to prove the existence of semiconjugacies of the type \eqref{eq:cd-param-method} using the parameterization method \cite{cabre2003parameterization1,cabre2003parameterization2,cabre2005parameterization3}.
        We repeat here an observation of \cite[Sec.~3]{cabre2003parameterization1} and \cite[Rem.~5.5]{cabre2005parameterization3} which is also relevant for numerical computations of linearizing semiconjugacies of the type \eqref{eq:cd-gen-efunc} (such as Koopman eigenfunctions) based on our proof of Theorem \ref{th:main-thm}.
        Consider $P\in \CLKA(\R^n,\C^m)$ satisfying \eqref{eq:main-th-4} as in Remark \ref{rem:main-thm-1-outline}; $N$, $S$, and $\F$ as in the same remark; and an initial guess $\psi_0|_N = P|_N + \varphi_0|_N$ for a local linearizing factor with $\varphi_0|_N\in \F$. If $$\Lip{S}\leq \kappa <  1\qquad \norm{S(\varphi_0|_N)-\varphi_0|_N}\leq \delta$$ 
        where $\Lip{S}$ is the Lipschitz constant of $S$, then the standard proof of the contraction mapping theorem implies the estimate
        \begin{equation}\label{eq:a-posteriori}
        \norm{\varphi|_N - \varphi_0|_N} \leq \delta/(1-\kappa),
        \end{equation}
        where $\varphi|_N\in \F$ is such that $\psi|_N = P|_N + \varphi|_N$ is the unique actual local linearizing factor.
        Thus equation \eqref{eq:a-posteriori} furnishes an upper bound on the distance between the initial guess $\varphi_0|_N$ and the true solution $\varphi|_N$, and can be used for \concept{a posteriori estimates} in numerical analysis.
        \end{Rem}
        
    \end{Rem}

	Theorem \ref{th:main-thm} gave conditions ensuring existence and uniqueness of linearizing factors under spectral spread and nonresonance conditions.
	Before stating Theorem \ref{th:main-thm-per}, we state a result on the uniqueness of linearizing factors which does not assume any nonresonance conditions.
	Proposition \ref{prop:uniqueness-without-nonresonance} follows immediately from Lemma \ref{lem:psi-identically-0} (used to prove the uniqueness statement of Theorem \ref{th:main-thm}) and the fact that $Q$ is diffeomorphic to $\R^{\dim(Q)}$ as mentioned above.
	\begin{Prop}\label{prop:uniqueness-without-nonresonance}
	Fix $k\in \N_{\geq 1}\cup \{+\infty\}$ and $0 \leq \alpha \leq 1$, and let $\Phi\colon Q \times \Grp \to Q$ be a $C^1$ dynamical system with $Q$ the basin of an attracting hyperbolic fixed point $x_0\in Q$,  where $Q$ is a smooth manifold with $\dim(Q)\geq 1$ and either $\Grp = \Z$ or $\Grp = \R$.  
	Let $m\in \N_{\geq 1}$ and $e^{A}\in \GL(m,
	\C)$ have spectral radius $\rho(e^{A})<1$ and satisfy $\nu(e^A,\D_0 F) < k + \alpha$.
	Let $\varphi\in \CLKA(Q,\C^m)$ satisfy $\D_{x_0}^i \varphi = 0$ for all integers\footnote{Note the case $\alpha = 0$ which does not require vanishing of the $k$-th derivative.} $0\leq i < k+\alpha$  and $$\varphi \circ \Phi^1 = e^{A} \varphi.$$
	Then it follows that $\varphi \equiv 0$.
	In particular, if $\varphi = \psi_1 - \psi_2$, then $$\psi_1 = \psi_2.$$	
	\end{Prop}

	\begin{restatable}[Existence and uniqueness of $\CLKA$ global linearizing factors for a limit cycle attractor]{Th}{ThmMainPer}\label{th:main-thm-per}
		Fix $k\in \N_{\geq 1}\cup \{+\infty\}$ and $0 \leq \alpha \leq 1$. 
		Let $\Phi\colon Q \times \R \to Q$ be a $\CLKA$ flow with $Q$ the basin of an attracting hyperbolic $\tau$-periodic orbit with image $\Gamma\subset Q$,  where $Q$ is a smooth manifold with $\dim(Q)\geq 2$.  
		Fix $x_0\in \Gamma$ and let $E^s_{x_0}$ denote the unique $\D_{x_0}\Phi^\tau$-invariant subspace complementary to $\T_{x_0} \Gamma$.
		Let $m\in \N_{\geq 1}$ and $e^{\tau A}\in \GL(m,
		\C)$ have spectral radius $\rho(e^{\tau A})<1$, and let the linear map $B\colon E^s_{x_0} \to \C^m$ satisfy
		\begin{equation}\label{eq:main-th-per-1}
		B \D_{x_0}\Phi^\tau|_{E^s_{x_0}} = e^{\tau A} B.
		\end{equation}
		Assume that $(e^{\tau A},\D_{x_0} \Phi^\tau|_{E^s_{x_0}})$ is $k$-nonresonant, and assume that $\nu(e^{\tau A},\D_{x_0} \Phi^\tau|_{E^s_{x_0}}) < k + \alpha$.
		
		Then there exists a unique $\psi \in \CLKA(Q,\C^m)$ satisfying
		\begin{equation}\label{eq:main-th-per-2}
		\forall t\in \R\colon \psi \circ \Phi^t = e^{t A} \psi, \qquad \D_{x_0} \psi|_{E^s_{x_0}} = B,
		\end{equation}
		and if $B\colon E^s_{x_0}\to \R^m\subset \C^m$ and $A\in \R^{m\times m}\subset \C^{m\times m}$ are real, then $\psi\colon Q\to \R^m \subset \C^m$ is real.		
	\end{restatable}
	
	\begin{Rem}
	For the uniqueness statement of Theorem \ref{th:main-thm} (and Proposition \ref{prop:koopman-cka-fix}) it is only assumed that $\Phi\in C^1$, whereas $\Phi\in \CLKA$ is assumed for the uniqueness statement of Theorem \ref{th:main-thm-per} (and Proposition \ref{prop:koopman-cka-per}).
	This is because our proof of the uniqueness statement of Theorem \ref{th:main-thm-per} relies on $\CLKA$ smoothness of the isochrons \cite{isochrons} or (equivalently) strong stable manifolds \cite{fenichel1974asymptotic,fenichel1977asymptotic,hirsch1977}  of the periodic orbit, a property which is ensured by the assumption that $\Phi\in \CLKA$.
	We leave open the question as to whether the need for this additional assumption is merely an artifact of our proof of Theorem \ref{th:main-thm-per}.
	\end{Rem}
	
	\begin{Rem}
	By considering the Poincar\'{e} (first-return) map with Poincar\'{e} section an isochron \cite{isochrons} and using Theorem \ref{th:main-thm}, it is readily seen that the linearizing factor $\psi$ of Theorem \ref{th:main-thm-per} can also be represented as a limit analogous to \eqref{eq:main-th-5} which converges in the pointwise sense.
	While we believe that this limit also converges in the topology of $C^{k,\alpha}$-uniform convergence on compact subsets of $Q$ as in Theorem \ref{th:main-thm}, we did not attempt to prove this stronger convergence statement. 
	\end{Rem}
	
    \section{Applications}\label{sec:applications}
	In this section, we give some applications of Theorems \ref{th:main-thm} and \ref{th:main-thm-per} and Proposition \ref{prop:uniqueness-without-nonresonance}.
	\S \ref{sec:app:stern-floq} contains results on Sternberg linearizations and Floquet normal forms.
	\S \ref{sec:app-p-eigs} gives applications to principal Koopman eigenfunctions and isostable coordinates.
	\S \ref{sec:app-classify} contains our classification theorems for $C^\infty$ eigenfunctions of $C^\infty$ dynamical systems.
	\subsection{Sternberg linearizations and Floquet normal forms}\label{sec:app:stern-floq}
	
	The following result is an improved statement of Sternberg's linearization theorem for hyperbolic sinks \cite[Thms~2,3,4]{sternberg1957local}.
	Our improvements include: uniqueness of the linearizing conjugacy, refined $\CLKA$ regularity rather than just $C^k$, and global definition of the linearization on the entire basin of attraction $Q$ rather than just on some small neighborhood of $x_0$.
	Our techniques for globalizing the domain of the linearization are essentially the same as those used in \cite{lan2013linearization,kvalheim2018global}.
	For the case $\Grp = \R$, the $k$-nonresonance assumption in the following result can be replaced by the slightly less restrictive nonresonance condition in Remark \ref{rem:weaker-nonresonance-real-time}.

	\begin{Prop}[Existence and uniqueness of global $\CLKA$ Sternberg linearizations]\label{prop:sternberg}
		Fix $k\in \N_{\geq 1}\cup \{+\infty\}$ and $0 \leq \alpha \leq 1$.
		Let $\Phi\colon Q \times \Grp \to Q$ be a $\CLKA$ dynamical system with $Q$ the basin of an attracting hyperbolic fixed point $x_0\in Q$,  where $Q$ is a smooth manifold with $\dim(Q)\geq 1$ and either $\Grp = \Z$ or $\Grp = \R$.  
		Assume that $\nu(\D_{x_0} \Phi^1,\D_{x_0} \Phi^1) < k + \alpha$, and assume that $(\D_{x_0} \Phi^1,\D_{x_0} \Phi^1)$ is $k$-nonresonant.
		
		Then there exists a unique diffeomorphism $\psi\in \CLKA(Q,\T_{x_0}Q)$ satisfying
		\begin{equation}\label{eq:sternberg-lin}
		\forall t \in \Grp\colon \psi \circ \Phi^t = \D_{x_0} \Phi^t \psi, \qquad \D_{x_0}\psi = \id_{\T_{x_0}Q}.
		\end{equation}
		(In writing $\D_{x_0}\psi = \id_{\T_{x_0}Q}$, we are making the standard and canonical identification $\T_{0}(\T_{x_0}Q)\cong \T_{x_0} Q$.)
	\end{Prop}
	\begin{Rem}[Uniqueness of general linearizing conjugacies modulo a linear coordinate transformation]\label{rem:sternberg}
	Under the hypotheses of Proposition \ref{prop:sternberg}, let $L_1, L_2\colon \T_{x_0}Q\times \Grp\to \T_{x_0}Q$ be any linear dynamical systems (i.e., such that each time-$t$ map $L_i^t$ is linear)  and  $\psi_1,\psi_2\in \CLKA(Q,\T_{x_0}Q)$ be any diffeomorphisms satisfying $\psi_i \circ \Phi^t = L_i^t \psi_i$ for all $t\in \Grp$ and $i\in \{1,2\}$.
	Differentiation at the fixed point $x_0$ using the chain rule yields $\D_{x_0}\psi_i \D_{x_0}\Phi^t = L_i^t \D_{x_0}\psi_i$ for all $t\in \Grp$, so $L_i^t = B_i \D_{x_0}\Phi^t B_i^{-1}$ where $B_i\coloneqq \D_{x_0}\psi_i$.
	It follows that $B_i^{-1}\psi_i \circ \Phi^t = \D_{x_0}\Phi^t B_i^{-1}\psi_i$ for all $t\in \Grp$, so $B_i^{-1}\psi_i$ satisfies \eqref{eq:sternberg-lin} for $i\in \{1,2\}$.
	The uniqueness statement of Proposition \ref{prop:sternberg} then implies that $\psi_i = B_i \psi$, from which it follows that $\psi_1 = B_1 B_2^{-1}\psi_2$.
	\end{Rem}
	\begin{proof}
		Identifying $\T_{x_0} Q$ with $\R^n$ by choosing a basis and letting $A\in \GL(n,\C)$ be any matrix logarithm of $\D_{x_0}\Phi^1$, we apply Theorem \ref{th:main-thm} with $e^{tA} = \D_{x_0} \Phi^t$ and $B = \id_{\T_{x_0} Q}$ to obtain a unique $\psi\in \CLKA(Q,\T_{x_0} Q)$ satisfying \eqref{eq:sternberg-lin} and $\D_{x_0} \psi = \id_{\T_{x_0}Q}$.
		It remains only to show that $\psi$ is a diffeomorphism.
		To do this, we separately show that $\psi$ is injective, surjective, and a local diffeomorphism.
		
		By continuity, $\D_{x_0} \psi = \id_{\T_{x_0}Q}$ implies that $\D_x \psi$ is invertible for all $x$ in some neighborhood $U \ni x_0$.
		Since $Q = \bigcup_{t\geq 0}\Phi^{-t}(U)$ by asymptotic stability of $x_0$, \eqref{eq:sternberg-lin} and the chain rule imply that $\D_{x}\psi$ is invertible for all $x\in Q$.
		Hence $\psi$ is a local diffeomorphism.
		
		To see that $\psi$ is injective, let $U$ be a neighborhood of $x_0$ such that $\psi|_U\colon U\to \psi(U)$ is a diffeomorphism, and let $x,y \in Q$ be such that $\psi(x) = \psi(y)$.
		By asymptotic stability of $x_0$, there is $T > 0$ such that $\Phi^T(x),\Phi^T(y)\in U$, and  \eqref{eq:sternberg-lin} implies that $\psi\circ \Phi^T(x) = \psi \circ \Phi^T(y)$.
		Injectivity of $\psi|_U$ then implies that $\Phi^T(x) = \Phi^T(y)$, and injectivity of $\Phi^T$ then implies that $x = y$. 
		Hence $\psi$ is injective.
		
		To see that $\psi$ is surjective, fix any $y\in \T_{x_0}Q$ and let the neighborhood $U$ be as in the last paragraph.
		Asymptotic stability of $0$ for $\D_{x_0}\Phi \colon \T_{x_0}Q\times \Grp\to \T_{x_0}Q$ implies that there is $T > 0$ such that $\D_{x_0} \Phi^T\cdot y \in \psi(U)$, so there exists $x\in U$ with $\D_{x_0}\Phi^T\cdot y = \psi(x)$.
		Hence $y = \D_{x_0} \Phi^{-T}\cdot\psi(x) = \psi \circ \Phi^{-T}(x)$, where we have used \eqref{eq:sternberg-lin}.
		It follows that $\psi$ is surjective.
		This completes the proof.
	\end{proof}
	
	The following is an existence and uniqueness result for the $\CLKA$ Floquet normal form of an attracting hyperbolic periodic orbit of a flow, a nonlinear change of coordinates in which the dynamics become the product of a linear system with constant-rate rotation on a circle.
	References discussing the Floquet normal form include \cite[Sec.~26]{abraham1967transversal}, \cite[Sec.~I.3]{smoothInvariant}, and \cite[Sec.~4.3]{abraham2001manifolds}; the Floquet normal form is a nonlinear generalization of the classical Floquet theory of linear time-periodic systems \cite[Sec.~III.7]{hale1980ode}.
	The following result is proved using a combination of Proposition \ref{prop:sternberg} and stable manifold theory \cite{fenichel1974asymptotic,fenichel1977asymptotic,hirsch1977,ruelle1989elements} specialized to the theory of isochrons \cite{isochrons}.
	For the statement, recall that a map between manifolds is a $C^1$ embedding if it is a homeomorphism onto its image (equipped with the subspace topology) and if its derivative is everywhere one-to-one \cite[p.~21]{hirsch1976differential}; a $\CLKA$ embedding is a $\CLKA$ map which is also a $C^1$ embedding.
	A map between topological spaces is proper if the preimage of every compact subset is compact \cite[p.~610]{lee2013smooth}.
	\begin{Prop}[Existence and uniqueness of $\CLKA$ global Floquet normal forms]\label{prop:floq-norm-form}
		Fix $k\in \N_{\geq 1}\cup \{+\infty\}$ and $0 \leq \alpha \leq 1$.
		Let $\Phi\colon Q \times \R \to Q$ be a $\CLKA$ flow with $Q$ the basin of an  attracting hyperbolic $\tau$-periodic orbit with image $\Gamma \subset Q$,  where $Q$ is a smooth manifold with $\dim(Q)\geq 2$.
		Fix $x_0 \in \Gamma$ and let $E^s_{x_0}\subset \T_{x_0}Q$ denote the unique $\D_{x_0}\Phi^\tau$-invariant subspace complementary to $\T_{x_0}\Gamma$.  
		Assume that $\nu(\D_{x_0} \Phi^\tau|_{E^s_{x_0}},\D_{x_0} \Phi^\tau|_{E^s_{x_0}}) < k + \alpha$, and assume that $(\D_{x_0} \Phi^\tau|_{E^s_{x_0}},\D_{x_0} \Phi^\tau|_{E^s_{x_0}})$ is $k$-nonresonant.
		
		Then if we write $\D_{x_0}\Phi^\tau|_{E^s_{x_0}}=e^{\tau A}$ for some complex linear $A\colon (E^s_{x_0})_\C \to (E^s_{x_0})_\C$, there exists a unique, proper, $\CLKA$ embedding $\psi = (\psi_\theta, \psi_z) \colon Q\to S^1 \times (E^s_{x_0})_\C$ such that $\psi_\theta(x_0) = 1$, $(\D_{x_0}\psi_z)|_{E^s_{x_0}} = (E^s_{x_0} \hookrightarrow (E^s_{x_0})_\C)$, and 
		\begin{equation}\label{eq:floquet}
		\forall t \in \R\colon \psi_\theta \circ \Phi^t(x) = e^{2\pi i \frac{t}{\tau}}\psi_\theta(x), \qquad \psi_z\circ \Phi^t(x) = e^{tA} \psi_z(x),  
		\end{equation}
		where $S^1\subset \C$ is the unit circle and $i=\sqrt{-1}$. 
		If $A|_{E^s_{x_0}}\colon E^s_{x_0}\to E^s_{x_0}\subset (E^s_{x_0})_\C$ is real, then $\psi_z\in \CLKA(Q, E^s_{x_0})$ is real, and the codomain-restricted map $\psi \colon Q \to S^1 \times E^s_{x_0}\subset S^1\times (E^s_{x_0})_\C$ is a diffeomorphism. 
	\end{Prop}
	\begin{proof}
		Theorem \ref{th:main-thm-per} implies that a map $\psi_z\in \CLKA(Q,(E^s_{x_0})_\C)$ satisfying all applicable conclusions above exists.
		Letting $\Ws_{x_0}$ denote the global strong stable manifold (isochron) through $x_0$, we have $\T_{x_0}\Ws_{x_0}= E^s_{x_0}$ and $\Phi^\tau(\Ws_{x_0}) = \Ws_{x_0}$.
		Since $\Ws_{x_0}$ is the stable manifold for the fixed point $x_0$ of the $\CLKA$ diffeomorphism $\Phi^\tau$, it follows that $\Ws_{x_0}$ is a $\CLKA$ submanifold \cite[pp.~2, 27; Thm~6.1]{ruelle1989elements} which is properly embedded (rather than merely immersed) in $Q$ because $\Gamma$ is stable \cite[pp.~4208--4209]{kvalheim2018global}.
		Proposition \ref{prop:sternberg} then implies that $\psi_z|_{\Ws_{x_0}}\colon \Ws_{x_0}\to E^s_{x_0} \subset (E^s_{x_0})_\C$ is a diffeomorphism onto its image $E^s_{x_0}$.\footnote{Strictly speaking, Proposition \ref{prop:sternberg} was---for simplicity---stated for smooth manifolds.
		Hence in order to apply Proposition \ref{prop:sternberg} here (and also in the proofs of Theorems \ref{th:main-thm-per} and \ref{th:classify-per}) we must first give $\Ws_{x_0}$ a compatible $C^\infty$ structure, but this can always be done \cite[Thm~2.2.9]{hirsch1976differential}, so we will not mention this anymore.}
		Since $E^s_{x_0}$ is closed in $(E^s_{x_0})_\C$, it follows that $\psi_z|_{\Ws_{x_0}}\colon \Ws_{x_0}\to (E^s_{x_0})_\C$ is a proper $\CLKA$ embedding \cite[Prop.~A.53(c)]{lee2013smooth}.

		Define $\exp_A\colon (E^s_{x_0})_\C\times \R\to (E^s_{x_0})_\C$ via $\exp_A(z,t)\coloneqq e^{tA}z$, let $K\subset (E^s_{x_0})_\C$ be any subset, define $K_\tau\coloneqq \exp_A(K\times [-\tau,0])$, and define $J_{\tau} \coloneqq (\psi_z|_{\Ws_{x_0}})^{-1}(K_\tau)$.
		From the second equation of \eqref{eq:floquet}, we have that $$\psi_z^{-1}(K) = \bigcup_{t\in [0,\tau]}\Phi^t\left((\psi_z|_{\Ws_{x_0}})^{-1}(e^{-tA}K)\right) \subset \Phi(J_\tau\times [0,\tau]).$$
		If $K$ is compact, then so are $K_\tau$ and $J_\tau$ by continuity of $\exp_A$, properness of $\psi_z|_{\Ws_{x_0}}$, and the fact that $\Ws_{x_0}$ is a closed subset of $Q$ since it is properly embedded \cite[Prop.~5.5]{lee2013smooth}.
		Thus, if $K$ is compact, continuity implies that $\psi_z^{-1}(K)$ is a closed subset of the compact set $\Phi(J_\tau\times [0,\tau])$ and is therefore compact.
		This establishes that $\psi_z\colon Q\to (E^s_{x_0})_\C$ is proper.

       Since the vector field generating $\Phi$ intersects $\Ws_{x_0}$ transversely, a standard argument \cite[p.~243]{hirsch1974differential} and the $\CLKA$ implicit function theorem \cite[Cor.~A.4]{eldering2013normally} imply that a real-valued $\CLKA$ ``time-to-impact $\Ws_{x_0}$'' function can be defined on a neighborhood of any point.
		Using these facts, one can show that the function $\psi_\theta\colon Q\to S^1$ defined via $\psi_\theta(\Ws_{x_0})\equiv 1$ and $\psi_\theta(\Phi^t(\Ws_{x_0})) \equiv e^{2\pi i \frac{t}{\tau}}$ is a $\CLKA$ function.
		By construction, this function $\psi_\theta$ satisfies $\psi_\theta(x_0) = 1$ and \eqref{eq:floquet}.
		$\psi_\theta$ is unique among all continuous functions satisfying these equalities, since if $\tilde{\psi}_\theta$ is any other such function, then asymptotic stability of $\Gamma$ implies that the quotient $(\psi_\theta/\tilde{\psi}_\theta)$ is constant on $Q$, and since $(\psi_\theta/\tilde{\psi}_\theta)(x_0) = 1$ it follows that $\tilde{\psi}_\theta \equiv \psi_\theta$.        
         
        Note that, for all $t\in \R$ and $x\in \Ws_{\Phi^{t}(x_0)}= \Phi^t(\Ws_{x_0})$,  $$\ker\left(\D_x \psi_\theta\right)=\T_x \Ws_{\Phi^{t}(x_0)}\qquad \textnormal{and}\qquad \ker\left(\D_x \psi_z\right)\cap \T_x \Ws_{\Phi^t(x_0)}=\{0\},$$ with the second equality following since $\psi_z|_{\Ws_{\Phi^t(x_0)}} = e^{tA}\circ \psi_z|_{\Ws_{x_0}}\circ \Phi^{-t}|_{\Ws_{\Phi^t(x_0)}}$ is a $\CLKA$ embedding.
        It follows that $\psi\coloneqq (\psi_\theta, \psi_z)\colon Q\to S^1\times (E^s_{x_0})_\C$ is an immersion, i.e., that $\D_x \psi$ is injective for all $x\in Q$.
        Furthermore, $\psi$ is injective since the restriction of $\psi_z$ to any level set $\Ws_{\Phi^t(x_0)}$ of $\psi_{\theta}$ is the composition of injective maps $e^{tA} \circ \psi_z|_{\Ws_{x_0}} \circ \Phi^{-t}|_{\Ws_{\Phi^t(x_0)}}$.
        Let $\pi_z\colon S^1\times (E^s_{x_0})_\C\to (E^s_{x_0})_\C$ be projection onto the second factor. 
        Since $\psi^{-1}(K)\subset \psi_z^{-1}(\pi_z(K))$ for any subset $K$, properness of $\psi_z$ and continuity of $\pi_z$ and $\psi$ imply that $\psi$ is also proper.
        Since (i) proper maps between manifolds are closed maps \cite[Thm~A.57]{lee2013smooth}, (ii) closed injective continuous maps are homeomorphisms onto their images \cite[Lem.~A.52.C]{lee2013smooth}, and (iii) $\psi$ is a $\CLKA$ injective immersion, it follows that $\psi$ is a proper $\CLKA$ embedding.
		
		If $A$ is real, then the image of the proper embedding $\psi$ is contained in $S^1\times E^s_{x_0}\subset S^1\times (E^s_{x_0})_\C$.
		Since $\dim(S^1\times E^s_{x_0})=\dim(Q)$, and since $C^1$ proper embeddings between manifolds of the same dimension are both open and closed maps \cite[Prop.~4.28,~Thm~A.57]{lee2013smooth}, it follows that the image of $\psi$ is both open and closed in $S^1\times E^s_{x_0}$.
		Since $S^1\times E^s_{x_0}$ is connected, it follows that the image of $\psi$ is all of $S^1\times E^s_{x_0}$ \cite[Prop.~A.39(e)]{lee2013smooth}.
		Thus, $\psi$ is a $\CLKA$ diffeomorphism onto $S^1\times E^s_{x_0}$ if $A$ is real.
		This completes the proof.	
		\end{proof}	
	
	\subsection{Principal Koopman eigenfunctions, isostables, and isostable coordinates}\label{sec:app-p-eigs}
	Given a $C^1$ dynamical system $\Phi\colon Q\times \Grp \to Q$, where $Q$ is a smooth manifold and either $\Grp = \Z$ or $\Grp = \R$, we say that $\psi \colon Q\to \C$ is a Koopman \concept{eigenfunction} if $\psi$ is not identically zero and satisfies
	\begin{equation}\label{eq:koopman-efunc}
	\forall t \in \Grp\colon \psi \circ \Phi^t = e^{\mu t} \psi
	\end{equation}
	for some $\mu \in \C$.	
	The following are intrinsic definitions of principal eigenfunctions and the principal algebra which extend the definitions for linear systems given in \cite[Def.~2.2--2.3]{mohr2016koopman}; a more detailed comparison is given later in Remark \ref{rem:mohr-mezic}.
	The condition $\psi|_\Gamma \equiv 0$ was motivated in part by the definition of a certain space $\F_{A_c}$ of functions in \cite[p.~3358]{mauroy2016global}.
	\begin{Def}\label{def:principal-eigenfunction}
		If $Q$ is the basin of an asymptotically stable fixed point $x_0\in Q$ for $\Phi$, we say that an eigenfunction $\psi \in C^1(Q)$ is a \concept{principal} eigenfunction if $\psi(x_0) = 0$ and $\D_{x_0}\psi \neq 0$.
		If instead $Q$ is the basin of an asymptotically stable periodic orbit with image $\Gamma \subset Q$ for $\Phi$, we say that an eigenfunction $\psi \in C^1(Q)$ is a \concept{principal} eigenfunction if $\psi|_{\Gamma}\equiv 0$ and $\D_{x_0}\psi \neq 0$ for all $x_0\in \Gamma$.\footnote{By \eqref{eq:koopman-efunc} and the chain rule, it suffices to assume there exists one point $x_0\in \Gamma$ such that $\psi(x_0)\neq 0$ and $\D_{x_0}\psi \neq 0$.}
        In either case, we define the $\CLKA$ \concept{principal algebra} $\A^{k,\alpha}_{\Phi}$ to be the complex subalgebra of $\CLKA(Q,\C)$ generated by all $\CLKA$ principal eigenfunctions.
	\end{Def}

	Given a (real or complex) linear self-map $Y\colon V\to V$, we say that a linear map $w\colon V\to \C$ is a \concept{left eigenvector} of $Y$ with \concept{eigenvalue} $\lambda\in \C$ if $wY = \lambda w$. 
	(If a basis is chosen for $V$, then $Y$ can be identified with a matrix and $w$ can be identified with a row vector so that $wY = \lambda w$ in the usual sense of matrix multiplication.)
	Differentiating \eqref{eq:koopman-efunc} and using the chain rule immediately yields Propositions \ref{prop:p-eig-evec-point} and \ref{prop:p-eig-evec-cycle}, which have previously appeared in the literature (see, e.g., the proof of \cite[Prop.~2]{mauroy2016global}; our stability assumptions are for convenience of exposition and are not necessary).
	\begin{Prop}\label{prop:p-eig-evec-point}
		Let $Q$ be the basin of an asymptotically stable fixed point $x_0$ for the $C^1$ dynamical system $\Phi\colon Q\times \Grp \to Q$. If $\psi$ is a principal Koopman eigenfunction for $\Phi$ satisfying \eqref{eq:koopman-efunc} with exponent $\mu\in \C$, then for any $t\in \Grp$, it follows that $\D_{x_0}\psi$ is a left eigenvector of $\D_{x_0} \Phi^t$ with eigenvalue $e^{\mu t}$. 	
	\end{Prop}
	
	\begin{Prop}\label{prop:p-eig-evec-cycle}
		Let $Q$ be the basin of an asymptotically stable $\tau$-periodic orbit with image $\Gamma\subset Q$ for the $C^1$ dynamical system $\Phi\colon Q\times \R \to Q$.		
		If $\psi$ is a principal Koopman eigenfunction for $\Phi$ satisfying \eqref{eq:koopman-efunc} with exponent $\mu\in \C$, then for any $x_0 \in \Gamma$, it follows that $\D_{x_0}\psi$ is a left eigenvector of $\D_{x_0}\Phi^\tau$ with eigenvalue $e^{\mu \tau}$; in particular, $e^{\mu \tau}$ is a Floquet multiplier for $\Gamma$.
	\end{Prop}	
	
	\begin{Rem}\label{rem:koop-eig-modulus-for-stable}
	For a dynamical system with $Q$ the basin of an attracting compact invariant set $M$, any continuous eigenfunction defined on $Q$ satisfying \eqref{eq:koopman-efunc} with exponent $\mu \in \C$ must have $|e^{\mu}|\leq 1$.
    If this attracting set $M$ is furthermore a hyperbolic fixed point, then there is the stronger observations that either $e^{\mu} = 1$ or $|e^{\mu}|< 1$.
    These observations are straightforward consequences of continuity and \eqref{eq:koopman-efunc}.
	\end{Rem}
	\begin{Rem}\label{rem:koopman-spectral-spread}
	Let $\mu \in \C$ and consider the case that $Q$ is the basin of an attracting hyperbolic fixed point $x_0\in Q$ for the $C^1$ dynamical system $\Phi\colon Q\times \Grp \to Q$. 
	Recall that the spectral radius $\rho(\D_{x_0}\Phi^1)\in (0,1)$ of $\D_{x_0}\Phi^1$ is defined to be the largest modulus (absolute value) of the eigenvalues of (the complexification of) $\D_{x_0}\Phi^1$. 
	From Equation \eqref{eq:spread-def} of Definition \ref{def:nonres-spread}, the spectral spread $\nu(e^\mu,\D_{x_0}\Phi^1)$ satisfies
	\begin{equation}
	\nu(e^\mu,\D_{x_0}\Phi^1) = \min\left\{r\in \R\colon |e^\mu|\geq\left(\rho\left(\D_{x_0}\Phi^1\right)\right)^r\right\}.
	\end{equation}
	It follows that, for any $r\in \R\cup \{+\infty\}$,
	\begin{equation}\label{eq:koopman-spread-iff}
	\nu(e^\mu,\D_{x_0}\Phi^1)<r \qquad \iff \qquad  |e^\mu|>\left(\rho\left(\D_{x_0}\Phi^1\right)\right)^r
	\end{equation}	
	where $\left(\rho\left(\D_{x_0}\Phi^1\right)\right)^{+\infty}\coloneqq 0$ in the special case that $r=+\infty$.
	\end{Rem}
	In light of Remarks \ref{rem:koop-eig-modulus-for-stable} and \ref{rem:koopman-spectral-spread} (taking $r = k+\alpha$ in \eqref{eq:koopman-spread-iff}), Proposition \ref{prop:koopman-cka-fix} below is now nearly immediate from Theorem \ref{th:main-thm} and Proposition \ref{prop:p-eig-evec-point}.
	We prove the non-immediate portion following the statement of Proposition \ref{prop:koopman-cka-fix}.
    We emphasize that, when $k=+\infty$ in Proposition \ref{prop:koopman-cka-fix}, the spectral spread condition $$|e^\mu|>\left(\rho\left(\D_{x_0}\Phi^1\right)\right)^{k+\alpha}= \left(\rho\left(\D_{x_0}\Phi^1\right)\right)^{+\infty}\coloneqq 0$$ is automatically satisfied since $|e^\mu|>0$ for all $\mu \in \C$ (cf. Remark \ref{rem:point-cinf-case}).
    
   	For the case $\Grp = \R$, the $k$-nonresonance assumptions in the following result can be replaced by the slightly less restrictive nonresonance condition in Remark \ref{rem:weaker-nonresonance-real-time}.
   	Applications-oriented readers may find it helpful to consult Example \ref{ex:intuition} with $m=1$ for remarks relevant to the following proposition.

   	\begin{Prop}[Existence and uniqueness of $\CLKA$ Koopman eigenvalues and principal eigenfunctions for a point attractor]\label{prop:koopman-cka-fix}
   		Let $\Phi\colon Q \times \Grp \to Q$ be a $C^{1}$ dynamical system with $Q$ the basin of an attracting hyperbolic fixed point $x_0\in Q$,  where $Q$ is a smooth manifold with $\dim(Q)\geq 1$ and either $\Grp = \Z$ or $\Grp = \R$.
   		Fix $k \in \N_{\geq 1}\cup\{+\infty\}$ and $0 \leq \alpha \leq 1$, and assume that the spectral radius $\rho\left(\D_{x_0}\Phi^1\right)\in (0,1)$ satisfies 
   		$$|e^\mu|>\left(\rho\left(\D_{x_0}\Phi^1\right)\right)^{k+\alpha}$$
   		in all of the following statements.	
			
		\emph{\textbf{Uniqueness of Koopman eigenvalues and principal eigenfunctions.}}
		Let $\psi_1\in \CLKA(Q,\C)$ be any Koopman eigenfunction satisfying \eqref{eq:koopman-efunc} with exponent $\mu \in \C$.
		\begin{enumerate}
			\item\label{item:uniq-thm-1} Then there exists $m = (m_1,\ldots,m_n)\in \N_{\geq 0}^n$ such that $$e^{\mu} =  e^{m\cdot \lambda},$$ where $e^{\lambda_1},\ldots,e^{\lambda_n}$ are the eigenvalues of $\D_{x_0} \Phi^1$ repeated with multiplicities and $\lambda\coloneqq (\lambda_1,\ldots,\lambda_n)$. 
			\item\label{item:uniq-thm-2} Assume that $\psi_1$ is a principal eigenfunction so that $e^{\mu} \in \textnormal{spec}(\D_{x_0}\Phi^1)$, and assume that $(e^\mu,\D_{x_0}\Phi^1)$ is $k$-nonresonant.
			Then $\psi_1$ is uniquely determined by $\D_{x_0}\psi_1$, and if $e^{\mu}$ and $\D_{x_0}\psi_1$ are real, then $\psi_1\colon Q\to \R\subset \C$ is real.
			In particular, if $e^{\mu}$ is an algebraically simple eigenvalue of (the complexification of) $\D_{x_0}\Phi^1$ and if $\psi_2$ is any other principal eigenfunction satisfying \eqref{eq:koopman-efunc} with the same exponent $\mu$, then there exists $c\in \C\setminus \{0\}$ such that $$\psi_1 = c\psi_2.$$		
		\end{enumerate}   		
   		
   		\emph{\textbf{Existence of principal eigenfunctions.}} Assume that $\Phi \in \CLKA$ and that $(e^\mu,\D_{x_0}\Phi^1)$ is $k$-nonresonant.
   		Let $w\colon \T_{x_0} Q\to \C$ be any left eigenvector of $\D_{x_0}\Phi^1$ with eigenvalue $e^\mu$ such that
   		$$\forall t\in \Grp\colon w\D_{x_0}\Phi^t = e^{\mu t}w.$$
   		\begin{enumerate}
   			\item Then there exists a unique principal eigenfunction $\psi \in \CLKA(Q,\C)$ satisfying \eqref{eq:koopman-efunc} with exponent $\mu$ and satisfying  $\D_{x_0}\psi = w$.
   			\item  In fact, if $P\in \CLKA(Q,\C)$ is any ``approximate eigenfunction'' satisfying $\D_{x_0}P = w$ and 
   			\begin{equation}\label{eq:Koop-main-approx}
   			P \circ \Phi^1 = e^\mu P + R
   			\end{equation} 
   			with $\D_{x_0}^i R = 0$ for all integers $0\leq i < k+\alpha$, then 
   			\begin{equation}\label{eq:Koop-main-converge}
   			\psi = \lim_{t\to \infty} e^{-\mu t} P \circ \Phi^t
   			\end{equation}
   			in the topology of $C^{k,\alpha}$-uniform convergence on compact subsets of $Q$  if $k < +\infty$, and in the topology of $C^{k'}$-uniform convergence on compact subsets of $Q$ for any $k'\in \N_{\geq 1}$ if $k=+\infty$.
   		\end{enumerate}
   	\end{Prop}
   	\begin{proof}
   	All of the claims are immediate from Theorem \ref{th:main-thm} and Proposition \ref{prop:p-eig-evec-point} except for the first uniqueness claim, which we now justify.
   	Suppose (to obtain a contradiction) that the first uniqueness claim does not hold.
   	Then (i) $e^{\mu}$ is not an eigenvalue of $\D_{x_0}\Phi^1$ and (ii) $(e^{\mu},\D_{x_0}\Phi^1)$ is $\infty$-nonresonant. 
   	The first observation together with Proposition \ref{prop:p-eig-evec-point} implies $\psi$ is not a principal eigenfunction, i.e., $\D_{x_0}\psi=0$.
   	From Remark \ref{rem:koopman-spectral-spread} and the uniqueness portion of Theorem~\ref{th:main-thm} (with $A=\mu$ and $B=0$), it follows that $\psi\equiv 0$ is identically zero.
   	However, Koopman eigenfunctions are (by definition) not identically zero, so we have obtained a contradiction.
   	\end{proof}

   \begin{Rem}[Laplace averages]\label{rem:laplace}
   	Given $P\colon Q\to \C$ and $\Grp = \R$, in the Koopman literature the \concept{Laplace average} $$\psi\coloneqq  \lim_{T\to\infty} \frac{1}{T}\int_0^T e^{-\mu t} P\circ\Phi^t\, dt$$ is used to produce a Koopman eigenfunction satisfying \eqref{eq:koopman-efunc} with exponent $\mu$ as long as the limit exists \cite{mauroy2013isostables,mohr2014construction}.
   	(When $\Grp = \Z$, a similar definition can be given with a sum replacing the integral.)
   	Since convergence of the limit \eqref{eq:Koop-main-converge} easily implies convergence of the Laplace average to the same limiting function, the existence portion of Proposition \ref{prop:koopman-cka-fix} gives sufficient conditions under which the Laplace average of $P$ exists and is equal to a unique $\CLKA$ principal eigenfunction $\psi$ satisfying $\D_{x_0}\psi = \D_{x_0}P$.
   \end{Rem}
   	
	\begin{Rem}[Isostables and isostable coordinates]\label{rem:isostable}
		
		It follows from the discussion after \cite[Def.~2]{mauroy2013isostables} that the definition of \concept{isostables} given in that paper---for $\Phi$ having an attracting hyperbolic fixed point $x_0$ with basin of attraction $Q$ and with $\D_{x_0}\Phi^1$ having a unique eigenvalue $e^{\mu_1}$ (or complex conjugate pair of eigenvalues) of largest modulus---is equivalent to the following. 
		Isostables as defined in \cite{mauroy2013isostables} are the level sets of the modulus $|\psi_1|$ of a principal eigenfunction $\psi_1$ defined on $Q$ and satisfying \eqref{eq:koopman-efunc} with exponent $\mu = \mu_1$.
		Because $e^{\mu_1}$ is the ``slowest'' eigenvalue of $\D_{x_0}\Phi^1$, Proposition \ref{prop:koopman-cka-fix} implies that, for any $\alpha > 0$, any such $\psi_1\in \CL^{1,\alpha}(Q,\C)$ satisfying \eqref{eq:koopman-efunc} with exponent $\mu_1$ is unique modulo scalar multiplication for a $C^1$ dynamical system $\Phi$ without any further assumptions (since $|e^{\mu_1}|>|e^{\mu_1}|^{1+\alpha}=\left(\rho\left(\D_{x_0}\Phi^1\right)\right)^{1+\alpha}$).
		Furthermore, such a unique eigenfunction always exists if $\Phi\in C^{1,\alpha}_{\textnormal{loc}}$ and if $w\D_{x_0}\Phi^t = e^{\mu t}w$ for all $t\in \Grp$, where $w$ is a left eigenvector of $\D_{x_0}\Phi^1$ with eigenvalue $e^{\mu}$.\footnote{By Example \ref{ex:intuition}, if $\Grp = \R$ and $\Phi$ is the flow of a $\CL^{1,\alpha}$ vector field $f$, these assumptions are automatically satisfied if $w$ is a left eigenvector of $\D_{x_0}f$ with eigenvalue $\mu$.}
		Since the complex conjugate $\bar{\psi}_1$ is a principal eigenfunction satisfying \eqref{eq:koopman-efunc} with exponent $\mu = \bar{\mu}_1$, it follows that the isostables as defined in \cite{mauroy2013isostables} are unique even if $\mu_1 \in \C \setminus \R$. 
		A uniqueness proof for analytic isostables under the additional assumptions of $(\D_{x_0}\Phi^1, \D_{x_0}\Phi^1)$ $\infty$-nonresonance and of dynamics generated by an analytic vector field was given in \cite[App.~A]{mauroy2013isostables}.
		For the special case that the eigenvalue of largest modulus is real, unique, and algebraically simple, in \cite[p.~23]{mauroy2013isostables} these authors do point out that uniqueness of $C^1$ isostables (if they exist) follows from the fact that they coincide with the unique $C^1$ \concept{global (strong) stable manifolds} \cite[pp.~4208,~4211]{kvalheim2018global} over\footnote{In general, the (strong) stable manifolds are the \concept{leaves} of the unique \concept{global (strong) stable foliation} \cite[p.~4208]{kvalheim2018global} of the \concept{global (center-)stable manifold} \cite[p.~4208]{kvalheim2018global} of an (inflowing) normally hyperbolic invariant manifold \cite{fenichel1971persistence,hirsch1977,eldering2013normally}; these leaves generalize the isochrons \cite{isochrons} of an attracting hyperbolic limit cycle.\label{foot:foliation}}  an attracting, \concept{normally hyperbolic} \cite[p.~4207]{kvalheim2018global}, $1$-dimensional, \concept{inflowing invariant manifold} \cite[p.~4211]{kvalheim2018global}; this argument works even if the dynamical system is only $C^1$ (see \cite{kvalheim2018global} for detailed information on the global stable foliation of an inflowing invariant manifold).
		The $1$-dimensional invariant ``slow'' manifold is itself generally non-unique without further assumptions, but this does not affect the isostable uniqueness argument.
		However, as pointed out in \cite[p.~23]{mauroy2013isostables}, this argument does not work when the eigenvalue of largest modulus is not real, because in this case the isostables can no longer be interpreted as strong stable manifolds (e.g., the relevant slow manifold is now $2$-dimensional, so the dimension of the codimension-1 isostables is too large by $1$).
		
		For the case that $\Grp = \R$ and $\Phi$ has an attracting hyperbolic periodic orbit, several authors have investigated various versions of \concept{isostable coordinates} without restricting attention to the ``slowest'' isostable coordinate.
		The authors in \cite[Eq.~5]{wilson2016isostable} defined a ``finite-time'' approximate version of \concept{isostable coordinates} which provide an approximation of our principal eigenfunctions.
		Subsequently, \cite[Sec.~2]{shirasaka2017phase} defined a version of ``exact'' isostable coordinates (termed \concept{amplitudes} and \concept{phases}) directly in terms of Koopman eigenfunctions, and in particular our Proposition \ref{prop:koopman-cka-per} and Theorem \ref{th:classify-per} can be used to directly infer existence, uniqueness and regularity properties of these coordinates under relatively weak assumptions.
		It appears that \cite{wilson2018greater,monga2019phase} intended to define a different version of ``exact'' isostable coordinates close in spirit to the approximate version in \cite{wilson2016isostable}.
		However, these definitions \cite[Eq.~24,~Eq.~58]{wilson2018greater, monga2019phase} are given in terms of a limit which might not exist for principal eigenfunctions other than the ``slowest'', as we show in Example \ref{ex:koop-converge} below.
		In any case, it appears that principal Koopman eigenfunctions provide a means for defining all of the isostable coordinates for a periodic orbit attractor which does not require such limits.
	\end{Rem}   	
   	
   	\begin{Rem}[Relationship to the principal eigenfunctions, principal algebras, and pullback algebras of \cite{mohr2016koopman}]\label{rem:mohr-mezic}
   	  	Given a nonlinear dynamical system $\Phi\colon Q\times \Grp \to Q$ with $Q$ the basin of an attracting hyperbolic fixed point $x_0$, Mohr and Mezi\'{c} defined (in our notation) the principal eigenfunctions for the associated linearization $\D_{x_0}\Phi\colon \T_{x_0}Q\times \Grp\to \T_{x_0}Q$ to be those of the form $v\mapsto w(v)$, where $w\colon \T_{x_0} Q\to \C$ is a left eigenvector of $\D_{x_0}\Phi^1$ \cite[Def.~2.2]{mohr2016koopman}, and they defined the principal algebra $\A_{\D_{x_0}\Phi^1}$ to be the subalgebra of $C^0(\T_{x_0}Q,\C)$ generated by the principal eigenfunctions \cite[Def.~2.3]{mohr2016koopman}.
   	  	Mohr and Mezi\'{c} do not define principal eigenfunctions or the principal algebra for the nonlinear system itself but, given a topological conjugacy $\tau\colon \T_{x_0}Q \to Q$ between $\Phi$ and $\D_{x_0}\Phi$, they define the \concept{pullback algebra}
   	  	\begin{equation}\label{eq:pullback-algebra}
   	  	\left(\A_{\D_{x_0}\Phi^1}\right) \circ \tau^{-1} \coloneqq \{\varphi \circ \tau^{-1}\colon \varphi \in \A_{\D_{x_0}\Phi^1}\}.
   	  	\end{equation}
   	  	Assuming that $\Phi\in \CLKA$, Proposition \ref{prop:koopman-cka-fix} implies that the relationship between the concepts in our Definition \ref{def:principal-eigenfunction} and those of \cite{mohr2016koopman} is as follows.
   	  	If $(\D_{x_0}\Phi^1,\D_{x_0}\Phi^1)$ is $k$-nonresonant and $\nu(\D_{x_0}\Phi^1,\D_{x_0}\Phi^1)< k + \alpha$ (see Definition \ref{def:nonres-spread}), our principal eigenfunctions for $\D_{x_0}\Phi^1$ coincide precisely with their principal eigenfunctions $w\colon \T_{x_0}Q\to\C$.
   	  	This implies that our principal algebra $\A^{k,\alpha}_{\D_{x_0}\Phi^1}$ coincides with their $\A_{\D_{x_0}\Phi^1}$.
   	  	Next, notice that the pullback algebra \eqref{eq:pullback-algebra} is generated by the functions $w \circ \tau^{-1}$ where $w\colon \T_{x_0}Q\to \C$ is a principal eigenfunction of the linearization. 
   	  	If we further assume that the conjugacy $\tau$ is a $\CLKA$ diffeomorphism, then the chain rule implies that each $w \circ \tau^{-1}$ is a $\CLKA$ principal eigenfunction for $\Phi$, and therefore $\left(\A_{\D_{x_0}\Phi^1}\right)\circ \tau^{-1} = \A^{k,\alpha}_{\Phi}$ by Proposition \ref{prop:koopman-cka-fix}.
   	  	In particular, under the above hypotheses it follows that $\left(\A_{\D_{x_0}\Phi^1}\right)\circ \tau^{-1}$ is independent of $\tau$ and generated by at most $n$ $\CLKA$ principal eigenfunctions for $\Phi$. 
   	  	This is perhaps surprising since \eqref{eq:pullback-algebra} depends on the a priori non-unique conjugacy $\tau$; here the assumption that $\tau$ is a $\CLKA$ diffeomorphism is essential. 
   	\end{Rem}
    
    For an attracting hyperbolic $\tau$-periodic orbit of a $\CLKA$ flow with image $\Gamma$ and basin $Q\supset \Gamma$, stable manifold theory \cite{fenichel1974asymptotic,fenichel1977asymptotic,hirsch1977,isochrons,ruelle1989elements} can be used to show that the global strong stable manifold (isochron) $\Ws_{x_0}$ through $x_0\in \Gamma$ is a $\CLKA$ submanifold of $Q$, and $\Phi^\tau(\Ws_{x_0}) = \Ws_{x_0}$.
    Furthermore, any eigenfunction $\varphi\in \CLKA(\Ws_{x_0},\C)$ of $\Phi^\tau|_{\Ws_{x_0}}$ satisfying $\varphi \circ \Phi^\tau|_{\Ws_{x_0}} = e^{\mu \tau}\varphi$ admits the unique extension to an eigenfunction $\psi\in \CLKA(Q,\C)$ with exponent $\mu$ given by $$\psi|_{\Ws_{\Phi^{t}(x_0)}}\coloneqq e^{\mu t} \varphi \circ \Phi^{-t}|_{\Ws_{\Phi^t(x_0)}}$$
    for all $t\in \R$.
    That $\psi\in \CLKA$ follows from considering locally-defined $\CLKA$ ``time-to-impact $\Ws_{x_0}$'' functions as in the proof of Proposition \ref{prop:floq-norm-form}.
    This observation combined with Propositions \ref{prop:p-eig-evec-cycle} and \ref{prop:koopman-cka-fix} yields the following result.
    (Alternatively, the statement concerning existence and uniqueness of principal eigenfunctions follows from Theorem \ref{th:main-thm-per} and Proposition \ref{prop:p-eig-evec-cycle}.)

   	\begin{Prop}[Existence and uniqueness of $\CLKA$ Koopman eigenvalues and principal eigenfunctions for a limit cycle attractor]\label{prop:koopman-cka-per}
 		Fix $k\in \N_{\geq 1}\cup \{+\infty\}$ and $0 \leq \alpha \leq 1$. 
 		Let $\Phi\colon Q \times \R \to Q$ be a $\CLKA$ flow with $Q$ the basin of an attracting hyperbolic $\tau$-periodic orbit with image $\Gamma \subset Q$, where $Q$ is a smooth manifold with $\dim(Q)\geq 2$.
        Fix $x_0\in \Gamma$ and let $E^s_{x_0}=\T_{x_0}\Ws_{x_0}$ denote the unique $\D_{x_0}\Phi^\tau$-invariant subspace complementary to $\T_{x_0} \Gamma$.   
        Assume that the spectral radius $\rho\left(\D_{x_0}\Phi^\tau|_{E^s_{x_0}}\right)\in (0,1)$ satisfies 
   		$$|e^{\mu\tau}|>\left(\rho\left(\D_{x_0}\Phi^\tau|_{E^s_{x_0}}\right)\right)^{k+\alpha}$$
   		in all of the following statements.	        

   		\emph{\textbf{Uniqueness of Koopman eigenvalues.}} 
   		Let $\psi_1\in \CLKA(Q,\C)$ be any Koopman eigenfunction satisfying \eqref{eq:koopman-efunc} with exponent $\mu \in \C$ and $\Grp = \R$.
   		Then there exists $m = (m_1,\ldots,m_n)\in \N_{\geq 0}^n$ such that
   		$$e^{\mu \tau} =  e^{(m\cdot \lambda)\tau},$$ where  $e^{\lambda_1\tau},\ldots,e^{\lambda_n\tau}$ are the eigenvalues of $\D_{x_0} \Phi^\tau|_{E^s_{x_0}}$ repeated with multiplicities and $\lambda \coloneqq (\lambda_1,\ldots, \lambda_n)$.
   		 
   	\emph{\textbf{Existence and uniqueness of principal eigenfunctions.}} Assume that $(e^{\mu\tau}, \D_{x_0}\Phi^\tau|_{E^s_{x_0}})$ is $k$-nonresonant.
   	Let $w\colon E^s_{x_0} \to \C$ be any left eigenvector of $\D_{x_0}\Phi^\tau|_{E^s_{x_0}}$ with eigenvalue $e^{\mu \tau}$.
   	Then there exists a unique principal eigenfunction $\psi\in \CLKA(Q,\C)$ for $\Phi$ satisfying \eqref{eq:koopman-efunc} with exponent $\mu$ and $\Grp = \R$ and satisfying $\D_{x_0}\psi|_{E^s_{x_0}}= w$. Additionally, if $\mu$ and $w$ are real, then $\psi\colon Q\to \R \subset \C$ is real.

   	\end{Prop}
    A well-known example of Sternberg shows that, even for an analytic diffeomorphism $\Phi^1$ of the plane having the globally attracting fixed point $0$, there need not exist a $C^2$ principal eigenfunction corresponding to $e^{\mu}\in \textnormal{spec}(\D_0\Phi^1)$ if $(e^{\mu},\D_0\Phi^1)$ is not $2$-nonresonant \cite[p.~812]{sternberg1957local}.
    Concentrating now on the issue of uniqueness of principal eigenfunctions, the following example shows that our nonresonance and spectral spread conditions are both necessary for the uniqueness statements of Propositions \ref{prop:koopman-cka-fix} and \ref{prop:koopman-cka-per} (hence also for the uniqueness statements of Theorems \ref{th:main-thm} and \ref{th:main-thm-per}). 
	\begin{Ex}[Uniqueness of principal eigenfunctions]\label{ex:thm-sharp}
		Consider $\Phi = (\Phi_1,\Phi_2)\colon \R^2\times \Grp \to \R^2$ defined by
		\begin{equation}\label{eq:example-theorem-sharp-maps}
		\begin{split}
		\Phi^t_1(x,y) &= e^{-t} x\\
		\Phi^t_2(x,y) &= e^{-(k+\alpha)t}y
		\end{split}
		\end{equation}
		where $k \in \N_{\geq 1}$, $0\leq \alpha \leq 1$, and either $\Grp = \Z$ or $\Grp = \R$.
		$\Phi$ is a diagonal linear dynamical system with $x_0=0$ a globally exponentially stable fixed point, and the eigenvalues of $\D_0 \Phi^1=\Phi^1$ are $e^{-1}$ and $e^{-(k+\alpha)}$.
		Furthermore, for any irrational $\alpha\in [0,1]$, $(e^{-(k+\alpha)},\D_0\Phi^1)$ is $\infty$-nonresonant; $\infty$-nonresonance also holds if, e.g., $k=1$ and $\alpha = 0$ (see Definition \ref{def:nonres}).
		However, if we define $\sigma_\alpha(x)\coloneqq |x|$ for $\alpha > 0$ and $\sigma_{\alpha}(x)\coloneqq x$ for $\alpha = 0$, then for any $k\in\N_{\geq 1}$ and $\alpha \in [0,1]$ both
		\begin{equation}\label{eq:ex-h1}
		h_1(x,y)\coloneqq y
		\end{equation}
		and 
		\begin{equation}\label{eq:ex-h2}
		h_2(x,y)\coloneqq y + \sigma_{\alpha}(x)^{k+\alpha}
		\end{equation}		
		are $C^{k,\alpha}$ principal eigenfunctions satisfying \eqref{eq:koopman-efunc} with the same exponent $$\mu = -(k+\alpha).$$
		In particular, this shows that $\CL^{k,\alpha}$ principal Koopman eigenfunctions are not necessarily unique (modulo scalar multiplication)  even if the  $\infty$-nonresonance condition is satisfied.
		Since here $h_2 \in \CLKA(\R^2,\C)$ and $r=(k+\alpha)$ is the smallest $r\in \R$ satisfying $e^{-(k+\alpha)}\geq\left(\rho\left(\D_{0}\Phi^1\right)\right)^{r}=(e^{-1})^r$, this shows that the spectral spread condition $|e^\mu|>\left(\rho\left(\D_{x_0}\Phi^1\right)\right)^{k+\alpha}$ is both necessary and sharp for the principal eigenfunction uniqueness statement of Proposition \ref{prop:koopman-cka-fix} to hold, at least in the case that $\alpha > 0$.\footnote{In the terminology and notation of Theorem \ref{th:main-thm} with $m=1$, the condition $|e^\mu|>\left(\rho\left(\D_{0}\Phi^1\right)\right)^{k+\alpha}$ is equivalent to the condition $\nu(e^{\mu},\D_0\Phi^1) < k + \alpha$, where the spectral spread $\nu(\slot,\slot)$ is defined in Definition \ref{def:nonres-spread}. See Remark \ref{rem:koopman-spectral-spread}.} 
		(Note that Proposition \ref{prop:koopman-cka-fix} \emph{does} imply that $\CL^{k',\alpha'}$ principal eigenfunctions are unique for any $k'+\alpha' > k+\alpha$.)
		If instead $k = 1$ and $\alpha = 0$, then $h_1$ and $h_2$ are both analytic eigenfunctions satisfying \eqref{eq:koopman-efunc} with the same exponent $\mu = -1,$ while $(e^{-1},\D_0\Phi^1)$ is $\infty$-nonresonant, but now these eigenfunctions are distinguished by their derivatives at the origin; this is consistent with the uniqueness statement of  Proposition \ref{prop:koopman-cka-fix}. 
		On the other hand, if $k = 2$ and $\alpha = 0$ so that $(e^{-2},\D_0\Phi^1)$ is not $2$-nonresonant, \eqref{eq:ex-h1} and \eqref{eq:ex-h2} show that analytic eigenfunctions are not unique despite the fact that the spectral spread condition $|e^{-2}|>\left(\rho\left(\D_{0}\Phi^1\right)\right)^{+\infty}=0$ certainly holds.
		Hence the nonresonance condition is also necessary for the principal eigenfunction uniqueness statement of Proposition \ref{prop:koopman-cka-fix} to hold.
		Finally, by taking $\Grp = \R$, changing the state space $\R^2$ above to $\R^2 \times S^1$, and prescribing the unit circle $S^1\subset \C$ with the decoupled dynamics $\Phi_3^t(x,y,\theta)\coloneqq e^{it}\theta$ (where $i=\sqrt{-1})$ yields an example showing that the spectral spread and nonresonance conditions are both necessary for the uniqueness statement in Proposition \ref{prop:koopman-cka-per} to hold as well.
		
		\begin{Ex}[Existence of the limit \eqref{eq:Koop-main-converge} and isostable coordinates]\label{ex:koop-converge}
		Existence of the limit in \eqref{eq:Koop-main-converge} is not automatic if the ``approximate eigenfunction'' $P$ is not an approximation to sufficiently high order. 
		To demonstrate this, fix $k \in \N_{\geq 1}$, $\alpha \in [0,1]$, $r\in \R_{\geq 1}$, and $\epsilon \in \R_{>0}$.
		Define $\sigma_\alpha(x)\coloneqq |x|$ for $\alpha > 0$ and $\sigma_{\alpha}(x)\coloneqq x$ for $\alpha = 0$, and consider the $\CLKA$ dynamical system $\Phi\colon \R^2\times \Grp \to \R^2$ defined by
		\begin{equation}
		\begin{split}
		\Phi_1^t(x,y)&= e^{-t}x\\
		\Phi_2^t(x,y)&= e^{-rt}(y-\epsilon \sigma_{\alpha}(x)^{k + \alpha}) + \epsilon e^{-(k + \alpha)t} \sigma_{\alpha}(x)^{k+\alpha},
		\end{split}
		\end{equation}
		where either $\Grp = \Z$ or $\Grp = \R$.
		To see that $\Phi$ is indeed a dynamical system (i.e., that $\Phi$ satisfies the group property $\Phi^{t+s}=\Phi^t\circ \Phi^s$), define the diagonal linear system $\widetilde{\Phi}^t(x,y) = (e^{-t} x, e^{-rt}y)$ and the $C^{k,\alpha}$ diffeomorphism $H\colon \R^2\to \R^2$ via $H(x,y) \coloneqq (x, y + \epsilon \sigma_{\alpha}(x)^{k + \alpha})$, and note that $\Phi^t = H \circ \widetilde{\Phi}^t \circ H^{-1}$.
		In other words, $\Phi$ is obtained from a diagonal linear dynamical system via a $C^{k,\alpha}$ change of coordinates; note also that this change of coordinates can be made arbitrarily close to the identity by taking $\epsilon$ arbitrarily small.
		Since $x_0=0$ is a globally exponentially stable fixed point for $\widetilde{\Phi}$, it is also so for $\Phi$.
		We note that $r_0=r\geq 1$ is the smallest $r_0\in \R$ such that the spectral radius $\rho\left(\D_{0}\Phi^1\right)=e^{-1}\in (0,1)$ satisfies $|e^{-r}|\geq\left(\rho\left(\D_{0}\Phi^1\right)\right)^{r_0}$.\footnote{That is, the spectral spread satisfies $\nu(e^{-r},\D_0 \Phi^1) = r$ in the terminology of Definition \ref{def:nonres-spread} and Theorem \ref{th:main-thm}. See Remark \ref{rem:koopman-spectral-spread}.}
		We further note that, for any choice of $\epsilon$, the analytic function $P(x,y)\coloneqq y$ satisfies $$P \circ \Phi^1 = e^{-r} P + R$$
		where $\D^j_{(0,0)}R = 0$ for all integers $0\leq j < k+\alpha$.
		However, 
		\begin{equation}
		\label{eq:ex-diverge}
		\begin{split}
		\lim_{t\to\infty}e^{rt} P \circ \Phi^t(x,y) &= y- \epsilon \sigma_{\alpha}(x)^{k + \alpha}  + \epsilon \sigma_{\alpha}(x)^{k + \alpha} \lim_{t\to\infty} e^{(r-(k + \alpha))t}\\
		&= 
		\begin{cases} 
		y-\epsilon \sigma_{\alpha}(x)^{k + \alpha} & 1\leq  r < k + \alpha \\
		y & r = k + \alpha\\
		+\infty & r > k + \alpha
		\end{cases}
		\end{split}
		\end{equation}
		for any $x \neq 0$ and $\epsilon > 0$.
		We see that the limit \eqref{eq:ex-diverge} diverges when $r > k+\alpha$ (so that $|e^{-r}|<\left(\rho\left(\D_{0}\Phi^1\right)\right)^{k+\alpha}$), but the limit converges when $r\leq k + \alpha$ (so that $|e^{-r}|\geq\left(\rho\left(\D_{0}\Phi^1\right)\right)^{k+\alpha}$). 
		For the case that $r<k+\alpha$ and $r\not \in \N_{\geq 2}$, this is consistent with Proposition \ref{prop:koopman-cka-fix} which guarantees that the limit converges if $\Phi\in \CLKA$, if $|e^{-r}|>\left(\rho\left(\D_{0}\Phi^1\right)\right)^{k+\alpha}$, and if  $(e^{-r},\D_0 \Phi^1)$ is $k$-nonresonant.
		When $r = k + \alpha$ and $r\not \in \N_{\geq 2}$, convergence is also guaranteed by Proposition \ref{prop:koopman-cka-fix} for this specific example, because then (i) $(e^{-r},\D_0 \Phi^1)$ is  $\infty$-nonresonant, (ii) $\Phi$ is linear (the nonlinear terms cancel when $r=k+\alpha$) and hence $C^\infty$, and (iii)  Proposition \ref{prop:koopman-cka-fix} guarantees that this limit always exists if $\Phi\in C^\infty$ and  $(e^{-r},\D_0 \Phi^1)$ is  $\infty$-nonresonant because the spectral spread condition $|e^{-r}|>\left(\rho\left(\D_{0}\Phi^1\right)\right)^{+\infty}=0$ always holds.
		As alluded to in Remark \ref{rem:no-nonres-needed}, the preceding reasoning can actually be applied even without the assumption that $r$ is not an integer if Lemma \ref{lem-make-approx-exact} is used instead of Proposition \ref{prop:koopman-cka-fix} as the tool of inference (i.e., nothing about nonresonance actually needs to be assumed for this example).
		We emphasize that the divergence in \eqref{eq:ex-diverge} is associated purely with the spectral spread condition since, e.g., we can choose $r \geq 1$ so that $(e^{-r},\D_0 \Phi^1)$ is $\infty$-nonresonant and take $\alpha = 0$ so that $\Phi$ is analytic.		
		
		Note that by taking $\Grp = \R$, changing the state space $\R^2$ to $\R^2 \times S^1$, and prescribing the unit circle $S^1\subset \C$ with the decoupled dynamics $\Phi_3^t(x,y,\theta)\coloneqq e^{it}\theta$ (where $i=\sqrt{-1})$ yields a corresponding example with a globally attracting hyperbolic periodic orbit $\{(0,0)\}\times S^1$.
		In this case, for this example \cite[Eq.~24,~Eq.~58]{wilson2018greater, monga2019phase}
		would attempt to \emph{define} the ``faster'' isostable coordinate (principal eigenfunction in our terminology) $\psi_2$ satisfying \eqref{eq:koopman-efunc} with exponent $\mu_2 \coloneqq -r$ via the limit \eqref{eq:ex-diverge}, but \eqref{eq:ex-diverge} shows that this limit does not exist if $r > k + \alpha$.
        This phenomenon should be compared with the explanation in the preceding paragraph based on our general results. 
		\end{Ex}

	\end{Ex}

	\subsection{Classification of all $C^\infty$ Koopman eigenfunctions}\label{sec:app-classify}
    
    \begin{Not}
    To improve the readability of Theorems \ref{th:classify-point} and \ref{th:classify-per} below, we introduce the following multi-index notation.    	
    We define an $n$-dimensional multi-index to be an $n$-tuple $i = (i_1,\ldots, i_n)\in \N^n_{\geq 0}$ of nonnegative integers, and define its sum to be $|i|\coloneqq i_1+\cdots  + i_n.$ 
    For a multi-index $i\in \N^n_{\geq 0}$ and $z = (z_1,\ldots, z_n)\in \C^n$, we define $z^{[i]}\coloneqq z_1^{i_1}\cdots z_n^{i_n}.$
    Given a $\C^n$-valued function $\psi = (\psi_1,\ldots, \psi_n)\colon Q\to \C^n$, we define $\psi^{[i]}\colon Q\to \C$ via $\psi^{[i]}(x)\coloneqq (\psi(x))^{[i]}$ for all $x\in Q$.
    We also define the complex conjugate of $\psi = (\psi_1,\ldots, \psi_n)$ element-wise: $\bar{\psi}\coloneqq (\bar{\psi}_1,\ldots, \bar{\psi}_n).$
    \end{Not}
	
   	For the case $\Grp = \R$, the $\infty$-nonresonance assumption in the following result can be replaced by the slightly less restrictive nonresonance condition in Remark \ref{rem:weaker-nonresonance-real-time}.
   	Applications-oriented readers may find it helpful to consult Example \ref{ex:intuition} with $m=1$ for relevant remarks.

    \begin{Th}[Classification of all $C^\infty$ eigenfunctions for a point attractor]\label{th:classify-point}
   		Let $\Phi\colon Q \times \Grp \to Q$ be a $C^{\infty}$ dynamical system with $Q$ the basin of an attracting hyperbolic fixed point $x_0\in Q$,  where $Q$ is a smooth manifold with $\dim(Q)\geq 1$ and either $\Grp = \Z$ or $\Grp = \R$.    	
    	Assume that $\D_{x_0}\Phi^1$ is semisimple (diagonalizable over $\C$) and that $(\D_{x_0}\Phi^1,\D_{x_0}\Phi^1)$ is $\infty$-nonresonant.
    	
    	Letting $n = \dim(Q)$, it follows that there exists an $n$-tuple $$\psi = (\psi_1,\ldots,\psi_n)$$ of $C^\infty$ principal eigenfunctions such that every $C^\infty$ Koopman eigenfunction $\varphi$ is a finite sum of scalar multiples of products of the $\psi_j$ and their complex conjugates $\bar{\psi}_j$:
    	\begin{equation}\label{eq:th-classify-expansion}
    	\varphi = \sum_{|\ell|+|m| \leq k}c_{\ell,m}\psi^{[\ell]}\bar{\psi}^{[m]}
    	\end{equation}
    	for some $k \in \N_{\geq 0}$ and some coefficients $c_{\ell,m}\in \C$. 
    \end{Th}	    
    \begin{proof}
    	By Proposition \ref{prop:sternberg} and linear algebra, there exists a $C^\infty$ embedding $Q\hookrightarrow \C^n$ which maps $Q$ diffeomorphically onto an $\R$-linear subspace of $\C^n$, maps $x_0$ to $0$, and semiconjugates $\Phi$ to the diagonal $\C$-linear dynamical system $\Theta^t(z_1,\ldots,z_n)=(e^{\lambda_1 t} z_1,\ldots, e^{\lambda_n t} z_n)$.\footnote{The standard definition of $C^\infty$ embedding is recalled preceding Proposition \ref{prop:floq-norm-form} (take $\alpha = 0$ and $k = +\infty$).}
    	Thus, for simplicity, we may (and do) view $Q$ as a $\Theta$-invariant $\R$-linear subspace of $\C^n$ with $\Phi = \Theta|_{Q\times \Grp}$.    	
    	Let $\varphi \in C^\infty(Q,\C)$ be any $C^\infty$ Koopman eigenfunction satisfying \eqref{eq:koopman-efunc} with exponent $\mu\in \C$.  	    

        Write $z = (z_1,\ldots,z_n)\in \C^n$.
    	For any $k\in \N_{\geq 0}$, Taylor's theorem implies the existence of $R_k\in C^\infty(Q,\C)$ satisfying $0=R_k(0) = \D_0 R_k = \cdots = \D_0^k R_k$ and coefficients $c_{\ell,m}$ such that
    	\begin{equation}\label{eq:classify-point-expansion}
    	\begin{split}
    	\varphi(z) &= \sum_{|\ell|+|m|\leq k} c_{\ell,m}z^{[\ell]}\bar{z}^{[m]}+ R_k(z)
    	\end{split}
    	\end{equation}
    	for all $z\in Q$.
    	Defining $\lambda\coloneqq (\lambda_1,\ldots, \lambda_n)$ and writing the eigenfunction equation $\varphi \circ \Theta^1 = e^{\mu} \varphi$ in terms of the expansion \eqref{eq:classify-point-expansion} yields, for all $z\in Q$, 
    	\begin{align*}
    	 \sum_{|\ell| + |m| \leq k} e^{\ell\cdot \lambda + m\cdot \bar{\lambda}}c_{\ell,m}z^{[\ell]}\bar{z}^{[m]} + R_k\circ \Phi^1(z) 
    	 = \sum_{|\ell| + |m| \leq k} e^{\mu} \left(c_{\ell,m}z^{[\ell]}\bar{z}^{[m]} + R_k(z)\right).
    	\end{align*}
    	Upon subtracting terms, we see that $S_k\coloneqq R_k\circ \Phi^1 - e^{\mu} R_k$ is the restriction of a $k$-th order polynomial to $Q$.
    	On the other hand, we infer that $0=S_k(0) = \D_0 S_k = \cdots = \D_0^k S_k$ from the corresponding property of $R_k$.
    	Thus, $S_k$ is the zero function, so $R_k\circ \Phi^1 = e^{\mu}R_k$.
    	Recall Definition \ref{def:nonres-spread} of the (always finite) spectral spread $\nu(\slot,\slot)$.
    	If we choose $k$ sufficiently large so that $k > \nu(e^{\mu},\D_0 \Phi^1)$, Proposition \ref{prop:uniqueness-without-nonresonance} implies that $R_k\equiv 0$.
    	From \eqref{eq:classify-point-expansion}, we then see that $\varphi$ is equal to a sum of products of the principal eigenfunctions $\psi_j(z)\coloneqq z_j$, $\bar{\psi}_j(z) = \bar{z}_j$ as desired.    	
    \end{proof}	
    
    For a an attracting hyperbolic $\tau$-periodic orbit of a $\CLKA$ flow with image $\Gamma$ and basin $Q\supset \Gamma$, let $\Ws_{x_0}$ be the global strong stable manifold (isochron) through the point $x_0\in \Gamma$.
    As discussed in the proof of Proposition \ref{prop:floq-norm-form}, there is a unique (modulo scalar multiplication) continuous eigenfunction satisfying \eqref{eq:koopman-efunc} with exponent $\mu = i\frac{2\pi}{\tau}$ and $\Grp = \R$, where $i=\sqrt{-1}$, and this eigenfunction is in fact $C^\infty$ for a $C^\infty$ flow.  
    Moreover, such an eigenfunction is constant on $\Ws_{x_0}$.
    In the theorem below, let $\psi_\theta$ be the unique such  eigenfunction satisfying $\psi_\theta|_{\Ws_{x_0}}\equiv 1$, where the point $x_0$ is as in the theorem statement.
    Explicitly, $\psi_\theta$ is given by
    $$\psi_\theta|_{\Ws_{\Phi^t(x_0)}} = e^{i\frac{2\pi}{\tau} t}$$
    for all $t\in \R$.
    This defines $\psi_\theta$ on all of $Q$ since $Q = \bigcup_{t\in \R}\Ws_{\Phi^t(x_0)}$, and the definition makes sense since $\Ws_{\Phi^{j\tau}(x_0)} = \Ws_{x_0}$ for all $j\in \Z$.
    \begin{Th}[Classification of all $C^\infty$ eigenfunctions for a limit cycle attractor]\label{th:classify-per}
    	Let $\Phi\colon Q \times \R \to Q$ be a $C^{\infty}$ dynamical system with $Q$ the basin of an attracting hyperbolic $\tau$-periodic orbit with image $\Gamma\subset  Q$, where $Q$ is a smooth manifold with $\dim(Q)\geq 2$.    	
    	Fix $x_0 \in \Gamma$ and denote by $E^s_{x_0}$ the unique $\tau$-invariant subspace complementary to $\T_{x_0}\Gamma$.
    	Assume that $\D_{x_0}\Phi^\tau|_{E^s_{x_0}}$ is semisimple and that $(\D_{x_0}\Phi^\tau|_{E^s_{x_0}},\D_{x_0}\Phi^\tau|_{E^s_{x_0}})$ is $\infty$-nonresonant.
    	
    	Letting $n + 1= \dim(Q)$, it follows that there exists an $n$-tuple $$\psi = (\psi_1,\ldots,\psi_n)$$ of $C^\infty$ principal eigenfunctions such that every $C^\infty$ Koopman eigenfunction $\varphi$ is a finite sum of scalar multiples of products of integer powers of $\psi_\theta$ with products of the $\psi_j$ and their complex conjugates $\bar{\psi}_j$:
    	\begin{equation}
    	\varphi = \sum_{|\ell|+|m|\leq k}c_{\ell,m}\psi^{[\ell]}\bar{\psi}^{[m]} \psi_\theta^{j_{\ell,m}}
    	\end{equation}
    	for some $k\in \N_{\geq 0}$, some coefficients $c_{\ell,m} \in \C$, and $j_{\ell,m} \in \Z$.
    \end{Th}	        
    \begin{proof}
    	Let $\Ws_{x_0}$ be the $C^\infty$ global strong stable manifold through $x_0$.
    	We remind the reader of the facts $Q = \bigcup_{t\in \R}\Ws_{\Phi^t(x_0)}$ and $\Ws_{\Phi^t(x_0)} = \Phi^t(\Ws_{x_0})$ which are implicitly used in the remainder of the proof.
    	
    	First, we note that every eigenfunction $\chi\in C^\infty(\Ws_{x_0},\C)$ of $F^j(x)\coloneqq \Phi^{j\tau}|_{\Ws_{x_0}}(x)$ satisfying \eqref{eq:koopman-efunc} with exponent $\mu\in \C$ and $\Grp = \Z$ admits a unique extension to an eigenfunction $\tilde{\chi}\in C^\infty(Q,\C)$ of $\Phi$ satisfying \eqref{eq:koopman-efunc} with exponent $\mu$ and $\Grp = \R$; this unique extension $\tilde{\chi}$ is defined via 
    	\begin{align}\label{eq:th-classify-per-extension-formula}
        \tilde{\chi}|_{\Ws_{\Phi^{-t}(x_0)}} = e^{-\mu t} \chi \circ \Phi^t|_{\Ws_{\Phi^{-t}(x_0)}}   \end{align}
        for all $t\in \R$.
        That $\tilde{\chi}\in C^\infty$ follows from considering locally-defined $C^\infty$ ``time-to-impact $\Ws_{x_0}$'' functions as in the proof of Proposition \ref{prop:floq-norm-form}, and $\chi$ is a principal eigenfunction if and only if its extension $\tilde{\chi}$ is.
    	
    	Next, let $\varphi\in C^\infty(Q,\C)$ be a eigenfunction satisfying \eqref{eq:koopman-efunc} with exponent $\mu$ and $\Grp = \R$.
    	Theorem \ref{th:classify-point} implies that $\varphi|_{\Ws_{x_0}}$ is equal to a sum of products of principal eigenfunctions $\chi_1,\ldots,\chi_n,\bar{\chi}_1,\ldots, \bar{\chi}_n$ of $\Phi^\tau|_{\Ws_{x_0}}$ of the form:
    	\begin{equation}\label{eq:classify-per-expand-1}
    	\varphi|_{\Ws_{x_0}} = \sum_{|\ell|+|m| \leq k}c_{\ell,m}\chi^{[\ell]}\bar{\chi}^{[m]} 
    	\end{equation}   
    	for some $k\in \N_{\geq 0}$, where $\chi = (\chi_1,\ldots,\chi_n)$. 	
    	Let $\lambda = (\lambda_1, \ldots \lambda_n) \in \C^n$ be such that each $\chi_j$ satisfies $\chi_j \circ \Phi^\tau|_{\Ws_{x_0}} = e^{\lambda_j \tau}\chi_j$. 	
    	Since $\varphi$ satisfies \eqref{eq:koopman-efunc} with exponent $\mu$, it follows that 
    	\begin{equation}\label{eq:th-classify-per-exp-equal}
    	e^{\mu \tau} = e^{(\ell\cdot \lambda + m\cdot \bar{\lambda})\tau}
    	\end{equation}
    	for all $\ell,m\in \N^{n}_{\geq 0}$ such that $c_{\ell,m} \neq 0$, so for such $\ell,m$ we have 
    	\begin{equation}\label{eq:th-classify-per-log-equal}
    	\mu = \ell\cdot \lambda + m\cdot \bar{\lambda}+ i \frac{2\pi}{\tau} j_{\ell,m}
    	\end{equation}  
    	for some $j_{\ell,m} \in \Z$, where $i=\sqrt{-1}$.  	
    	By the previous paragraph, we may uniquely write $\chi = \psi|_{\Ws_{x_0}} = (\psi_1|_{\Ws_{x_0}},\ldots,\psi_n|_{\Ws_{x_0}})$ for principal eigenfunctions $\psi_j$ of $\Phi$ satisfying \eqref{eq:koopman-efunc} with exponent $\lambda_j$ and $\Grp = \R$.
    	
    	Using \eqref{eq:classify-per-expand-1},\eqref{eq:th-classify-per-log-equal}, and the extension formula \eqref{eq:th-classify-per-extension-formula}, we obtain
    	\begin{align*}
    	\varphi|_{\Ws_{\Phi^{-t}(x_0)}}&=  \sum_{|\ell|+|m| \leq k}  c_{\ell,m}e^{-\mu t}\cdot \left(\chi^{[\ell]}\bar{\chi}^{[m]}\right) \circ \Phi^t|_{\Ws_{\Phi^{-t}(x_0)}}\\
    	&= \sum_{|\ell|+ |m|\leq k}  c_{\ell,m}e^{-\mu t}\cdot \left(\psi^{[\ell]}\bar{\psi}^{[m]}\right)|_{\Ws_{x_0}} \circ \Phi^t|_{\Ws_{\Phi^{-t}(x_0)}}\\
    	&= \sum_{|\ell|+ |m|\leq k} c_{\ell,m} e^{-(i\frac{2\pi}{\tau}j_{\ell,m})t}\cdot \left(\psi^{[\ell]}\bar{\psi}^{[m]}\right)|_{\Ws_{\Phi^{-t}(x_0)}}\\
    	&= \sum_{|\ell|+|m|\leq k} c_{\ell,m}  \left(\psi^{[\ell]}\bar{\psi}^{[m]}\psi_\theta^{j_{\ell,m}}\right)|_{\Ws_{\Phi^{-t}(x_0)}}
    	\end{align*} 
    	for all $t\in \R$ as desired.
    	To obtain the last equality we used the fact that $\psi_\theta|_{\Ws_{x_0}}\equiv 1$, so the extension formula \eqref{eq:th-classify-per-extension-formula} implies that $\psi_\theta|_{\Ws_{\Phi^{-t}(x_0)}}\equiv e^{-i\frac{2\pi}{\tau}t}$ and hence also $\left(\psi_\theta^{j_{\ell,m}}\right)|_{\Ws_{\Phi^{-t}(x_0)}}\equiv e^{-(i\frac{2\pi}{\tau}j_{\ell,m})t}.$  
        This completes the proof.
    \end{proof}

	\section{Proofs of the main results}\label{sec:proofs-main-results}
	\subsection{Proof of Theorem \ref{th:main-thm}}
	In this section we prove Theorem \ref{th:main-thm}, which we repeat here for convenience.
	\ThmMain*	
	We prove the uniqueness and existence portions of Theorem \ref{th:main-thm} in the following \S  \ref{sec:main-proof-uniq} and \S \ref{sec:main-proof-exist}, respectively.
	Some of our statements and proofs make use of higher-order derivatives of maps between Euclidean spaces \cite[Sec.~A.5]{smoothInvariant} and the fact that a multilinear map is equivalent to a linear map out of a tensor product (the ``universal property of the tensor product'').
		
	\subsubsection{Proof of uniqueness}\label{sec:main-proof-uniq}
	In this section, we prove the uniqueness portion of Theorem \ref{th:main-thm}.
	The proof of uniqueness consists of an algebraic part and an analytic part.
	The algebraic portion is carried out in Lemmas \ref{lem:nonresonance-implies-invertible} and \ref{lem:jets-zero-maps}, and the analytic portion is carried out in Lemma \ref{lem:psi-identically-0}.
	
	\begin{Lem}\label{lem:nonresonance-implies-invertible}
		Let $k \in \N_{\geq 1}\cup \{+\infty\}$, $m,n\in \N_{\geq 1}$, $X\in \C^{m\times m}$, and $Y\in \R^{n \times n}$ be such that $(X,Y)$ is $k$-nonresonant.
        For all  $1 < i < k + 1$, let $\Ll((\R^n)^{\otimes i},\C^m)$ denote the space of linear maps from the $i$-fold tensor product $(\R^n)^{\otimes i}$ to $\C^m$, and define the linear operator
        \begin{equation}
        T_i \colon \Ll((\R^n)^{\otimes i},\C^m) \to \Ll((\R^n)^{\otimes i},\C^m), \qquad T_i(P)\coloneqq P Y^{\otimes i} - X P.
        \end{equation}
        (By this formula we mean that $T_i(P)$ acts on tensors $\tau \in (\R^n)^{\otimes i}$ via $\tau \mapsto P(Y^{\otimes i}(\tau)) - XP(\tau)$.)
                
        Then for all $1 < i < k + 1$, $T_i$ is a linear isomorphism.
        (The conclusion holds vacuously if $k = 1$.)		
	\end{Lem}
	\begin{proof}
    Let $\lambda_1,\ldots, \lambda_n$ and $\mu_1,\ldots, \mu_m$ respectively be the eigenvalues of $Y$ and $X$ repeated with multiplicity.
    
    First assume that $Y$ and $X$ are both semisimple, i.e., diagonalizable over $\C$.
    Identifying $Y$ with its complexification, let $e_1,\ldots, e_n \in \C^n$ be a basis of eigenvectors for $Y$ and let $e^1, \ldots, e^n \in (\C^n)^*$ be the associated dual basis. 
    Let $f_1,\ldots, f_m \in \C^m$ be a basis of eigenvectors for $X$.
    Fix any integer $i$ with $1< i < k + 1$, any $p\in \{1,\ldots, m\}$, and any multi-indices $\ell,j\in \N_{\geq 1}^i$; defining $e^{\otimes[\ell]}\coloneqq e^{\ell_1}\otimes \cdots \otimes e^{\ell_i}$ and similarly for $e_{\otimes [j]}$, we compute
    \begin{align*}
    T_i\left(f_p \otimes e^{\otimes [\ell]}\right)\cdot e_{\otimes [j]} &= \lambda_{j_1}\cdots \lambda_{j_n}\cdot  (e^{\otimes[\ell]}\cdot e_{\otimes[j]})f_p
    - \mu_p\cdot (e^{\otimes[\ell]}\cdot e_{\otimes[j]}) f_p 
    \\&= \delta^\ell_j\cdot \left(\lambda_{j_1}\cdots \lambda_{j_i} - \mu_p\right)f_p 
    \end{align*}
    (no summation implied), where the multi-index Kronecker delta is defined by $\delta^{\ell}_\ell = 1$ and $\delta^\ell_j = 0$ if $\ell \neq j$.
    Hence the $f_p \otimes e^{\otimes[\ell]}$ are eigenvectors of $T_i$ with eigenvalues $\left(\lambda_{\ell_1}\cdots \lambda_{\ell_i} - \mu_p\right)$, and dimension counting implies that these are all of the eigenvector/eigenvalue pairs.
    The $k$-nonresonance assumption implies that none of these eigenvalues are zero, so $T_i$ is invertible if $Y$ and $X$ are both semisimple.
    
    Since the operator $T_i$ depends continuously on the matrices $X$ and $Y$, since eigenvalues of a matrix depend continuously on the matrix, and since semisimple matrices are dense, it follows by continuity that the eigenvalues of $T_i$ are all of the form $\left(\lambda_{\ell_1}\cdots \lambda_{\ell_i} - \mu_p\right)$ even if one or both of $X$ and $Y$ are not semisimple (cf. \cite[p.~37]{nelson1970topics}), and these eigenvalues are all nonzero by the asssumption that $(X,Y)$ is $k$-nonresonant.
    Hence $T_i$ is still invertible in the case of general $X$ and $Y$.		
	\end{proof}
	
	\begin{Lem}\label{lem:jets-zero-maps}
		Let $F\in C^1(\R^n,\R^n)$ have the origin as a fixed point, where $n\geq 1$.
		Let $k \in \N_{\geq 1}\cup \{+\infty\}$, $m\in \N_{\geq 1}$, and $X\in \C^{m\times m}$ be such that $(X,\D_0 F)$ is $k$-nonresonant.
		Assume that $\psi\in C^k(\R^n,\C^m)$ satisfies $\D_0 \psi = 0$ and 
		\begin{equation}\label{eq:lem-uniq-psi-conj}
		\psi \circ F = X \psi.
		\end{equation} 
		Then it follows that $\psi(0) = X\psi(0)$ and
		\begin{equation*}
		  \D^i_0 \psi= 0
		\end{equation*}
		for all $1 < i < k + 1$.   (The conclusion holds vacuously if $k = 1$.)
	\end{Lem}
	
	\begin{Rem}
		We can restate the conclusion of Lemma \ref{lem:jets-zero-maps} in the language of jets \cite{hirsch1976differential,golubitsky1985singularities,smoothInvariant}.
		If $\psi$ is a linearizing factor such that the $1$-jet $j^1_0(\psi - \psi(0)) = 0$, then automatically the $k$-jet $j^k_0(\psi - \psi(0)) = 0$.  	
	\end{Rem}
	
	\begin{proof}
	That $\psi(0) = X\psi(0)$ follows from setting $x = 0$ in \eqref{eq:lem-uniq-psi-conj} and using the assumption that $F(0)=0$.
	We will prove the remaining claim that $\D_0^i\psi = 0$ for $1\leq i< k + 1$ by induction on $i$.
	The base case of the induction, $\D_0^1 \psi = \D_0 \psi = 0$, is one of the hypotheses of the lemma.
	For the inductive step, assume that $\D_0 \psi = \cdots = \D_0^{i}\psi = 0$ for an integer $i$ satisfying $1 \leq i < k$.
	If it were the case that $F\in C^{i+1}$, one way to proceed would be to differentiate \eqref{eq:lem-uniq-psi-conj} $(i+1)$ times and somehow deduce that $\D_{0}^{i+1}\psi=0$.
	However, we are assuming only that $F\in C^1$, so that approach is problematic.
	We instead proceed as follows.
	
	By Taylor's theorem and the inductive hypothesis, we have (here $x^{\otimes \ell}$ denotes the tensor product of $x$ with itself $\ell$ times)
	\begin{equation}\label{eq:F-psi-taylor}
	F(x) = \D_0 F\cdot x + R_F(x), \qquad \psi(x) = \psi(0) + \D_{0}\psi^{i+1}\cdot x^{\otimes (i+1)} + R_\psi(x)
	\end{equation}
	for all $x$, where the remainders satisfy $\lim_{x\to 0}\frac{R_F(x)}{\norm{x}}=0$ and  $\lim_{x\to 0}\frac{R_\psi(x)}{\norm{x}^{i+1}}=0$.
	It follows that 
	\begin{equation}\label{eq:psi-of-F}
	\psi(F(x)) = \psi(0) +  \D_{0}\psi^{i+1}\cdot (\D_0 F \cdot x)^{\otimes (i+1)} + R(x), 
	\end{equation}
	where 
	\begin{equation}\label{eq:R-def}
	R(x) \coloneqq R_\psi(F(x)) + \D_{0}\psi^{i+1}\cdot\sum_{\ell=0}^{i}C_\ell [(\D_0 F\cdot x)^{\otimes \ell}\otimes (R_F(x))^{\otimes (i+1-\ell)}]
	\end{equation}
	for suitable combinatorially determined constants $C_\ell> 0$.
	Rewriting \eqref{eq:lem-uniq-psi-conj} using \eqref{eq:F-psi-taylor} and \eqref{eq:psi-of-F}, we obtain
	\begin{equation}\label{eq:lem-uniq-psi-rewritten}
	 \psi(0) +  \D_{0}\psi^{i+1}\cdot (\D_0 F \cdot x)^{\otimes (i+1)} + R(x) = X\psi(0) + X \D_{0}\psi^{i+1}\cdot x^{\otimes (i+1)} + XR_\psi(x).
	\end{equation}
	In order to deduce the information we need from \eqref{eq:lem-uniq-psi-rewritten}, we now show that $\lim_{x\to 0}\frac{R(x)}{\norm{x}^{i+1}}=0$.
	Using the tensor product property $a^{p+q} v^{\otimes p}\otimes w^{\otimes q} = (av)^{\otimes p}\otimes (aw)^{\otimes q}$ for a scalar $a$ and vectors $v$ and $w$, this follows from \eqref{eq:R-def} and the computation
	\begin{equation}\label{eq:R-is-o-iplus1}
	\begin{split}
	\lim_{x\to 0}\frac{R(x)}{\norm{x}^{i+1}} &=  \lim_{x\to 0}\frac{R_{\psi}(F(x))}{\norm{F(x)}^{i+1}}\frac{\norm{F(x)}^{i+1}}{\norm{x}^{i+1}} + \D_0 \psi^{i+1}\cdot \lim_{x\to 0}\sum_{\ell=1}^{i+1}C_\ell \left[\left(\frac{\D_0 F\cdot x}{\norm{x}}\right)^{\otimes (i+1-\ell)}\otimes \left(\frac{R_F(x)}{\norm{x}}\right)^{\otimes \ell}\right]\\ &= 0.
	\end{split}
	\end{equation}
	The first limit on the right is zero since $F(x)\to 0$ as $x\to 0$ and since $F\in C^1$, so $\norm{F(x)}/\norm{x}\leq \max_{\norm{y}\leq 1}\norm{\D_{y}F}< +\infty$ when $\norm{x}\leq 1$ by the mean value theorem; the second limit on the right is zero since $\norm{\D_0 F\cdot x}/\norm{x}\leq\norm{\D_0 F}< +\infty$.
	
	Let $r> 0$ and $\hat{x}$ be a unit vector.
	Set $x = r \hat{x}$ in  \eqref{eq:lem-uniq-psi-rewritten}.
    By the first sentence of the proof, $\psi(0) = X\psi(0)$.
	Canceling these equal terms from \eqref{eq:lem-uniq-psi-rewritten}, dividing both sides of the resulting equation by $r^{i+1}$, and taking the limit as $r\to 0$ using \eqref{eq:R-is-o-iplus1} yields $\D_0\psi^{i+1} (\D_0 F)^{\otimes (i+1)}\cdot \hat{x}^{\otimes (i+1)} = X\D_0 \psi^{i+1}\cdot \hat{x}^{\otimes (i+1)}$ for all unit vectors $\hat{x}$.
	Since derivatives are symmetric tensors, the latter equation has the form $S\cdot \hat{x}^{\otimes(i+1)} = T\cdot \hat{x}^{\otimes(i+1)}$ with symmetric tensors $S\coloneqq\D_0\psi^{i+1} (\D_0 F)^{\otimes (i+1)}$ and $T\coloneqq X\D_0 \psi^{i+1}$.
	Since symmetric tensors are completely determined by their action on tensors of the form $\hat{x}^{\otimes (i+1)}$ \cite[Thm~1]{thomas2014polarization}, it follows that $0 = S - T$, or 
	\begin{equation}\label{eq:lem-chain-rule}
	0 = \D_0^{i+1} \psi \left(\D_0 F\right)^{\otimes (i+1)} - X\D_0^{i+1} \psi = T_{(i+1)}(\D_0^{i+1}\psi),
	\end{equation}
	where the linear operator $T_{(i+1)} \colon \Ll((\R^n)^{\otimes (i+1)},\C^m)\to \Ll((\R^n)^{\otimes (i+1)},\C^m)$ is as defined in Lemma \ref{lem:nonresonance-implies-invertible} (taking $Y \coloneqq \D_0 F$).	
	Lemma \ref{lem:nonresonance-implies-invertible} implies that $T_{(i+1)}$ is invertible since $(X,\D_0 F)$ is $k$-nonresonant, so \eqref{eq:lem-chain-rule} implies that $\D_0^{i+1}\psi = 0$.
	This completes the inductive step and the proof.
	\end{proof}
	
	\begin{Lem}\label{lem:psi-identically-0}
		Let $F\in C^1(\R^n,\R^n)$ be a diffeomorphism such that the origin is a globally attracting hyperbolic fixed point for the dynamical system defined by iterating $F$, where $n\geq 1$.
		Fix $k\in \N_{\geq 1}\cup \{+\infty\}$ and $0\leq \alpha \leq 1$.
		Let $e^A\in \GL(m,\C)$ have spectral radius $\rho(e^A) < 1$ and satisfy $\nu(e^A,\D_0 F) < k + \alpha$.
		Assume $\psi\in \CLKA(\R^n,\R^m)$ satisfies
		\begin{equation}\label{eq:lem-uniq-remainder-conj}
		\psi \circ F = e^{A} \psi
		\end{equation}
		and
		\begin{equation}\label{eq:psi-ders-vanish}
		 \D^i_0 \psi= 0
		\end{equation}
		for all integers $0\leq i < k+\alpha$ (note the case $\alpha = 0$ which does not require vanishing of the $k$-th derivative).
		Then $\psi \equiv 0$.
	\end{Lem}
	\begin{proof}
		We first observe that since (i) $0$ is asymptotically stable for the iterated dynamical system defined by $F$, (ii) $\psi$ is continuous, and (iii) $\rho(e^{A})<1$, it follows that $\psi(0) = 0$ since 
		\begin{equation}\label{eq:psi-zero-zero}
		0 = \lim_{n\to \infty}e^{nA}\psi(x_0) = \lim_{n\to\infty}\psi(F^n(x_0)) = \psi(0)
		\end{equation}
		for any $x_0 \in \R^n$.
		The second equality follows from \eqref{eq:lem-uniq-remainder-conj}.
		
		For the remainder of the proof, fix any $r>0$ with $\nu(e^A,\D_0 F) < r< k+\alpha$, fix $x_0\in \R^n\setminus \{0\}$, and define $x_j \coloneqq F^j(x_0)$ for $j\in \N$.
		Let $0\leq k'<r$ be the largest integer smaller than $r$. 
		Taylor's theorem for $\CLKA$ functions \cite[p.~162]{de1999regularity} says that $$\psi(x) = \sum_{i=0}^{k'} \D_0^i \psi\cdot x^{\otimes i} + R(x),$$
		where $\lim_{x\to 0}\frac{R(x)}{\norm{x}^{r}} = 0$.
		Equations \eqref{eq:psi-ders-vanish} and \eqref{eq:psi-zero-zero} imply that all of the terms in the sum above vanish, so we obtain $\psi= R$.
		Using \eqref{eq:lem-uniq-remainder-conj} it follows that $e^{j A}\psi = \psi \circ F^j = R\circ F^j $, and since $x_j = F^j(x_0)$ we obtain
		\begin{equation}\label{eq:main-thm-1st-eqn}
		e^{jA}\psi(x_0) = R(x_j), \qquad \lim_{x\to 0}\frac{R(x)}{\norm{x}^{r}} = 0.
		\end{equation}
		Denote by $M\coloneqq \rho(\D_0 F)<1$ the spectral radius of $\D_0 F$.
		Since $\nu(e^A,\D_0 F)< r$, \eqref{eq:spread-def} implies that all eigenvalues $\mu$ of $e^A$ satisfy $|\mu|> M^r$.
		Since this inequality is strict, by continuity there is $\epsilon > 0$ such that $0 < (M+\epsilon) < 1$ and 
		\begin{equation}\label{eq:M-eps-mu-bound}
		\forall \mu \in \textnormal{spec}(e^A)\colon |\mu|> (M+\epsilon)^r.
		\end{equation}
		By replacing $\norm{\slot}$ with an \concept{adapted norm}, we may assume that $\norm{\D_0 F}<(M+\epsilon/2)$.\footnote{Such an adapted norm always exists. It can be constructed as the Euclidean norm with respect to a choice of basis placing $\D_0 F$ in ``$\epsilon$-Jordan form'' \cite[pp.~279--280]{hirsch1974differential}, wherein the off-diagonal unity entries of the usual Jordan normal form are replaced by $\epsilon$. An alternative construction of an adapted norm proceeds by suitably averaging a given norm along the dynamics linearized at the fixed point \cite[Sec.~A.1]{cabre2003parameterization1}; an analogous technique also works in more general situations.}
		Since $\norm{F(x)-\D_0F\cdot x}/\norm{x}\to 0$ as $\norm{x} \to 0$, there exists $b > 0$ such that $\norm{F(x)}< (M+\epsilon)\norm{x}$ if $\norm{x} < b$ (cf. \cite[p.~281]{hirsch1974differential}).
		Since the origin is globally asymptotically stable and since $\{\norm{x}<b\}$ is positively invariant by the preceding sentence (recall that $(M+\epsilon)<1$), there exists $j_0\in \N_{\geq 1}$ such that $\norm{x_j}< b$ for all $j\geq j_0$.
		Hence for all $j\geq j_0$:  
		\begin{equation}\label{eq:norm-xj-bound}
		\norm{x_j}<(M+\epsilon)^{j-j_0}\norm{x_{j_0}} = C\cdot (M+\epsilon)^j\norm{x_0},
		\end{equation}
		where $C\coloneqq (M+\epsilon)^{-j_0}\norm{x_{j_0}}\norm{x_0}^{-1}$.
		
		Dividing both sides of \eqref{eq:main-thm-1st-eqn} by $\norm{x_j}^{r}$, multiplying by $1=\frac{(M+\epsilon)^{jr}}{(M+\epsilon)^{jr}}$ and taking the limit as $j\to \infty$ yields 
		\begin{equation}\label{eq:main-thm-2nd-eqn}
		0 = \lim_{j\to\infty} e^{jA} \frac{\psi(x_0)}{\norm{x_j}^{r}} = \lim_{j\to \infty} \left(\frac{(M+\epsilon)^j}{\norm{x_j}}\right)^r \left(\frac{e^{A}}{(M+\epsilon)^{r}}\right)^j\psi(x_0).
		\end{equation}
		Since \eqref{eq:M-eps-mu-bound} implies that all eigenvalues of $\frac{e^A}{(M+\epsilon)^r}$ have modulus strictly larger than $1$, the moduli of all nonzero entries in the (upper triangular, complex) Jordan normal form of $(\frac{e^A}{(M+\epsilon)^r})^j$ approach $\infty$ as $j\to \infty$.\footnote{The desire for this conclusion was part of what motivated our definition of the spectral spread $\nu(\slot,\slot)$.}
		If $\psi(x_0)\neq 0$, it follows that the absolute value of at least one component of $(\frac{e^A}{(M+\epsilon)^r})^j\psi(x_0)$ with respect to the Jordan basis approaches $\infty$ as $j\to \infty$. 
		Moreover, \eqref{eq:norm-xj-bound} implies that $\left(\frac{(M+\epsilon)^j}{\norm{x_j}}\right)^r>C^{-r}\norm{x_0}^{-r}>0$ for all $j$, so the product of this quantity with the diverging quantity $(\frac{e^A}{(M+\epsilon)^r})^j\psi(x_0)$ also diverges as $j\to\infty$.
		It follows that \eqref{eq:main-thm-2nd-eqn} holds if and only if $\psi(x_0) = 0$.
		Since $x_0 \in \R^n\setminus\{0\}$ was arbitrary, and since we already obtained $\psi(0) = 0$ in \eqref{eq:psi-zero-zero}, it follows that $\psi \equiv 0$ on $\R^n$. 
		This completes the proof. 
	\end{proof}
	Using Lemmas \ref{lem:jets-zero-maps} and \ref{lem:psi-identically-0}, we now prove the uniqueness portion of Theorem \ref{th:main-thm}.
	\begin{proof}[Proof of the uniqueness portion of Theorem \ref{th:main-thm}]
	Since $x_0$ is globally asymptotically stable, the Brown-Stallings theorem \cite[Lem~2.1]{wilson1967structure} implies that there is a diffeomorphism $Q\approx \R^n$ sending $x_0$ to $0$, where $n = \dim(Q)$, so we may assume that $Q = \R^n$ and $x_0 = 0$.\footnote{For example, Wilson states in \cite[Thm~2.2]{wilson1967structure} this result for the special case of a flow generated by a $C^1$ vector field, but his argument based on the Brown-Stallings theorem \cite[Lem~2.1]{wilson1967structure} works equally well for any $C^1$ flow or diffeomorphism having a globally asymptotically stable fixed point.}
	Define the diffeomorphism $F\coloneqq \Phi^1$ to be the time-$1$ map.
	Let $\psi_1$ and $\psi_2$ be two functions satisfying $\D_0 \psi_i = B$ and $\psi_i \circ F = e^{A} \psi_i$ for $i = 1,2$.
	Then $\psi\coloneqq \psi_1-\psi_2$ satisfies $\D_0 \psi = 0$ and $\psi \circ F = e^{A} \psi$. 
	Lemma \ref{lem:jets-zero-maps} implies that $\D_0^i \psi =  0$ for all $1\leq i < k + 1$, and Lemma \ref{lem:psi-identically-0} then implies that $\psi_1 - \psi_2 = \psi \equiv 0$.
	If $e^{A}$ and $B$ are real, then we can define $\psi_2\coloneqq \bar{\psi}_1$ to be the complex conjugate of $\psi_1$, and so the preceding implies that $\psi_1 = \bar{\psi}_1$; hence $\psi_1$ is real if $e^{A}$ and $B$ are real.
	This completes the proof of the uniqueness statement of Theorem \ref{th:main-thm}.
	\end{proof}
	
	\subsubsection{Proof of existence}\label{sec:main-proof-exist}
	In this section, we prove the existence portion of Theorem \ref{th:main-thm}.
	As with the proof of the uniqueness portion, the proof consists of an algebraic part and an analytic part.
	The techniques we use in the existence proof are similar to those used in \cite{sternberg1957local,cabre2003parameterization1}.
	The algebraic portion of our proof is carried out in Lemma \ref{lem:existence-approx-conj}, and the analytic portion is carried out in Lemma \ref{lem-make-approx-exact}.
		
	\begin{Lem}[Existence and uniqueness of approximate polynomial linearizing factors for diffeomorphisms]\label{lem:existence-approx-conj}
		Fix $k \in \N_{\geq 1}$ and let $F\in C^k(\R^n,\R^n)$ have the origin as a fixed point, where $n\geq 1$. 
		Let $m\in \N_{\geq 1}$ and $X \in \C^{m\times m}$ be such that $(X,\D_0 F)$ is $k$-nonresonant, and assume $B\in \C^{m\times n}$ satisfies $$B \D_0 F = X B.$$
		Then there exists a unique degree-$k$ symmetric polynomial $P\colon \R^n\to \C^m$  vanishing at $0$ such that $\D_0 P = B$ and such that 
		\begin{equation}\label{eq:poly-conjugacy}
		P \circ F = X P + R,
		\end{equation}
		where $R$ satisfies $\D_0^i R = 0$ for all $0\leq i \leq k$.
		Furthermore, if $X\in \R^{m\times m}$ and $B\in \R^{m\times n}$ are real, then this unique polynomial $P\colon \R^n \to \R^m\subset \C^m$ is real.
	\end{Lem}
	
   \begin{Rem}
    We state and prove Lemma \ref{lem:existence-approx-conj} for the case of finite $k$ only and rely on a bootstrapping method to prove the existence portion of Theorem \ref{th:main-thm} for the case $k = +\infty$ at the end of this section.
    We believe it is possible to prove a $C^\infty$ version of the existence portion of Lemma \ref{lem:existence-approx-conj} (loosely speaking) using the fact that for every formal power series there exists a $C^\infty$ function with matching derivatives \cite[p.~34]{nelson1970topics}, but we did not attempt to do this.   	
   \end{Rem}	
	
	\begin{proof}
		By Lemma \ref{lem:nonresonance-implies-invertible}, the linear operator
        \begin{equation}\label{eq:t-penult-lemma}
        T_i \colon \Ll((\R^n)^{\otimes i},\C^m) \to \Ll((\R^n)^{\otimes i},\C^m), \qquad T_i(P_i)\coloneqq P_i (\D_0 F)^{\otimes i} - X P_i
        \end{equation}	
        is invertible for all $1 < i \leq k$.
        Denoting by $\Sym^i((\R^n)^{\otimes i},\C^m)\subset \Ll((\R^n)^{\otimes i},\C^m)$ the linear subspace corresponding to \emph{symmetric} multilinear maps $(\R^n)^{i}\to \C^m$ via the universal property of the tensor product, we see from \eqref{eq:t-penult-lemma} that $T(\Sym^i((\R^n)^{\otimes i},\C^m))\subset \Sym^i((\R^n)^{\otimes i},\C^m)$.
        Invertibility of $T_i$ and dimension counting imply the opposite inclusion, so $T_i$ restricts to a well-defined linear automorphism of $\Sym^i((\R^n)^{\otimes i},\C^m)$.
        
		By Taylor's theorem we may uniquely write $F$ as a degree-$k$ symmetric polynomial plus remainder, $F(x) = \sum_{i=1}^k F_i\cdot x^{\otimes i} + R_1$, where $F_1 = \D_0 F$ and $\D_0^i R_1 = 0$ for all $0\leq i \leq k$.
		Defining $F_{\otimes [j]}\coloneqq F_{j_1}\otimes \cdots \otimes F_{j_\ell}$  for any multi-index $j\in \N_{\ell \geq 1}$ and using the notation $|j|\coloneqq j_1+\cdots + j_\ell$,  \eqref{eq:poly-conjugacy} is equivalent to
		\begin{equation}\label{eq:poly-replace-with-evaluation}
		\sum_{\ell=1}^k P_\ell \cdot\sum_{\substack{j \in \N_{\geq 1}^\ell\\|j|\leq k}}  F_{\otimes[j]} \cdot x^{\otimes |j|}  = X \sum_{\ell=1}^k P_\ell \cdot x^{\otimes \ell},
		\end{equation}
		where $j=|(j_1,\ldots,j_\ell)| = \sum_{i=1}^\ell j_i$, $P(x) = \sum_{\ell=1}^k P_\ell \cdot x^{\otimes \ell}$, and $P_1 = B$.
		It follows from an inductive argument (equating coefficients of $x^{\otimes i}$) that \eqref{eq:poly-conjugacy} is equivalent to
		\begin{equation}\label{eq:poly-replace-with-evaluation-rewrite}
		\forall i \in \{1,\ldots, k\}\colon \left(\sum_{\ell=1}^i P_\ell \sum_{\substack{j \in \N_{\geq 1}^\ell\\|j|=i}}  F_{\otimes[j]}\right) \cdot x^{\otimes i}  = X P_i \cdot x^{\otimes i}.
		\end{equation}		
		If we require that all tensors $P_\ell$ are symmetric then, for each fixed $i$, the two tensors acting on $x^{\otimes i}$ in \eqref{eq:poly-replace-with-evaluation-rewrite} are symmetric.
		Since symmetric tensors are completely determined by their action on all tensors of the form $x^{\otimes i}$ \cite[Thm~1]{thomas2014polarization}, it follows that \eqref{eq:poly-conjugacy} is equivalent to
		\begin{equation}\label{eq:poly-replace}
		\begin{split}
		\forall i \in \{1,\ldots, k\}\colon   \sum_{\ell=1}^i P_\ell \sum_{\substack{j \in \N_{\geq 1}^\ell\\|j|=i}}  F_{\otimes[j]}&= X P_i	
		\end{split}
		\end{equation}
		or, after rearranging terms, 
		\begin{equation}\label{eq:poly-induct}
		\begin{split}
		\forall i \in \{1,\ldots, k\}\colon   \sum_{\ell=1}^{i-1}P_\ell \sum_{\substack{j \in \N_{\geq 1}^\ell\\|j|=i}}  F_{\otimes[j]}&=  \underbrace{X P_i - P_i (\D_0 F)^{\otimes i}}_{-T_i(P_i)}	
		\end{split}
		\end{equation}
		since $\D_0 F = F_1$.
		By our assumptions, $X B- B\D_0 F = 0$ and $P_1 = \D_0 P = B$.
        Moreover, as discussed above, $T_i|_{\Sym^i((\R^n)^{\otimes i},\C^m)}$ is a well-defined linear automorphism of the subspace $\Sym^i((\R^n)^{\otimes i},\C^m)$.	
        Thus, \eqref{eq:poly-induct} can be inductively solved for the tensors $P_1,\ldots, P_k$, and the preceding sentence implies that these solutions are unique and symmetric.
        Thus, so is $P$.		
		
		Finally, assume that $X\in \R^{m\times m}$ and $B\in \R^{m\times n}$ are real, and assume by induction that $B = P_1, P_2,\ldots, P_{i-1}$ are real.
		Taking the complex conjugate of \eqref{eq:poly-induct}, we see that $P_i$ solves \eqref{eq:poly-induct} if and only if its complex conjugate $\bar{P}_i$ solves \eqref{eq:poly-induct}.
		Invertibility of $T_i$ thus implies that $P_i = \bar{P}_i$, so $P_i$ is real.
		By induction, $P_1,\ldots, P_k$ are real.
		Thus, so is $P$.
		This completes the proof. 
	\end{proof}	
	
	\begin{Lem}[Making approximate linearizing factors exact]\label{lem-make-approx-exact}
		Fix $k \in \N_{\geq 1} \cup \{+\infty\}$, $0 \leq \alpha \leq 1$, and let $F\colon \R^n \to \R^n$ be a $\CLKA$ diffeomorphism such that the origin is a globally attracting hyperbolic fixed point for the dynamical system defined by iterating $F$, where $n\geq 1$.
		Let $m\in \N_{\geq 1}$ and $e^{A} \in \GL(m,\C)$ satisfy $\nu(e^{A},\D_0 F) < k + \alpha$,
		and assume that there exists $P\in \CLKA(\R^n,\C^m)$  such that $$P\circ F = e^{A} P + R,$$
		where $R\in \CLKA(\R^n,\C^m)$ satisfies  $\D_0^i R = 0$ for all integers $0\leq i < k+\alpha$ (note the case $\alpha = 0$ which does not require vanishing of the $k$-th derivative).

		Then there exists a unique $\varphi \in \CLKA(\R^n,\C^m)$ such that $\D_0^i \varphi=0$ for all integers $0 \leq i < k + \alpha$ and such that $\psi\coloneqq P + \varphi$ satisfies $$\psi \circ F = e^{A} \psi.$$
		In fact, $$\psi = \lim_{j\to \infty}e^{-jA}P\circ F^j$$ in the topology of $C^{k,\alpha}$-uniform convergence on compact subsets of $\R^n$  if $k < +\infty$, and in the topology of $C^{k'}$-uniform convergence on compact subsets of $Q$ for any $k'\in \N_{\geq 1}$ if $k=+\infty$.
		Furthermore, if $e^{A}\in \GL(m,\R)$ is real and $P\in \CLKA(\R^n,\R^m)$ is real, then $\varphi,\psi \colon \R^n \to \R^m \subset \C^m$ are real.				
	\end{Lem}
		
	\begin{proof}
		We first assume that $k<+\infty$ and delay consideration of the case $k = +\infty$ until the end of the proof.
		
		\textit{Adapted norms.}
		Later in the proof we will require that the following bound on operator norms holds (needed following \eqref{eq:contract-holder-coeff}):
		\begin{equation}\label{eq:existence-spread-rewritten}
        \norm{e^{-A}}\norm{\D_0 F}^{k+\alpha}< 1.
        \end{equation}
		Due to our assumption that $\nu(e^A,\D_0 F) < k +\alpha$, this bound can always be made to hold by using an appropriate choice of ``adapted'' norms (which induce the operator norms) on the underlying vector spaces $\R^n$ and $\C^m$, and so we may (and do) assume that \eqref{eq:existence-spread-rewritten} holds in the remainder of the proof.	
		
		But first we argue that such norms can indeed be chosen. 
		Let $\lambda \in \textnormal{spec}(\D_0 F)$ and $\mu\in \textnormal{spec}(e^{A})$ be the eigenvalues of $\D_0 F$ and $e^{A}$ with largest and smallest modulus, respectively.
		For any $\kappa > 0$, there exist adapted norms  (both denoted by $\norm{\cdot}$) on $\R^n$ and $\C^m$ having the property that the induced operator norms $\norm{e^{-A}}$ and $\norm{\D_0 F}$ satisfy   \cite[pp.~279--280]{hirsch1974differential}, \cite[Sec.~A.1]{cabre2003parameterization1}: 
		\begin{equation}\label{eq:adapted-norm}
		|\norm{\D_0 F}  - |\lambda|| \leq \kappa, \qquad \left|\norm{e^{-A}} - |\mu|^{-1}\right| \leq \kappa.
		\end{equation}
        Now since $\nu(e^{A},\D_0 F) < k+\alpha$ and since $|\lambda|<1$, it follows from \eqref{eq:spread-def} that $|\mu|^{-1}|\lambda|^{k+\alpha}<1$.
        The inequalities \eqref{eq:adapted-norm} imply that $\norm{e^{-A}}\norm{\D_0 F}^{k+\alpha}$ can be made arbitrarily close to $|\mu|^{-1}|\lambda|^{k+\alpha}$ by making $\kappa$ small, so choosing $\kappa$ sufficiently small yields \eqref{eq:existence-spread-rewritten} as claimed.
        For later use we also note  \eqref{eq:adapted-norm} implies that $\D_0 F$ is a strict contraction if $\kappa$ is small enough since $|\lambda| < 1$. 
        This in turn implies that
        \begin{equation}\label{eq:B-pos-invariant}
        F(B)\subset B
        \end{equation}
        if $B\subset \R^n$ is a sufficiently small (adapted norm) ball centered at the origin \cite[p.~281]{hirsch1974differential}. 
        We henceforth assume this is the case.

		\textit{Definition of function spaces.}
        Let $U \subset \R^n$ be a precompact open set and denote by $\bar{U} = \textnormal{cl}(U)$ its compact closure.		
		Given any Banach space $X$ and $k\in \N_{\geq 0}$, let $C^k(\bar{U},X)$ be the space of continuous functions $G\colon \bar{U}\to X$ having partial derivatives of order less than or equal to $k$ which are uniformly continuous on $U$, in which case they extend to continuous functions on $\bar{U}$.
		We equip $C^k(\bar{U},X)$ with the standard norm
		\begin{equation*}
		\norm{G}_k\coloneqq \sum_{i=0}^k \sup_{x\in U}\norm{\D_x^i G}
		\end{equation*}
		making $C^k(\bar{U},X)$ into a Banach space  \cite{de1999regularity}.
		For a Banach space $Y$ and $0< \alpha \leq 1$, we define the $\alpha$-H\"{o}lder constant $[H]_\alpha$ of a map $H\colon \bar{U}\to Y$ via
		\begin{equation*}
		[H]_\alpha \coloneqq \sup_{\substack{x,y\in U\\x \neq y}}\frac{\norm{H(x) - H(y)}}{\norm{x-y}^\alpha}.
		\end{equation*}
		For $0 < \alpha \leq 1$ we let
		$\CKA(\bar{U},X)$ be the subset of functions $G\in C^k(\bar{U},X)$ for which $[\D^k G]_\alpha<+\infty$, and we equip $\CKA(\bar{U},X)$ with the standard norm 
		\begin{equation}\label{eq:cka-norm-def}
		\norm{G}_{k,\alpha}\coloneqq \norm{G}_k + [\D^k G]_\alpha 
		\end{equation}
		making $\CKA(\bar{U},X)$ into a Banach space \cite{de1999regularity}.\footnote{Different $C^k$ and $\CKA$ norms are actually used in \cite[Def.~2.1,~Def.~2.5]{de1999regularity}, namely, $\max_{0\leq i\leq k} \norm{\D^i G}_0$ and $\max(\max_{0\leq i \leq k} \norm{\D^i G}_0, [\D^k G]_\alpha)$, but these two norms are \concept{equivalent} (in the sense of norms) to the corresponding norms we have chosen.}
		For $\alpha = 0$, we identify $C^{k,0}(\bar{U},X)$ with $C^k(\bar{U},X)$ and make the special definition $$\norm{\slot}_{k,0}\coloneqq \norm{\slot}_k.$$
		
		In what follows, let $B \subset \R^n$ be a closed ball centered at the origin and let $\F\subset  \CKA(B,\C^m)$ denote the subspace of functions $\varphi$ such that $\D_0^i\varphi = 0$ for all integers $0\leq i < k+\alpha$;
		$\F$ is a closed linear subspace of $\CKA(B,\C^m)$, hence also a Banach space.

		\textit{Preliminary estimates.}
		For any $\varphi\in \F$, $x\in B$, and all integers $1\leq i< k+\alpha$, we have that $\norm{\D^{i-1}_x\varphi}\leq \norm{x}\cdot \int_0^1 \norm{\D^{i}_{tx}\varphi}\, dt$ by the fundamental theorem of calculus, the chain rule, and the definition of $\F$.
		The derivatives $\D^i\varphi$ vanish at $x=0$ for all integers $0\leq i < k+\alpha$ by the definition of $\F$, so the preceding sentence and an induction argument imply that, for any $\epsilon > 0$, if the radius of $B$ is sufficiently small then for any $\varphi \in \F$:
		\begin{equation}\label{eq:norm-k-minus-one-bound}
		\begin{split}
		\norm{\varphi}_{k-1} &\leq \epsilon\norm{\D^k\varphi}_0\\
		\norm{\varphi}_k &\leq (1+\epsilon)\norm{\D^k\varphi}_0.
		\end{split}
		\end{equation}
		If $\alpha > 0$, the additional fact that $\norm{\D^k_x \varphi}\leq \norm{x} [\D^k\varphi]_\alpha$ further implies that
		\begin{equation}\label{eq:norm-k-bound}
		\begin{split}
		\norm{\varphi}_{k} &\leq \epsilon [\D^k \varphi]_{\alpha}\\
		\norm{\varphi}_{k,\alpha}&\leq (1+\epsilon) [\D^k \varphi]_{\alpha}
		\end{split}
		\end{equation}
		if the radius of $B$ is sufficiently small.

		\textit{Defining a linear contraction mapping on $\F$.}
		Recall that $F\colon \R^n\to \R^n$ is the diffeomorphism from the statement of the lemma.
		By \eqref{eq:B-pos-invariant}, all sufficiently small closed (adapted norm) balls $B\subset \R^n$ centered at the origin satisfy $F(B)\subset B$.
		Additionally, since $F\in \CLKA(\R^n,\R^n)$ and $B$ is compact, $F|_B \in \CKA(B,\R^n)$.
		It follows that there is a well-defined linear operator $T\colon \CKA(B,\C^m)\to \CKA(B,\C^m)$ given by\footnote{That $\D^k T(\varphi)$ is $\alpha$-H\"older follows from the chain rule, the fact that the first $k-1$ derivatives of $F$ and of $\varphi$ are $C^1$ and hence Lipschitz (hence also $\alpha$-H\"{o}lder), the fact that the composition of an $\alpha$-H\"older function with a Lipschitz function is again $\alpha$-H\"older, and the fact that the product of bounded $\alpha$-H\"older functions is again $\alpha$-H\"{o}lder (see, e.g., \cite[Lem~1.19]{eldering2013normally}).\label{foot:holder-function-properties}} 
		\begin{equation}
		T(\varphi)\coloneqq e^{-A}\varphi\circ F|_B.
		\end{equation} 
		Note that $T(\F)\subset \F$, so that $\F$ is an invariant subspace for $T$.
		We claim that there is a choice of $B$ so that $T|_\F\colon \F \to \F$ is a (strict) contraction with constant $\beta < 1$:
		\begin{equation}\label{eq:T-contract}
		\forall \varphi\in \F\colon \norm{T(\varphi)}_{k,\alpha}\leq \beta \norm{\varphi}_{k,\alpha}.
		\end{equation} 
		To show this we give an argument essentially due to Sternberg, but which generalizes the proof of \cite[Thm~2]{sternberg1957local} to the case of linearizing semiconjugacies and to the $\CKA$ setting (allowing $\alpha > 0$).
		Using the notation $\D^{\otimes[j]}_x F\coloneqq \D^{j_1}_x F \otimes \cdots \otimes \D^{j_i}_x F$ for a multi-index $j\in \N^i_{\geq 1}$, we compute 
		\begin{equation}\label{eq:existence-k-der-estimate}
		\D_x^k (T(\varphi)) =  e^{-A} \D^k_{F(x)} \varphi \cdot (\D_x F)^{\otimes k} + e^{-A} \sum_{i=1}^{k-1}\sum_{\substack{j \in \N^i_{\geq 1}\\|j|= k}} C_{i,j} \D^i_{F(x)}\varphi \cdot \D^{\otimes[j]}_x F,
		\end{equation}
		where the integer coefficients $C_{i,j}\in \N_{\geq 1}$ are combinatorially determined by Fa\`{a} di Bruno's formula for the ``higher-order chain rule'' and are therefore independent of $B$.\footnote{The ``higher-order chain rule,'' also known as Fa\`{a} di Bruno's formula, gives a general expression for higher-order derivatives of the composition of two functions (see \cite{jacobs2014stare} for an exposition).} 
		We choose $B$ sufficiently small that its diameter is less than $1$ and note that, by continuity and compactness of $\{\norm{x}\leq 1\}$ and the fact that $F\in \CLKA$, there exists a constant $N_0 > 0$ such that\footnote{Since we have not yet chosen $B$ (other than stipulating that its diameter be smaller than $1$), to avoid circular reasoning we are using $\{\norm{x}\leq 1\}$ in place of $B$ in \eqref{eq:fa-di-sum-bound} to make clear that the estimate holds for \emph{all} closed balls $B\subset \{\norm{x}\leq 1\}$ centered at $0$.} 
		\begin{equation}\label{eq:fa-di-sum-bound}
		\sum_{i=1}^{k-1}\sum_{\substack{j \in \N^\ell_{\geq 1}\\|j|= k}} C_{i,j}\cdot  \left[ \left(\sup_{\norm{x} \leq  1} \norm{\D^{\otimes[j]}_x F}\right)\left( 1 + \sup_{\norm{x} \leq  1}\norm{\D_x F}^\alpha\right) + \sup_{\substack{\norm{x},\norm{y}\leq 1\\x \neq y}}\frac{\norm{\D^{\otimes[j]}_x F - \D^{\otimes[j]}_y F}}{\norm{x-y}^\alpha}\right] < N_0.
		\end{equation}
		Using \eqref{eq:norm-k-minus-one-bound} and \eqref{eq:fa-di-sum-bound} to bound the sum in \eqref{eq:existence-k-der-estimate},  it follows that
		\begin{equation}\label{eq:contract-k-deriv}
		\norm{\D^k T(\varphi)}_0 \leq \norm{e^{-A}}(\norm{\D F}^k_0 + \epsilon N_0)\norm{\D^k \varphi}_0.
		\end{equation}
		
		For the case that $\alpha > 0$, we will now use \eqref{eq:norm-k-bound} to obtain a bound on $[\D_x^k T(\varphi)]_\alpha$ analogous to \eqref{eq:contract-k-deriv}.
		In order to do this, we use the estimate $[x\mapsto \D^k_{F(x)}\varphi]_\alpha\leq [\D^k \varphi]_\alpha\norm{\D F}_0^\alpha$ and the product rule $[fg]_\alpha \leq \norm{f}_0 [g]_\alpha + [f]_\alpha \norm{g}_0$ for H\"older constants (see, e.g., \cite[Lem~1.19]{eldering2013normally}) to bound the first term of \eqref{eq:existence-k-der-estimate} by $$\norm{e^{-A}}\left([\D^k \varphi]_\alpha\norm{\D F}_0^{k+\alpha} + \norm{\D^k \varphi}_0 [(\D F)^{\otimes k}]_\alpha\right) \leq \norm{e^{-A}}\left(\norm{\D F}_0^{k+\alpha} + \epsilon [(\D F)^{\otimes k}]_\alpha\right)[\D^k\varphi]_{\alpha},$$ where we have used \eqref{eq:norm-k-bound} to bound the second term in parentheses on the left side.
		Next, we use \eqref{eq:norm-k-bound}, \eqref{eq:fa-di-sum-bound}, the product rule for H\"older constants again, and for $1\leq i \leq k-1$ the estimates $[x\mapsto \D^i_{F(x)}\varphi]_\alpha\leq [\D^i \varphi]_\alpha\norm{\D F}_0^\alpha \leq \epsilon [\D^k \varphi]_\alpha\norm{\D F}_0^\alpha$ to bound the second term of \eqref{eq:existence-k-der-estimate} by $\epsilon N_0 \norm{e^{-A}}[\D^k \varphi]_\alpha$.
		This last estimate we used follows from \eqref{eq:norm-k-bound} and the fact that we are requiring $B$ to have diameter less than $1$, so that $[\D^i \varphi]_\alpha \leq \norm{\D^{i+1}\varphi}_0\leq \epsilon [\D^k \varphi]_\alpha$.
		We finally obtain
		\begin{align}\label{eq:contract-holder-coeff}
		\left[\D_x^k T(\varphi)\right]_\alpha 
		&\leq \norm{e^{-A}}\left(\norm{\D F}_0^{k+\alpha} + \epsilon [(\D F)^{\otimes k}]_\alpha + \epsilon N_0\right)[\D^k\varphi]_\alpha.
		\end{align}
		
		The estimate		 \eqref{eq:existence-spread-rewritten} and continuity imply that $\norm{e^{-A}}\norm{\D F}_0^{k+\alpha} < 1$ if $B$ is sufficiently small.
		Hence if $\epsilon$ is sufficiently small, the quantities respectively multiplying $\norm{\D^k\varphi}_0$ and $[\D^k \varphi]_\alpha$ in \eqref{eq:contract-k-deriv} and \eqref{eq:contract-holder-coeff} will be bounded above by some positive constant $\beta' < 1$.
        The discussion preceding \eqref{eq:norm-k-minus-one-bound} and \eqref{eq:norm-k-bound} implies that we can indeed take $\epsilon$ this small after possibly further shrinking $B$, so it follows that  $\norm{\D^k T(\varphi)}_0 < \beta' \norm{\D^k \varphi}_0$ and, if $\alpha > 0$, also $[D^k T(\varphi)]_\alpha < \beta' [\D^k \varphi]_\alpha$.
        We therefore obtain a contraction estimate on the highest derivative and its H\"{o}lder constant (if $\alpha > 0$) \emph{only}.
		However, we can combine this observation with the second inequalities from each of the two displays \eqref{eq:norm-k-minus-one-bound} and \eqref{eq:norm-k-bound}, together with the fact that $T(\F)\subset \F$, to obtain in both cases ($\alpha = 0$ and $\alpha > 0$) the following estimate involving \emph{all} of the derivatives:
		\begin{equation}
		\norm{T(\varphi)}_{k,\alpha}  \leq (1+\epsilon)\beta' \norm{\varphi}_{k,\alpha}.
		\end{equation}	
		(This technique for the case $\alpha = 0$ was also used in the proof of \cite[Thm 2]{sternberg1957local}.)
		Define $\beta \coloneqq (1+\epsilon)\beta'$.
		Since $\beta' < 1$, if necessary we may shrink $B$ further to ensure that $\epsilon$ may be taken sufficiently small that $\beta < 1$.
		It then follows that $T|_\F$ is a (strict) contraction; this completes the proof of \eqref{eq:T-contract}.
		
		\textit{Existence and uniqueness of a linearizing factor defined on $B$.}
		We will now find a locally-defined (i.e., defined on $B$ rather than on all of $\R^n$) linearizing factor $\tilde{\psi}\in \CKA(B,\C^m)$ of the form $\tilde{\psi} = P|_B + \tilde{\varphi}$, where $\tilde{\varphi}\in \F$ and $P\colon \R^n \to \C^m$ is as in the statement of the lemma. 	    
	    By definition, $\tilde{\psi}$ is linearizing if and only if $\tilde{\psi} = e^{-A} \tilde{\psi}\circ F|_B = T(\tilde{\psi})$, so we need to solve the equation $
	    P|_B + \tilde{\varphi} = T(P|_B + \tilde{\varphi})$ for $\tilde \varphi\in \F$.
		(We are writing $P|_B$ rather than $P$ because $\tilde{\varphi}$ is a function with domain $B$ rather than $\R^n$, and also because $T$ is a linear operator defined on functions with domain $B$.)	    
	    Since $T$ is linear, after rearranging we see that this amounts to solving
	    \begin{equation}\label{eq:tilde-phi-1}
	    \left(\id_{\CKA(B,\C^m)} - T\right)\tilde{\varphi} = T(P|_B) -P|_B.
	    \end{equation}
	    One of the assumptions of the lemma directly implies that $\left(P\circ F - e^A P\right)|_B \in \F$, and this implies that the right hand side of \eqref{eq:tilde-phi-1} belongs to $\F$ since $e^{-A} \cdot \F \subset \F$.
	    Since $T(\F)\subset \F$, it follows that we may rewrite \eqref{eq:tilde-phi-1} as
	    \begin{equation}\label{eq:tilde-phi-2}
        \left(\id_{\F} - T|_\F\right)\tilde{\varphi} = T(P|_B) -P|_B.
        \end{equation}	    
        We showed earlier that $T|_\F$ is a strict contraction, i.e., its operator norm satisfies $\norm{T|_\F}_{k,\alpha} < 1$.
        It follows that $(\id_\F - T|_\F)\colon \F \to \F$ has a bounded inverse given by the corresponding Neumann series, so that \eqref{eq:tilde-phi-2} has a unique solution $\tilde{\varphi}$ given by
        \begin{equation}\label{eq:tilde-phi-neumann}
        \tilde{\varphi} = \left(\id_{\F} - T|_\F\right)^{-1} \cdot \left(T(P|_B) -P|_B\right) = \sum_{n=0}^\infty (T|_\F)^n \cdot \left(T(P|_B) -P|_B\right),
        \end{equation}
        and $\tilde{\psi}\coloneqq P|_B + \tilde{\varphi}$ satisfies $\tilde{\psi} = e^{-A} \tilde{\psi}\circ F|_B = T(\tilde{\psi})$ as discussed above.
		
		\textit{Extension to a unique global linearizing factor.}
		Since $x_0$ is globally asymptotically stable and since $B$ is positively invariant, for every $x\in B$ there exists $j(x)\in \N_{\geq 0}$ such that, for all $j > j(x)$, $F^j(x)\in \interior(B)$.
		If $j$ is large enough that $F^j(x)\in \interior(B)$ and $\ell > j$, then $$e^{-\ell A}\tilde{\psi}( F^\ell(x)) = e^{-jA}\left( e^{(j-\ell)A} \tilde{\psi} \circ F^{(\ell-j)}|_B\right) (F^{j}(x))= e^{-jA} \left(T^{(\ell-j)}(\tilde{\psi})\right) (F^{j}(x)) = e^{-jA} \tilde{\psi} (F^j(x)),$$ so there is a well-defined map $\psi\colon \R^n\to \C^m$ given by
		\begin{equation}\label{eq:psi-const}
		\psi(x)\coloneqq e^{-j A}\tilde{\psi}( F^{j}(x)),
		\end{equation}
		where $j\in \N_{\geq 0}$ is any nonnegative integer sufficiently large that $F^j(x)\in \interior(B)$.
		Since $\tilde{\psi}\circ F|_B = e^A \tilde{\psi}$, it follows from \eqref{eq:psi-const} that $\psi \circ F = e^{A} \psi$.
		If $x \in \R^n$ and $F^j(x)\in \interior(B)$, then $x$ has a neighborhood $U$ with $F^j(U)\subset \interior(B)$ by continuity, so $\psi|_U$ is given by \eqref{eq:psi-const} with $j$ constant on $U$.
		By the chain rule and standard properties of locally $\alpha$-H\"{o}lder functions (see Footnote \ref{foot:holder-function-properties}), this shows that $\psi \in \CLKA(\R^n,\C^m)$.
		From \eqref{eq:psi-const} we see that $\psi$ and hence also $\varphi\coloneqq \psi - P$ are uniquely determined by $\psi|_B = P|_B + \varphi|_B = P|_B + \tilde{\varphi}$, which is in turn uniquely determined by $\tilde{\varphi}$, and since $\tilde{\varphi}$ is unique it follows that $\varphi$ and $\psi$ are also unique.
		If $e^{A}\in \GL(m,\R)$ and $P\in \CLKA(\R^n,\R^m)$ are real, then the complex conjugate $\bar{\psi} = P + \bar{\varphi}$ also satisfies $\bar{\psi} \circ F = e^{A} \bar{\psi}$ with $\bar{\varphi}\in \F$, so uniqueness implies that $\bar{\psi} = \psi$ and hence $\psi, \varphi\colon \R^n\to \R^m\subset \C^m$ are real.
		
		\textit{Convergence to the global linearizing factor.}
        We now complete the proof of the lemma by proving the sole remaining claim that $e^{-jA}P\circ F^j \to \psi$  in the topology of $\CKA$-uniform convergence on compact subsets of $\R^n$.	
        To do this, we first  inspect the finite truncations of the infinite series in \eqref{eq:tilde-phi-neumann}. 
        We see that, since $$\sum_{n=0}^j (T|_\F)^n \cdot \left(T(P|_B) -P|_B\right) = \sum_{n=0}^j T^{n+1}(P|_B) - T^n(P|_B) = T^{j+1}(P|_B) - P|_B$$ for each $j \in \N_{\geq 1}$, taking the limit $j\to\infty$ shows that the series in \eqref{eq:tilde-phi-neumann} is equal to $-P|_B + \lim_{j\to\infty} T^j(P|_B)$, with convergence in the Banach space $C^{k,\alpha}(B,\C^m)$.
		In other words,
		\begin{equation}\label{eq:tilde-psi-converge}
		\tilde{\psi} = \lim_{j\to\infty} e^{-j A} P \circ F^j|_B
		\end{equation}
		with convergence in $C^{k,\alpha}(B,\C^m)$.

       	Next, let $K\subset \R^n$ be the closure of any precompact open set (so that the Banach space $\CKA(K,\C^m)$ is defined as in the text containing \eqref{eq:cka-norm-def}). 
		Since $0$ is globally asymptotically stable and since $B$ contains a neighborhood of $0$, there exists $j_0 > 0$ such that $F^j(K)\subset B$ for all $j \geq j_0$.
		We justify the following computation below:
		\begin{align*}
		\lim_{j\to\infty}e^{-j A}P\circ F^j|_K &= \lim_{j\to \infty}e^{-j_0 A} \left(e^{-j A} P\circ F^j|_B \right)\circ F^{j_0}|_K\\
		&= e^{-j_0 A} \left(\lim_{j\to \infty} e^{-j A} P\circ F^j|_B \right)\circ F^{j_0}|_K\\
		&= e^{-j_0 A} \psi|_B \circ F^{j_0}|_K\\
		&= \psi|_K,
		\end{align*}
		with convergence in $C^{k,\alpha}(K,\C^m)$. 
		Since we are considering convergence in the space $C^{k,\alpha}(K,\C^m)$---rather than mere pointwise convergence---it is not obvious that we can move the limit inside the parentheses to obtain the second equality.
		The reason this is valid is that composition maps $\CKA(K,\C^m)\to \CKA(K,\C^m)$ of the form  $g\mapsto f\circ g \circ h$, where $f\in C^\infty$ and $h\in \CKA(K,\C^m)$, are continuous with respect to the $C^{k,\alpha}$ normed topologies \cite[Prop.~6.1,~Prop.~6.2~(iii)]{de1999regularity}.\footnote{We remark that, for maps between finite-dimensional spaces, there are somewhat weaker assumptions ensuring continuity of such composition maps \cite[Rem.~6.5]{de1999regularity}, but our present situation does not require this.}
        This completes the proof for the case $k < +\infty$.

		\textit{Consideration of the case $k = +\infty$.}
		For the case $k = +\infty$, repeating the proof above for any $1\leq k' < +\infty$ such that $\nu(e^{A},\D_0 F) < k'$ yields  unique $C^{k'}$ functions $\varphi\colon \R^n\to \C^m$ and $\psi\coloneqq P + \varphi$  such that $\D_0^i \varphi = 0$ for all $0\leq i \leq k'-1$.
		By the uniqueness statement already proved for the case $k< +\infty$, these functions $\varphi, \psi$ are independent of $k' > \nu(e^{A},\D_0 F)$, and since $k'$ is arbitrary it follows that $\varphi,\psi \in C^\infty(\R^n,\C^m)$.
		Finally, for the closure $K$ of any precompact open set, we have already shown that $e^{-j A}\circ P \circ F^j|_K \to \psi|_K$ in the Banach space $C^{k'}(K,\C^m)$ for every $k'\in \N_{\geq 1}$, as desired.
		This completes the proof.
	\end{proof}	
	
	Using Lemmas \ref{lem:existence-approx-conj} and \ref{lem-make-approx-exact}, we now complete the proof of Theorem \ref{th:main-thm} by proving the existence portion of its statement.
	
	\begin{proof}[Proof of the existence portion of Theorem \ref{th:main-thm}]
    As in the proof of the uniqueness portion of Theorem \ref{th:main-thm} at the end of \S \ref{sec:main-proof-uniq}, we may assume that $Q  = \R^n$ and $x_0 = 0$.
    We first consider the case that $\Grp = \Z$, and define the time-$1$ map $F\coloneqq \Phi^1$.
    
	First suppose that $k < +\infty$.
	Lemma \ref{lem:existence-approx-conj} implies that there exists a polynomial $P$ such that $\D_0 P = B$ and $P\circ F = e^{A} P + R$, where $R\in \CLKA(\R^n,\C^m)$ satisfies $\D_0^i R = 0$ for all integers $0\leq i < k + \alpha$.\footnote{Actually, Lemma \ref{lem:existence-approx-conj} implies that we can find $P$ such that $\D^i_0 R = 0$ for all integers $0\leq i \leq k$, with the only difference arising when $\alpha = 0$. However, we do not need this in the following.}
	Furthermore, $P$ and $R$ are real if $e^{A}$ and $B$ are real. 
	Lemma \ref{lem-make-approx-exact} then implies that there exists $\varphi\in \CLKA(\R^n,\C^m)$ such that $\psi = P + \varphi\in \CLKA(\R^n,\C^m)$ satisfies $\D_0 \psi = \D_0 P =  B$, $\psi \circ F = e^{A} \psi$, $e^{-jA}\tilde{P} \circ \Phi^j \to \psi$ $C^{k,\alpha}$-uniformly on compact subsets of $\R^n$ for any ``approximate linearizing factor''  $\tilde{P}$ satisfying the hypotheses of Theorem \ref{th:main-thm} (such as $P$), and that $\psi$ is real if $e^{A}$ and $B$ are real.
	This completes the proof for the case $k < +\infty$.
	
	Now suppose that $k = +\infty$.
	Repeating the proof above for finite $k' >  \nu(e^{A},\D_0)$ yields $\psi \in C^{k'}(\R^n,\C^m)$ satisfying $\D_0 \psi = B$ and $\psi \circ F = e^{A} \psi$.
	The uniqueness statement of Theorem \ref{th:main-thm} proved in \S \ref{sec:main-proof-uniq} implies that $\psi$ is independent of $k' > \nu(e^{A},\D_0 F)$, so since $k'$ is arbitrary it follows that $\psi \in C^\infty$.
	Additionally, we have already argued above that $e^{-jA}\tilde{P} \circ \Phi^j \to \psi$ $C^{k'}$-uniformly on compact subsets of $\R^n$ for any $k'\in \N_{\geq 1}$ and any ``approximate linearizing factor''   $\tilde{P}$ satisfying the hypotheses of Theorem \ref{th:main-thm}.
	This completes the proof for the case that $\Grp = \Z$.
	
	It remains only to consider the case that $\Grp = \R$, i.e., the case that $\Phi$ is a flow.
    By the proof of the case $\Grp = \Z$, there exists $\tilde{\psi}\in \CLKA(\R^n,\C^m)$ satisfying $\D_0 \tilde \psi = B$ and $\tilde \psi \circ \Phi^j = e^{jA} \tilde \psi$ for all $j \in \Z$.
    By adapting a technique of Sternberg \cite[Lem~4]{sternberg1957local}, from $\tilde{\psi}$ we will construct a map $\psi\in \CLKA(\R^n,\C^m)$ satisfying $\D_0 \psi = B$ and $\psi \circ \Phi^t = e^{tA} \psi$ for all $t\in \R$.
    In fact, we will show that 
    \begin{equation}
    \psi\coloneqq \int_0^1 e^{-sA}\tilde{\psi}\circ \Phi^s\, ds
    \end{equation}
    has these properties.
    By Leibniz's rule for differentiating under the integral sign and standard properties of locally $\alpha$-H\"{o}lder functions (see Footnote \ref{foot:holder-function-properties}), $\psi \in \CLKA(\R^n,\C^m)$, and using the assumption \eqref{eq:main-th-1}  we have that$$\D_0\psi = \int_0^1 e^{-sA}B \D_0 \Phi^s \, ds = \int_0^1 B\,ds = B.$$
    To prove that $\psi \circ \Phi^t = e^{tA} \psi $ for all $t\in \R$, we compute
    \begin{align*}
    \psi \circ \Phi^t &= \int_0^1 e^{-sA} \tilde{\psi} \circ \Phi^{s+t}\,ds
    = \int_t^{1+t} e^{(t-s)A} \tilde{\psi} \circ \Phi^s \, ds\\
    &= e^{tA} \int_t^1 e^{-sA} \tilde{\psi} \circ \Phi^s \, ds + e^{tA} \int_1^{1+t} e^{-sA} \tilde{\psi} \circ \Phi^s \, ds\\
    &= e^{tA} \int_t^1 e^{-sA} \tilde{\psi} \circ \Phi^s \, ds + e^{tA} \int_1^{1+t} e^{-sA}  \left(e^{A}\tilde{\psi} \circ \Phi^{-1} \right) \circ \Phi^s \, ds\\
    &= e^{tA} \int_t^1 e^{-sA}  \tilde{\psi} \circ \Phi^s \, ds + e^{tA} \int_0^{t} e^{-sA} \tilde{\psi} \circ \Phi^s \, ds\\
    &= e^{tA} \psi
    \end{align*}
    as desired.
    We remark that, since $\psi$ satisfies $\psi \circ \Phi^1 = e^{A} \psi$ and $\D_0 \psi = B$, the uniqueness result for the case $\Grp = \Z$ actually implies the (non-obvious) fact that $\psi = \tilde{\psi}$.
    
    Suppose now that $k < +\infty$.
    Letting $K\subset \R^n$ be the closure of any precompact open set (so that the Banach space $\CKA(K,\C^m)$ is defined as in the text containing \eqref{eq:cka-norm-def}) which is also positively invariant, the map $G\colon [0,1] \times C^{k,\alpha}(K,\C^m)\to C^{k,\alpha}(K,\C^m)$ given by $G(r, f) \coloneqq e^{-r A}f \circ \Phi^r|_K$ is continuous \cite[Thm~6.10]{de1999regularity} and satisfies $G(r,\psi|_K)= \psi|_K$ for all $r\in [0,1]$.
    Thus, compactness of $[0,1]$ implies that, for every neighborhood $V\subset C^{k,\alpha}(K,\C^m)$ of $\psi|_K$, there is a smaller neighborhood $U \subset V$ of $\psi|_K$ such that $G([0,1]\times U) \subset V$, i.e., $e^{-rA} \varphi \circ \Phi^r|_K \subset V$ for every $\varphi \in U$ and $r\in [0,1]$.
    Fix any such neighborhoods $V,U$ and fix  any ``approximate linearizing factor'' $P\in \CLKA(\R^n,\C^m)$ satisfying the hypotheses  of Theorem \ref{th:main-thm}. 
    By the proof for the case $\Grp = \Z$ there exists $N \in \N_{\geq 0}$ such that, for all $j > N$, $e^{-jA}P\circ \Phi^j|_K \in U$.
    By the definition of $U$ it follows that $e^{-tA}P\circ \Phi^t|_K \subset V$ for all $t > N+1$. 
    Since the neighborhood $V \ni \psi|_K$ was arbitrary, this implies that
    \begin{equation}\label{eq:psi-continuous-time-cka-converge}
    \psi|_K = \lim_{t\to \infty} e^{-tA}P\circ \Phi^t|_K
    \end{equation}
    with convergence in the Banach space $\CKA(K,\C^m)$.
    If instead $k = +\infty$, the same argument shows that \eqref{eq:psi-continuous-time-cka-converge} converges in the Banach space $C^{k'}(K,\C^m)$ for every $k'\in \N_{\geq 1}$.
    Since every compact subset of $\R^n$ is contained in the closure $K$ of some positively invariant precompact open set (e.g., a sublevel set of a smooth Lyapunov function \cite{wilson1969smooth,fathi2019smoothing}), this proves that $e^{-tA}P\circ \Phi^t \to \psi$ in the topology of $C^{k,\alpha}$-uniform convergence on compact sets if $k < +\infty$, and in the topology of $C^{k'}$-uniform convergence on compact sets for any $k'\in \N_{\geq 1}$ if $k=+\infty$. 
    This completes the proof of Theorem \ref{th:main-thm}.    
	\end{proof}	
	
	\subsection{Proof of Theorem \ref{th:main-thm-per}}
	In this section we prove Theorem \ref{th:main-thm-per}, which we repeat here for convenience.
	This proof invokes Theorem \ref{th:main-thm} and is much shorter because of this.
	
    \ThmMainPer*	
	
	\begin{proof}
		Let $\Ws_{x_0}$ be the global strong stable manifold (isochron) through $x_0$ \cite[Sec.~2.1]{kvalheim2018global}.
		Since $\Ws_{x_0}$ is the stable manifold for the fixed point $x_0$ of the $\CLKA$ diffeomorphism $\Phi^\tau$, it follows that $\Ws_{x_0}$ is a $\CLKA$ submanifold \cite[pp.~2, 27; Thm~6.1]{ruelle1989elements} which is properly embedded (rather than merely immersed) in $Q$ because $\Gamma$ is stable \cite[pp.~4208--4209]{kvalheim2018global}.
		
		After identifying $E^s_{x_0}$ with $\R^n$, the uniqueness portion of Theorem \ref{th:main-thm} applied to $\psi|_{\Ws_{x_0}}$ implies that $\psi|_{\Ws_{x_0}}$ is unique for any $\psi$ satisfying the uniqueness hypotheses, and furthermore $\psi|_{\Ws_{x_0}}$ is real if $A$ and $B$ are real.
	    Since $\psi$ is uniquely determined by $\psi|_{\Ws_{x_0}}$ and \eqref{eq:main-th-per-2} (which is true because $Q = \bigcup_{t\in \R}\Phi^t(\Ws_{x_0})$), this implies that $\psi$ is unique and that $\psi$ is real if $A$ and $B$ are real.
	    This completes the proof of the uniqueness statement of Theorem \ref{th:main-thm-per}.
		
		Under the existence hypotheses, the existence portion of Theorem \ref{th:main-thm} similarly implies that there exists a unique $\varphi\in \CLKA(\Ws_{x_0},\C^m)$ 
        satisfying $\D_{x_0}\varphi = B$ and 
        \begin{equation}\label{eq:main-th-per-proof-1}
        \forall j \in \Z\colon \varphi \circ \Phi^{j\tau}|_{\Ws_{x_0}} = e^{j\tau A} \varphi.
        \end{equation}		
		The unique extension of $\varphi$ to a $\CLKA$ function $\psi \colon Q\to \C^m$ satisfying \eqref{eq:main-th-per-2} is given by
        \begin{equation}\label{eq:main-th-per-proof-2}
        \forall t\in \R\colon \psi|_{\Ws_{\Phi^{-t}(x_0)}}\coloneqq e^{-tA} \varphi \circ \Phi^{t}|_{\Ws_{\Phi^{-t}(x_0)}}.
        \end{equation}
        $\psi$ is well-defined because $\Phi^\tau(\Ws_{x_0})= \Ws_{x_0}$ and $e^{\tau A} \varphi \circ \Phi^{-\tau}|_{\Ws_{x_0}} = \varphi$ by \eqref{eq:main-th-per-proof-1}.
        That $\psi\in \CLKA$ follows from considering locally-defined $\CLKA$ ``time-to-impact $\Ws_{x_0}$'' functions as in the proof of Proposition \ref{prop:floq-norm-form}.
        This completes the proof.		 
	\end{proof}

	\section*{Acknowledgments}
	The majority of this work was performed while Kvalheim was a postoctoral researcher at the University of Michigan.
	Both authors were supported by ARO award W911NF-14-1-0573 to Revzen and by the ARO under the Multidisciplinary University Research Initiatives (MURI) Program, award W911NF-17-1-0306 to Revzen.
	Kvalheim was also supported by the ARO under the SLICE
	MURI Program, award W911NF-18-1-0327.
	We thank George Haller, David Hong, Igor Mezi\'{c}, Jeff Moehlis, Corbinian Schlosser, and Dan Wilson for helpful discussions and comments.
	We owe special gratitude to Alexandre Mauroy and Ryan Mohr for their careful reading of the manuscript and valuable feedback; in particular, one of Mauroy's comments led to a sharpening of the uniqueness statements of Theorems \ref{th:main-thm} and \ref{th:main-thm-per} and Propositions \ref{prop:koopman-cka-fix} and \ref{prop:koopman-cka-per}, and Mohr found an error in Definition \ref{def:nonres}.

	\bibliographystyle{amsalpha}
	\bibliography{ref}
	
\end{document}